\documentclass[12pt]{amsart}
\usepackage[margin=0.6in]{geometry}

\usepackage[english]{babel}
\usepackage[alphabetic]{amsrefs}
\usepackage{amsmath,amssymb,amsfonts,amsthm,enumerate}
\usepackage{url}
\usepackage{graphicx,epstopdf,color}

\definecolor{Cyan}{gray}{0.6}

\numberwithin{equation}{section}

\newtheorem{theorem}{Theorem}[section]
\newtheorem{lemma}[theorem]{Lemma}

\newtheorem{corollary}[theorem]{Corollary}

\newcommand{\R}{\mathbb{R}}
\newcommand{\Z}{\mathbb{Z}}
\newcommand{\N}{\mathbb{N}}

\renewcommand{\varpi}{\omega}

\renewcommand{\le}{\leqslant}

\renewcommand{\ge}{\geqslant}

\renewcommand{\epsilon}{\varepsilon}
\newcommand{\eps}{\varepsilon}

\title[Nonlocal minimal surfaces]{Nonlocal minimal surfaces: \\
interior regularity, \\
quantitative estimates \\ and boundary stickiness}

\author{Serena Dipierro and Enrico Valdinoci}

\address{{\em Serena Dipierro:} School of Mathematics and Statistics,
University of Melbourne, 813 Swanston St, Parkville VIC 3010, Australia,
School of Mathematics and Statistics,
University of Western Australia,
35 Stirling Highway,
Crawley, Perth
WA 6009, Australia}
\email{sdipierro@unimelb.edu.au}

\address{{\em Enrico Valdinoci:} School of Mathematics and Statistics,
University of Melbourne, 813 Swanston St, Parkville VIC 3010, Australia,
School of Mathematics and Statistics,
University of Western Australia,
35 Stirling Highway, Crawley, Perth WA 6009, Australia,
Weierstra{\ss}-Institut f\"ur Angewandte
Analysis und Stochastik, Hausvogteiplatz 5/7, 10117 Berlin, Germany, and
Dipartimento di Matematica, Universit\`a degli studi di Milano,
Via Saldini 50, 20133 Milan, Italy}

\begin{document}

\begin{abstract}
We consider surfaces which minimize a nonlocal perimeter
functional and we discuss their interior regularity and rigidity
properties, in a quantitative and qualitative way, and their
(perhaps rather surprising) boundary behavior. We present
at least a sketch of the proofs of these results, 
in a way that aims to be as elementary and self contained as possible,
referring to the papers~\cite{CRS, SV, CV, bego, fig, stick, joaq}
for full details.
\end{abstract}

\maketitle

\bigskip

\hfill{\it ...taurino quantum possent circumdare tergo...}\bigskip

\section{Introduction}

The study of surfaces which minimize the perimeter is
a classical topic in analysis and geometry and probably one of the oldest
problems in the mathematical literature: according to the first book
of Virgil's Aeneid, Dido, the 
legendary queen of Carthage, needed to study problems
of geometric minimization 
in order to found her reign in 814 B.C. (in spite of the great mathematical
talent of Dido and of her vivid geometric intuition, Aeneas broke his
betrothal with her after a short time
to sail the Mediterranean towards the coasts of Italy,
but this is another story).\medskip

The first problem in the study of these surfaces of minimal perimeter
(minimal surfaces, for short) lies in proving that minimizers
do exist. Indeed ``nice'' sets, for which one can compute
the perimeter using an intuitive notion known from elementary school,
turn out to be a ``non compact'' family (roughly speaking,
for instance,
an ``ugly'' set can be approximated by a sequence of
``nice'' sets, thus the limit point of the sequence may end
up outside the family). To overcome this difficulty,
a classical tool of the Calculus of Variations is to look for
minimizers in a wider family of candidates: this larger family 
has to be chosen to satisfy
the desired compactness property to ensure
the existence of a minimum, and then the regularity of
the minimal candidate can be (hopefully) proved a posteriori.\medskip

To this end, one needs to set up an appropriate notion of
perimeter for the sets in the enlarged family of candidates,
since no intuitive notion of perimeter is available, in principle,
in this generality. The classical approach of Caccioppoli (see e.g.~\cite{zbMATH02581579})
to this
question lies in the observation that if~$\Omega$ and~$E$ are\footnote{{F}rom now on, we reserve the name of~$\Omega$ to an open set,
possibly with smooth boundary, which can be seen as the ``ambient space''
for our problem.} smooth sets
and~$\nu$ is the external normal of~$E$, then, for any vector
field~$T\in C^1_0(\Omega,\R^n)$
with~$|T(x)|\le1$ for any~$x\in\Omega$, we have that
$$ T\cdot\nu \le |T|\,|\nu|\le 1.$$
Consequently, the perimeter of~$E$ in $\Omega$, i.e.
the measure of the boundary of~$E$ inside~$\Omega$
(that is, the $(n-1)$-dimensional Hausdorff measure of~$\partial E$
in~$\Omega$), satisfies the inequality
\begin{equation}\label{JK:912345}
{\rm Per}\,(E,\Omega)={\mathcal{H}}^{n-1}\big( (\partial E)\cap\Omega\big)
\ge \int_{\partial E} T\cdot\nu \,d{\mathcal{H}}^{n-1}
= \int_E {\rm div}\, T(x)\,dx,\end{equation}
for every vector field~$T\in C^1_0(\Omega,\R^n)$
with~$\|T\|_{L^\infty(\R^n,\R^n)}\le1$,
where the Divergence Theorem has been used in the last identity.

Viceversa, if~$E$ is a smooth set, its normal vector can be extended
near~$\partial E$, and then to the whole of~$\R^n$, to a vector field~$\nu_*\in
C^1(\R^n,\R^n)$, with~$|\nu_*(x)|\le1$ for any~$x\in\R^n$.
Then, if~$\eta\in C^\infty_0(\Omega,[0,1])$, 
with~$\eta=1$ in an~$\epsilon$-neighborhood of~$\Omega$,
one can take~$T:=\eta\nu_*$
and find that~$T\in C^1_0(\Omega,\R^n)$, $|T(x)|\le1$ for any~$x\in\R^n$
and
\begin{eqnarray*} && \int_E {\rm div}\,T(x)\,dx =
\int_{\partial E} T\cdot\nu \,d{\mathcal{H}}^{n-1}
\\ &&\qquad= \int_{\partial E} \eta\nu_*\cdot\nu \,d{\mathcal{H}}^{n-1}
= \int_{\partial E} \eta\,d{\mathcal{H}}^{n-1}
\\ && \qquad\ge {\mathcal{H}}^{n-1}\big( (\partial E)\cap\Omega\big)-O(\epsilon)=
{\rm Per}\,(E,\Omega)-O(\epsilon)
.\end{eqnarray*}
By taking~$\epsilon$ as small as we wish and recalling~\eqref{JK:912345},
we obtain that
\begin{equation} \label{ji:09876rg}
{\rm Per}\,(E,\Omega)= \sup_{{T\in C^1_0(\Omega,\R^n)}\atop{
\|T\|_{L^\infty(\R^n,\R^n)}\le1}}
\int_E {\rm div}\,T(x)\,dx.\end{equation}
While~\eqref{ji:09876rg} was obtained for smooth sets~$E$,
the classical approach for minimal surfaces is in fact to take~\eqref{ji:09876rg}
as {\em definition} of perimeter of a (not necessarily smooth) set~$E$ in~$\Omega$.
The class of sets obtained in this way indeed has the necessary compactness
property (and the associated 
functional has the desired lower semicontinuity properties)
to give the existence of minimizers: that is, 
one finds (at least) one set~$E\subseteq\R^n$ 
satisfying
\begin{equation}
{\rm Per}\,(E,\Omega)\le
{\rm Per}\,(F,\Omega)
\end{equation}
for any~$F\subseteq\R^n$ such that~$F$
coincides with~$E$ in a neighborhood of~$\Omega^c$.

The boundary of this minimal set~$E$ satisfies, a posteriori,
a bunch of additional regularity properties --
just to recall the principal ones:
\begin{eqnarray}
\label{n=7}
&& {\mbox{If~$n\le7$ then $(\partial E)\cap\Omega$ is smooth;}} \\
\label{n=8}
&& {\mbox{If~$n\ge8$ then $\big((\partial E)\cap\Omega\big)\setminus\Sigma$
is smooth,}}\\ &&\nonumber{\mbox{being~$\Sigma$ a closed set
of Hausdorff dimension at most~$n-8$;}} \\
&& \label{cont}
{\mbox{The statement in~\eqref{n=8} is sharp,
since there exist}}\\&&\nonumber{\mbox{examples in which the singular set~$\Sigma$}}\\ &&\nonumber{\mbox{has
Hausdorff dimension~$n-8$.}}
\end{eqnarray}
We refer to~\cite{giusti} for complete statements and proofs
(in particular, the claim in~\eqref{n=7} here
corresponds to Theorem~10.11 in~\cite{giusti},
the claim in~\eqref{n=8} here to Theorem~11.8 there,
and the claim in~\eqref{cont} here to
Theorem~16.4 there).\medskip

A natural problem that is closely related to these regularity results
is the complete description
of classical minimal surfaces in the whole of the space which
are also graphs in some direction (the so-called minimal graphs).
These questions, that go under the name of Bernstein's problem,
have, in the classical case, the following positive answer:
\begin{eqnarray}
\label{BER}
&& {\mbox{If~$n\le8$ and $E$ is a minimal graph, then~$E$ is a halfspace;}}
\\&& \label{cont:BE}
{\mbox{The statement in~\eqref{BER} is sharp,
since there exist}}\\&&\nonumber{\mbox{examples of minimal graphs in dimension~$9$
and higher}}\\&&\nonumber{\mbox{that are not halfspaces.}}\end{eqnarray}
We refer to
Theorems~17.8 and~17.10 in~\cite{giusti} for further details
on the claims in~\eqref{BER}
and~\eqref{cont:BE}, respectively.\medskip

It is also worth recalling that
\begin{equation}\label{EQ:0mc}
{\mbox{surfaces minimizing perimeters have zero mean curvature,}}
\end{equation}
see e.g. Chapter~10 in~\cite{giusti}.\medskip

Recently, and especially in light of the seminal paper~\cite{CRS},
some attention has been devoted to a variation
of the classical notion of perimeters which takes into
account also long-range interactions between sets, as well as the corresponding
minimization problem.
This type of {\it nonlocal minimal surfaces}
arises naturally, for instance, in the study of fractals~\cite{MR1111612},
cellular automata~\cite{MR2487027, MR2564467} and phase transitions~\cite{MR2948285}
(see also~\cite{2015arXiv150408292B} for a detailed introduction to
the topic).\medskip

A simple idea for defining a notion of nonlocal perimeter may
be described as follows. First of all, such nonlocal perimeter should
compute the interaction~$I$ of all the points of~$E$ against
all the points of the complement of~$E$, which we denote by~$E^c$.

On the other hand, if we want to localize these contributions
inside the domain~$\Omega$, it is convenient to split~$E$ into~$E\cap\Omega$
and~$E\setminus\Omega$, as well as the set~$E^c$ into~$E^c\cap\Omega$
and~$E^c\setminus\Omega$, and so consider the four possibilities
of interaction between~$E$ and~$E^c$ given by
\begin{equation}\label{KLA:94r5tgH}
\begin{split}
& I(E\cap\Omega, E^c\cap\Omega), \qquad I(E\cap\Omega,E^c\setminus\Omega),
\\ &I(E\setminus\Omega,E^c\cap\Omega), \qquad{\mbox{ and }}
\qquad I(E\setminus\Omega ,E^c\setminus\Omega).
\end{split}\end{equation}
Among these interactions, we observe that the latter
one only depends on the configuration of the set outside~$\Omega$,
and so
$$ I(E\setminus\Omega ,E^c\setminus\Omega) =
I(F\setminus\Omega ,F^c\setminus\Omega)$$
for any~$F\subseteq\R^n$ such that~$F\setminus \Omega=E\setminus \Omega$.
Therefore, in a minimization process with fixed data outside~$\Omega$,
the term~$I(E\setminus\Omega ,E^c\setminus\Omega)$ does not change
the minimizers. It is therefore natural to omit this term
in the energy functional (and, as a matter of fact, omitting
this term may turn out to be important from the
mathematical point of view, since this term may provide an infinite
contribution to the energy). For this reason, the nonlocal perimeter
considered in~\cite{CRS} is given by the sum of the first
three terms in~\eqref{KLA:94r5tgH}, namely one defines
$$ {\rm Per}_s\,(E,\Omega)
:=
I(E\cap\Omega, E^c\cap\Omega)+I(E\cap\Omega,E^c\setminus\Omega)+
I(E\setminus\Omega,E^c\cap\Omega).$$
As for the interaction~$I(\cdot,\cdot)$, of course some freedom
is possible, and basically any interaction for
which~$ {\rm Per}_s\,(E,\Omega)$ is finite, say, for smooth sets~$E$
makes perfect sense. A natural choice performed
in~\cite{CRS} is to take the interaction as a weighted Lebesgue
measure, where the weight is translation invariant,
isotropic and homogeneous: more precisely, for any disjoint
sets~$S_1$ and~$S_2$, one defines\footnote{We remark that~\eqref{jAH:AI}
gives that the ``natural scaling'' of the interaction~$I$
is ``meters to the power $n-2s$''
(where~$2n$ comes from~$dx\,dy$ and~$-n-2s$
comes from~$|x-y|^{-n-2s}$). When~$s=1/2$,
this scaling boils down to the one of the classical perimeter.}
\begin{equation}\label{jAH:AI}
I(S_1,S_2):=\iint_{S_1\times S_2} \frac{dx\,dy}{|x-y|^{n+2s}},\end{equation}
with~$s\in\left(0,\frac12\right)$. With this choice of the fractional
parameter~$s$, one sees that
\begin{eqnarray*}
&&[\chi_E]_{W^{\sigma,p}(\R^n)} := 
\iint_{\R^n\times \R^n} \frac{|\chi_E(x)-\chi_E(y)|^p
}{|x-y|^{n+p\sigma}} \,dx\,dy \\
&&\qquad= 2\iint_{E\times E^c} \frac{dx\,dy
}{|x-y|^{n+p\sigma}} =2I(E,E^c)=
2\,{\rm Per}_s\,(E,\R^n)
\end{eqnarray*}
as long as~$p\sigma=2s$, that is the fractional perimeter of a set
coincides (up to normalization constants) to a fractional
Sobolev norm of the corresponding characteristic function
(see e.g.~\cite{guida} for a simple introduction to
fractional Sobolev spaces).\medskip

Moreover, for any fixed~$y\in\R^n$,
\begin{equation}\label{1234}
{\rm div}_x\,\frac{x-y}{|x-y|^{n+2s}} = -\frac{2s}{|x-y|^{n+2s} }
.\end{equation}
Also, for any fixed~$x\in\R^n$,
$$ {\rm div}_y\,
\frac{\nu(x)}{|x-y|^{n+2s-2}}
= (n+2s-2)
\frac{\nu(x)\cdot(x-y)}{|x-y|^{n+2s}}.$$
Accordingly, by the Divergence Theorem\footnote{We will often
use the Divergence Theorem here in a rather formal way, by
neglecting the possible singularity of the kernel -- for a rigorous
formulation one has to check that the possible singular contributions
average out, at least for smooth sets.}
\begin{equation}\label{8uj56789:A}\begin{split}
&{\rm Per}_s\,(E,\R^n)\\=\;&-\frac{1}{2s} \int_{E^c} \,dy \left[
\int_{E} {\rm div}_x\,\frac{x-y}{|x-y|^{n+2s}} \,dx \right] \\
=\;& -\frac{1}{2s}\int_{E^c} \,dy \left[
\int_{\partial E} \frac{\nu(x)\cdot(x-y)}{|x-y|^{n+2s}} \,d{\mathcal{H}^{n-1}}(x) \right]
\\ =\;& -\frac{1}{2s\,(n+2s-2)}
\int_{\partial E} \,d{\mathcal{H}^{n-1}}(x)\,\left[
\int_{E^c} {\rm div}_y\,
\frac{\nu(x)}{|x-y|^{n+2s-2}}\,dy\right] \\
=\;& \frac{1}{2s\,(n+2s-2)}
\iint_{(\partial E)\times(\partial E)}
\frac{\nu(x)\cdot\nu(y)}{|x-y|^{n+2s-2}}
\,d{\mathcal{H}^{n-1}}(x)\,d{\mathcal{H}^{n-1}}(y).
\end{split}\end{equation}
That is, 
$$ {\rm Per}_s\,(E,\R^n)
= \frac{1}{4s\,(n+2s-2)}
\iint_{(\partial E)\times(\partial E)}
\frac{2-|\nu(x) -\nu(y)|^2 }{|x-y|^{n+2s-2}}
\,d{\mathcal{H}^{n-1}}(x)\,d{\mathcal{H}^{n-1}}(y),$$
which suggests that
the fractional perimeter is
a weighted measure of the variation of the normal vector around the boundary of a set.
As a matter of fact, as $s\nearrow1/2$, the $s$-perimeter recovers
the classical perimeter from many point of views (a sketchy discussion about
this will be given in Appendix \ref{APP12}).

Also, in Appendix \ref{9ojknAAsaw}, we briefly discuss the second variation of the $s$-perimeter
on surfaces of vanishing nonlocal mean curvature and we show that
graphs with vanishing nonlocal mean curvature cannot have horizontal normals.
\medskip

Let us now recall (among the others)
an elementary, but useful, application of this notion of fractional
perimeter in the framework of digital image reconstruction.
Suppose that we have a black and white digitalized image, say a bitmap,
in which each pixel is either colored in black or in white.
We call~$E$ the ``black set'' and we are interested in measuring
its perimeter (the reason for that may be, for instance,
that noises or impurities could be distinguished by having
``more perimeter'' than the ``real'' picture, since they may
present irregular or fractal boundaries). In doing that, we need
to be able to compute such perimeter with a very good precision.
Of course, numerical errors could affect the computation,
since the digital process replaced the real picture by
a pixel representation of it, but we would like that our computation
becomes more and more reliable if the resolution of the image is
sufficiently high, i.e. if
the size of the pixels is 
sufficiently small.

Unfortunately, we see that, in general, an accurate
computation of the perimeter is not possible, not even for simple sets,
since the numerical error produced by the pixel may not become negligible,
even when the pixels are small. To observe this phenomenon
(see e.g.~\cite{joaq}) we can consider a grid of square
pixels of small side~$\epsilon$
and a black square~$E$ of side~$1$, with the black square rotated by~$45$ degrees
with respect to the orientation of the pixels. Now, the digitalization
of the square will produce a numerical error, since, say, the
pixels that intersect the square are taken as black, and so
each side of the
square is replaced by a ``sawtooth''
curve
(see Figure~\ref{dente di sega}). 

\begin{figure}[ht]
\includegraphics[width=9.4cm]{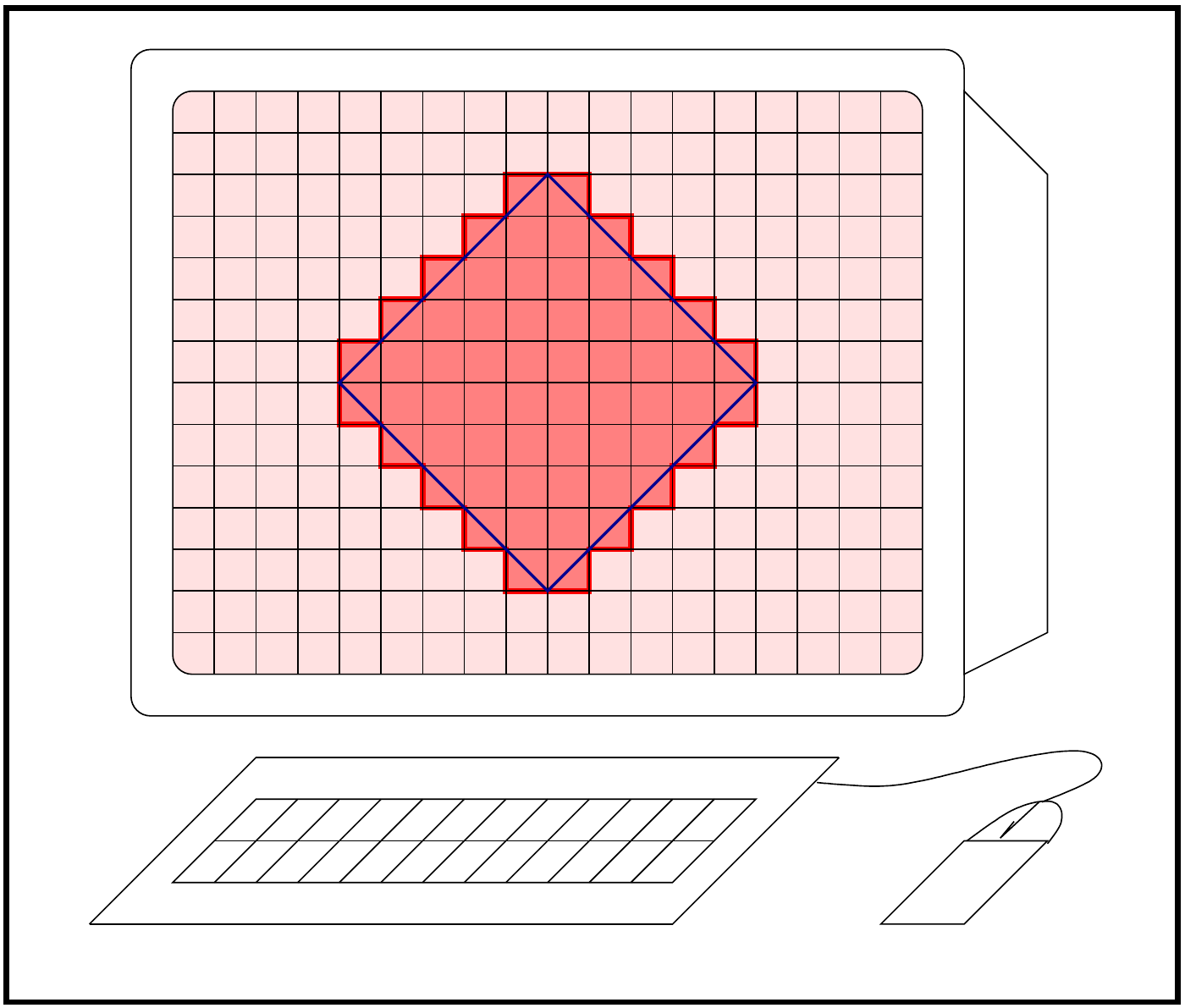}
\caption{\sl Numerical error in computing the perimeter.}
\label{dente di sega}
\end{figure}

Notice that the length of each of
these sawtooth curves is~$\sqrt{2}$ (independently on how small
each teeth is, that is independently on the size of~$\epsilon$).
As a consequence, the perimeter of the digitalized image is~$4\sqrt{2}$,
instead of~$4$, which was the original perimeter of the square.\medskip

This shows the rather unpleasant fact that the perimeter may
be poorly approximated numerically, even in case of high precision
digitalization processes. It is a rather remarkable fact
that fractional perimeters do not present the same inconvenience
and indeed the numerical error in computing the fractional perimeter 
becomes small when the pixels are small enough. Indeed,
the number of pixels which intersect the sides of the original square
is~$O(\epsilon^{-1})$ (recall that the side of the square is~$1$
and the side of each pixel is of size~$\epsilon$).
Also, the $s$-perimeter of each pixel is~$O(\epsilon^{2-2s})$
(since this is the natural scale factor of the interaction
in~\eqref{jAH:AI}, with~$n=2$). Then, the numerical error
in the fractional perimeter comes from the contributions
of all these pixels\footnote{More precisely, when the computer changes
the ``real'' square with the discretized one and produces a staircase border,
the only interactions changed are the ones affecting the union
of the triangles (that are ``half pixels'') that are added to the square
in this procedure. In the ``real'' picture, these triangles interact with the square, while in the digitalized picture they interact with
the exterior. To compute the error obtained 
one takes the signed superposition of these effects, therefore,
to estimate the error 
in absolute value, one can just sum up these contributions,
which in turn are bounded by the sum of the interactions of each triangle
with its complement, see Figure~\ref{INTER}.}
and it is therefore~$O(\epsilon^{-1})\cdot
O(\epsilon^{2-2s}) = O(\epsilon^{1-2s})$, which tends to zero
for small~$\epsilon$, thus showing that the nonlocal perimeters
are more efficient than classical ones in this type of digitalization
process.\medskip

\begin{figure}[ht]
\includegraphics[width=9.4cm]{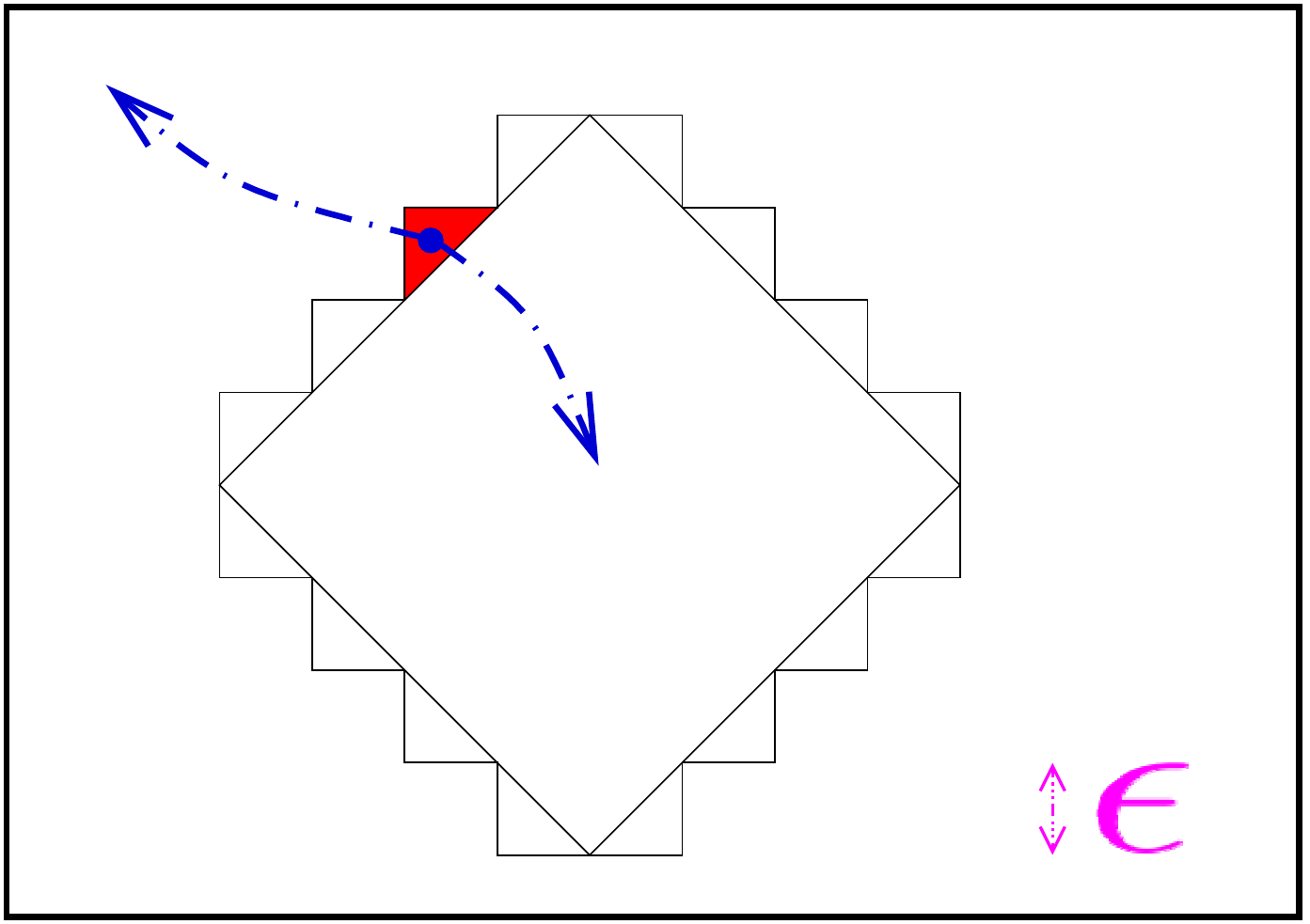}
\caption{\sl Pixel interactions and numerical errors.}
\label{INTER}
\end{figure}

Thus, given its mathematical interest and its importance in concrete
applications, it is desirable to reach
a better understanding of the surfaces which minimize
the $s$-perimeter (that one can call $s$-minimal surfaces).
To start with, let us remark that an analogue
of~\eqref{EQ:0mc} holds true, in the sense that 
$s$-minimal surfaces have vanishing $s$-mean curvature
in a sense that we now briefly describe. Given a set~$E$ with smooth
boundary and~$p\in\partial E$, we define
\begin{equation}\label{HS-DEF}
H^s_E(p):=\int_{\R^n} \frac{\chi_{E^c}(x)-\chi_E(x)}{|x-p|^{n+2s}}\,dx.\end{equation}
The expression\footnote{The definition
of fractional mean curvature in~\eqref{HS-DEF} may look
a bit awkward at a first glance. To make it appear more friendly,
we point out that
the classical mean curvature of~$\partial E$
at a point~$x\in\partial E$, up to normalizing
constants, can be computed via the
average procedure
$$ \lim_{r\searrow 0} \frac{1}{r^{n+1}}
\int_{B_r(x)} \chi_{E^c}(y)-\chi_E(y)\,dy.$$
To see this, up to rigid motions,
one can assume that~$x$ is the origin
and~$E$, in a small neighborhood of the
origin, is the subgraph of a function~$
u:\R^{n-1}\to\R$ with~$u(0)=0$ and~$\nabla
u(0)=0$. Then, we have that
\begin{eqnarray*}&&
\lim_{r\searrow 0} \frac{1}{r^{n+1}}
\int_{B_r} \chi_{E^c}(y)-\chi_E(y)\,dy
= \lim_{r\searrow 0} \frac{1}{r^{n+1}}
\int_{{\{|y'|\le r\}}\atop{\{y_n\in[-r,r]\}}}
\chi_{E^c}(y)-\chi_E(y)\,dy\\
&&\qquad\quad= -
\lim_{r\searrow 0} \frac{2}{r^{n+1}}
\int_{\{|y'|\le r\}} u(y')\,dy'
= -
\lim_{r\searrow 0} \frac{2}{r^{n+1}}
\int_{\{|y'|\le r\}} D^2 u(0)y'\cdot y'+
o(|y'|^2)\,dy'
= - c\,\Delta u(0).\end{eqnarray*}
}
in~\eqref{HS-DEF} is intended in the principal value sense, namely
the singularity is taken in an averaged limit, such as
$$ H^s_E(p)=\lim_{\rho\searrow0}\int_{\R^n\setminus B_\rho(p)} 
\frac{\chi_{E^c}(x)-\chi_E(x)}{|x-p|^{n+2s}}\,dx.$$
For simplicity, we omit the principal value from the notation.
It is also useful to recall~\eqref{1234} and to remark that~$H^s_E$
can be computed as a weighted boundary integral of the normal, namely
\begin{equation} 
\begin{split}
H^s_E(p)\,=\;& 
-\frac{1}{2s} \int_{\R^n} \big(\chi_{E^c}(x)-\chi_E(x)\big)
\,{\rm div}\,\frac{x-p}{|x-p|^{n+2s}}\,dx \\ =\;&
-\frac{1}{s} \int_{\partial E}
\frac{\nu(x)\cdot(p-x)}{|p-x|^{n+2s}}\,d{\mathcal{H}}^{n-1}(x) .
\end{split}\end{equation}
This quantity~$H^s_E$ is what we call the nonlocal mean curvature of~$E$
at the point~$p$, and the name is justified by the following
observation:

\begin{lemma}\label{iu678HHJKA}
If~$E$ is a set with smooth boundary that minimizes the $s$-perimeter
in~$\Omega$, then~$H^s_E(p)=0$ for any~$p\in (\partial E)\cap\Omega$.
\end{lemma}

The proof of Lemma~\ref{iu678HHJKA} will be given in Section~\ref{iu678HHJKA:PF}.
We refer to~\cite{CRS} for a version of Lemma~\ref{iu678HHJKA} that
holds true (in the viscosity sense) without assuming that
the set has smooth boundary. See also~\cite{abatangelo} for
further comments on this notion of nonlocal mean curvature.\medskip

Let us now briefly discuss the fractional analogue
of the regularity results in~\eqref{n=7} and~\eqref{n=8}.
At the moment, a complete regularity theory in the fractional case
is still not available. At best, one can obtain regularity results
either in low dimension or when~$s$ is sufficiently close to~$\frac12$
(see~\cite{SV, CV}
and also~\cite{bego} for higher regularity results): 
namely, the analogue of~\eqref{n=7}
is:

\begin{theorem}[Interior regularity results 
for $s$-minimal surfaces - I]\label{REG:A}
Let~$E\subset\R^n$ be a minimizer for the $s$-perimeter in~$\Omega$.
Assume that
\begin{itemize}
\item either~$n=2$,
\item or~$n\le7$ and~$\frac12-s\le\epsilon_*$, for some~$\epsilon_*>0$
sufficiently small.\end{itemize}
Then, $(\partial E)\cap\Omega$ is smooth.
\end{theorem}

Similarly, a fractional analogue of~\eqref{n=8} is known, by now,
only when~$s$ is sufficiently close to~$\frac12$:

\begin{theorem}[Interior regularity results
for $s$-minimal surfaces - II]\label{REG:B}
Let~$E\subset\R^n$ be a minimizer for the $s$-perimeter in~$\Omega$.
Assume that~$n\ge8$ and~$\frac12-s\le\epsilon_n$, for some~$\epsilon_n>0$
sufficiently small.
Then, $\big((\partial E)\cap\Omega\big)\setminus\Sigma$
is smooth, being~$\Sigma$ a closed set
of Hausdorff dimension at most~$n-8$.\end{theorem}

Differently from the statement in~\eqref{cont}, it is not known if
Theorems~\ref{REG:A} and~\ref{REG:B} 
are sharp, and in fact
there are no known
examples of $s$-minimal surfaces with singular sets:
and, as a matter of fact, in dimension $n\le6$, these pathological
examples -- if they exist -- cannot be built by symmetric cones
(which means that they either do not exist or are pretty hard
to find!), see~\cite{del pino}.\medskip

In~\cite{joaq}, several quantitative regularity estimates
for local minimizers are given (as a matter of fact,
these estimates are valid in a much more general setting,
but, for simplicity, we focus here on the most basic statements
and proofs). For instance, 
minimizers of the $s$-perimeter have locally finite perimeter
(that is, classical perimeter, not only fractional perimeter),
as stated in the next result:

\begin{theorem}\label{B V T}
Let~$E\subset\R^n$ be a minimizer for the $s$-perimeter in~$B_R$.
Then 
$$ {\rm Per}\,(E,B_{1/2})\le CR^{n-1},$$
for a suitable constant~$C>0$.
\end{theorem}

We stress that Theorem~\ref{B V T} presents several
novelties with respect to the existing literature.
First of all, it provides a scaling invariant regularity
estimate that goes beyond the natural scaling of the $s$-perimeter,
that is valid in any dimension and without any topological
restriction on the $s$-minimal surface (analogous
results for the classical perimeter are not known in this generality).
Also, in spite of the fact that, for the sake
of simplicity, we state and prove
Theorem~\ref{B V T} only in the case of
minimizers of the $s$-perimeter, more general versions
of this result hold true for stable solutions
and for more general interaction kernels (even for kernels
without any regularizing effect). This type
of results also leads to new compactness and existence theorems,
see~\cite{joaq} for full details on this topic.
\medskip

As a matter of fact, we stress that the analogue of
Theorem~\ref{B V T} for stable surfaces which are critical points
of the classical perimeter
is only known, up to now, for
two-dimensional surfaces that are
simply connected and immersed in~$ \R^3$
(hence, this is a case in which the nonlocal theory can go beyond
the local one).\medskip

Now, we briefly 
discuss the fractional analogue
of the Bernstein's problem.
Let us start by pointing out that, by combining~\eqref{n=7}
and~\eqref{BER}, we have an ``abstract'' version of
the Bernstein's problem, which states that
{\it if~$E$ is a minimal graph in~$\R^{n+1}$ and the minimal surfaces in~$\R^n$
are smooth, then~$E$ is a halfspace}.\medskip

Of course, for the way we have written~\eqref{n=7}
and~\eqref{BER}, this abstract statement seems only to say that~$8=7+1$:
nevertheless this abstract version of the Bernstein's problem
is very useful in the classical case, since it admits
a nice fractional counterpart, which is:

\begin{theorem}[Bernstein result for $s$-minimal surfaces - I]\label{GH:A}
If~$E$ is an $s$-minimal graph in~$\R^{n+1}$ and the
$s$-minimal surfaces in~$\R^n$
are smooth, then~$E$ is a halfspace.\end{theorem}

This result was proved in~\cite{fig}. By combining it with
Theorem~\ref{REG:A} (using the notation~$N:=n+1$), we obtain:

\begin{theorem}[Bernstein result for $s$-minimal surfaces - II]\label{BERs:A}
Let~$E\subset\R^{N}$ be an $s$-minimal graph.
Assume that
\begin{itemize}
\item either~$N=3$,
\item or~$N\le8$ and~$\frac12-s\le\epsilon_*$, for some~$\epsilon_*>0$
sufficiently small.\end{itemize}
Then,~$E$ is a halfspace.
\end{theorem}

This is, at the moment, 
the fractional counterpart of~\eqref{BER}
(we stress, however, that any improvement in the
fractional
regularity theory would give for free an improvement in the fractional
Bernstein's problem, via Theorem~\ref{GH:A}).\medskip

We remark again that, differently from the claim in~\eqref{cont:BE},
it is not known if 
the statement in Theorem~\ref{BERs:A}
is sharp,
since there are no known
examples of $s$-minimal graphs other than the hyperplanes.\medskip

It is worth recalling that, by a blow-down procedure,
one can deduce from Theorem~\ref{REG:A} that global
$s$-minimal surfaces are hyperplanes, as stated in the following result:

\begin{theorem}[Flatness of $s$-minimal surfaces]\label{REG:FLA}
Let~$E\subset\R^n$ be a minimizer for the $s$-perimeter in any domain of~$\R^n$.
Assume that
\begin{itemize}
\item either~$n=2$,
\item or~$n\le7$ and~$\frac12-s\le\epsilon_*$, for some~$\epsilon_*>0$
sufficiently small.\end{itemize}
Then, $E$ is a halfspace.
\end{theorem}

Of course, a very interesting spin-off of the regularity
theory in Theorem~\ref{REG:FLA}
lies in finding quantitative flatness estimates: namely,
if we know that a set~$E$ is an~$s$-minimizer in a large domain,
can we say that it is sufficiently close to be a halfspace,
and if so, how close, and in which sense?\medskip

This question has been recently addressed in~\cite{joaq}.
As a matter of fact, the results in~\cite{joaq} are richer
than the ones we present here, and they are valid for a very general
class of interaction kernels and of perimeters of nonlocal type.
Nevertheless we think it is interesting to give a flavor
of them even in their simpler form, to underline their connection
with the regularity theory that we discussed till now.\medskip

In this setting, we present here the following result when $n=2$
(see indeed~\cite{joaq} for more general statements):

\begin{theorem}\label{1123THJ}
Let~$R\ge2$. 
Let~$E\subset\R^2$ be a minimizer for the $s$-perimeter in~$B_R$. Then
there exists a halfplane~$h$ such that
\begin{equation}\label{8yhujsdvdf8778gg}
\big| (E\Delta h)\cap B_1\big|
\le \frac{C}{R^s},\end{equation}
where~$\Delta$ is here the symmetric difference of
the two sets (i.e. $E\Delta h:=
(E\setminus h)\cup(h\setminus E)$) and~$C>0$ is a constant.
\end{theorem}

We stress that Theorem~\ref{1123THJ} may be seen as a quantitative
version of Theorem~\ref{REG:FLA} when~$n=2$: indeed if~$E\subset\R^n$ is a minimizer for the $s$-perimeter in any domain of~$\R^n$
we can send~$R\nearrow+\infty$ in~\eqref{8yhujsdvdf8778gg}
and obtain that~$E$ is a halfplane.\medskip

We observe that, till now, we have presented and discussed a series of
results which are somehow in accordance, as much as possible,
with the classical case. Now we present something with striking
difference from the classical case.
The minimizers of the classical perimeter in a convex domain
reach continuously the boundary data (see e.g. Theorem~15.9
in~\cite{giusti}). Quite surprisingly,
the minimizers of the fractional perimeter have the tendency
to stick at the boundary. This phenomenon has been discovered in~\cite{stick},
where several explicit stickiness examples have been given
(see also \cite{BucLom} for other examples in more
general settings).\medskip

Roughly speaking, the stickiness phenomenon may be described
as follows. We know from Lemma~\ref{iu678HHJKA} that nonlocal minimal
surfaces in a domain~$\Omega$
need to adjust their shape in order to make the nonlocal minimal curvature
vanish inside~$\Omega$. This is a rather strong condition,
since the nonlocal minimal curvature ``sees'' the set all over
the space. As a consequence, in many cases in which
the boundary data are ``not favorable'' for this condition to hold,
the nonlocal minimal surfaces may prefer to modify their shape
by sticking at the boundary, where the condition is not prescribed,
in order to compensate the values of the nonlocal mean curvature
inside~$\Omega$.\medskip

In many cases, for instance, the nonlocal minimal set may
even prefer to ``disappear'', i.e. its contribution inside~$\Omega$
becomes empty and its boundary sticks completely to the boundary
of~$\Omega$. In concrete cases, the fact that the nonlocal minimal
set disappears may be induced by a suitable choice of
the data outside~$\Omega$ or by an appropriate choice of the
fractional parameter.
As a prototype example of these two phenomena, we recall
here the following results given in~\cite{stick}:

\begin{theorem}[Stickiness for small data]\label{ST:001}
For any~$\delta>0$, let
$$K_\delta := \big( B_{1+\delta}\setminus B_1\big)\cap \{x_n<0\}.$$
Let~$E_\delta$ be an $s$-minimal set in~$B_1$ 
among all the sets~$E$ such that~$E\setminus B_1=K_\delta$.

Then, there exists~$\delta_o>0$,
depending on $s$ and $n$, such that
for any~$\delta\in(0,\delta_o]$ we have that
$$ E_\delta=K_\delta.$$
\end{theorem}

\begin{theorem}[Stickiness for small~$s$]\label{ST:002}
As~$s\to0^+$, the $s$-minimal set in~$B_1\subset\R^2$
that agrees with a sector outside~$B_1$
sticks to the sector.

More precisely: let~$E_s$ be
the $s$-minimizer 
among all the sets~$E$ such that
$$ E\setminus B_1=\Sigma:= \{ (x,y)\in\R^2\setminus B_1 {\mbox{ s.t. }}
x>0 {\mbox{ and }} y>0\}.$$
Then, there exists~$s_o>0$ such that for any~$s\in(0,s_o]$
we have that~$E_s=\Sigma$.
\end{theorem}

We stress the sharp difference between the local and the nonlocal
cases exposed in Theorems~\ref{ST:001}
and~\ref{ST:002}: indeed, in the local framework, in both cases
the minimal surface is a segment inside the ball~$B_1$,
while in the nonlocal case it coincides with a piece of the circumference~$\partial B_1$.
\medskip

The stickiness phenomenon of nonlocal minimal surfaces
may also be caused
by a sufficiently high oscillation of the data outside~$\Omega$.
This concept is exposed in the following result:

\begin{theorem}[Stickiness coming from large oscillations of the data]\label{M:SR}
Let~$M>1$ and let~$E_M\subset\R^2$
be $s$-minimal in~$(-1,1)\times\R$
with datum outside~$(-1,1)\times\R$ given by~$J_M:=J^-_M \cup J^+_M$,
where
$$ J^-_M:= (-\infty,-1]\times (-\infty,-M)
\quad{\mbox{ and }}\quad
J^+_M:= [1,+\infty)\times (-\infty,M).$$
Then, if~$M$ is large enough,~$E_M$ sticks at the boundary.
Moreover, the stickiness region gets close to the origin,
up to a power of~$M$.

More precisely:
there exist~$M_o>0$ and~$C_o\ge C_o'>0$, depending on~$s$, such that
if~$M\ge M_o$ then
\begin{equation}\label{JAK:OP}
\begin{split}
& [-1,1)\times [C_o M^{\frac{1+2s}{2+2s}},\,M]\subseteq E_M^c 
\\{\mbox{and }}& \quad(-1,1]\times [-M,\,-C_o M^{\frac{1+2s}{2+2s}}]\subseteq E_M.
\end{split}\end{equation}
\end{theorem}

It is worth to remark that the stickiness phenomenon in Theorem~\ref{M:SR}
becomes ``more and more visible'' as the oscillation of the data increase,
since, referring to~\eqref{JAK:OP}, we have that
$$ \lim_{M\to+\infty} \frac{ M^{\frac{1+2s}{2+2s}} }{M}=0,$$
hence the sticked portion of~$E_M$ on~$\partial\Omega$
becomes, proportionally to~$M$, larger and larger when~$M\to+\infty$.

Also, the exponent~${\frac{1+2s}{2+2s}}$ in~\eqref{JAK:OP} is optimal,
see again~\cite{stick}. The stickiness phenomenon detected in Theorem~\ref{M:SR}
is described in Figure~\ref{M:SR:F}.

\begin{figure}[ht]
\includegraphics[width=6.9cm]{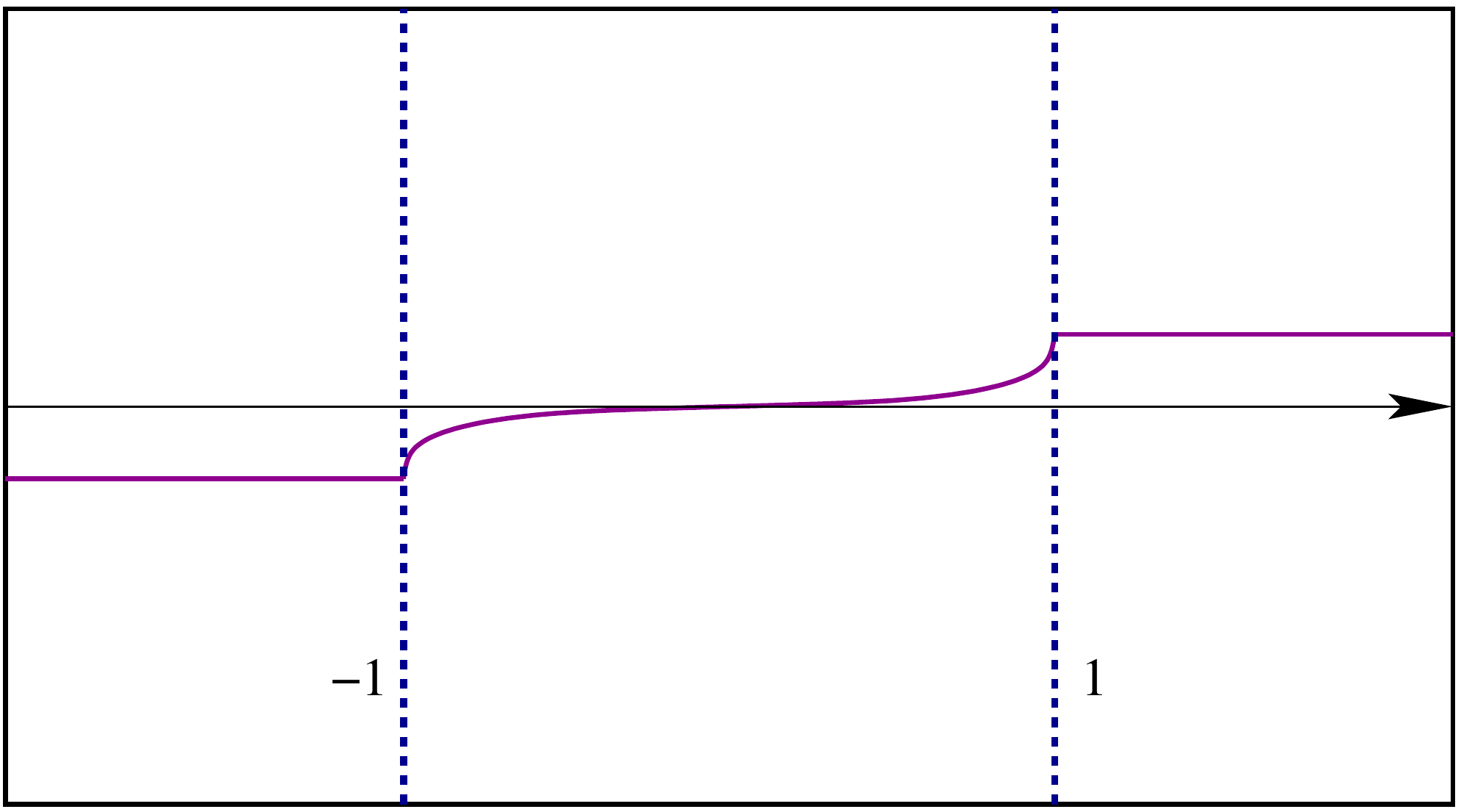}\hspace{0.3cm}
\includegraphics[width=6.9cm]{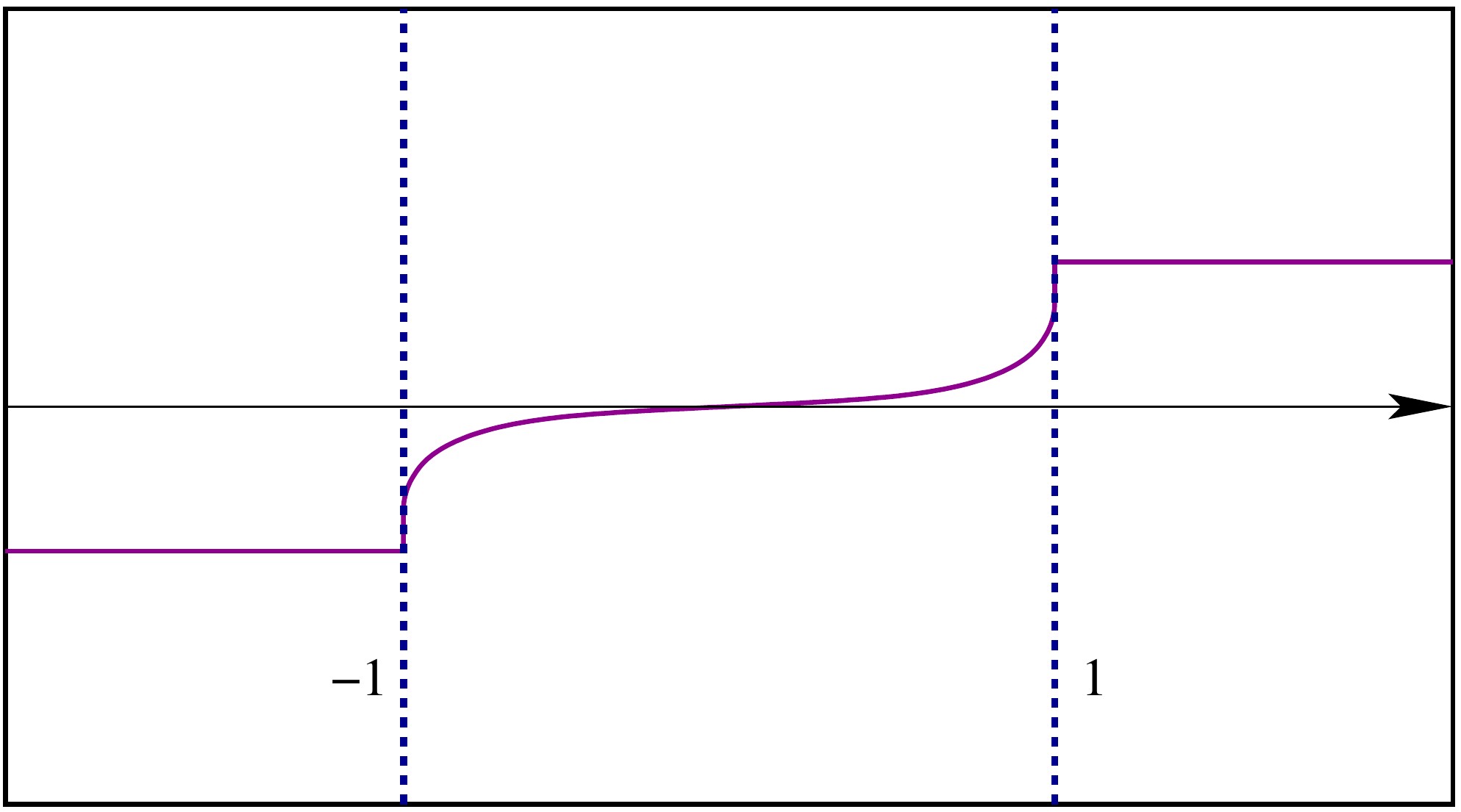}\vspace{0.3cm}
\includegraphics[width=6.9cm]{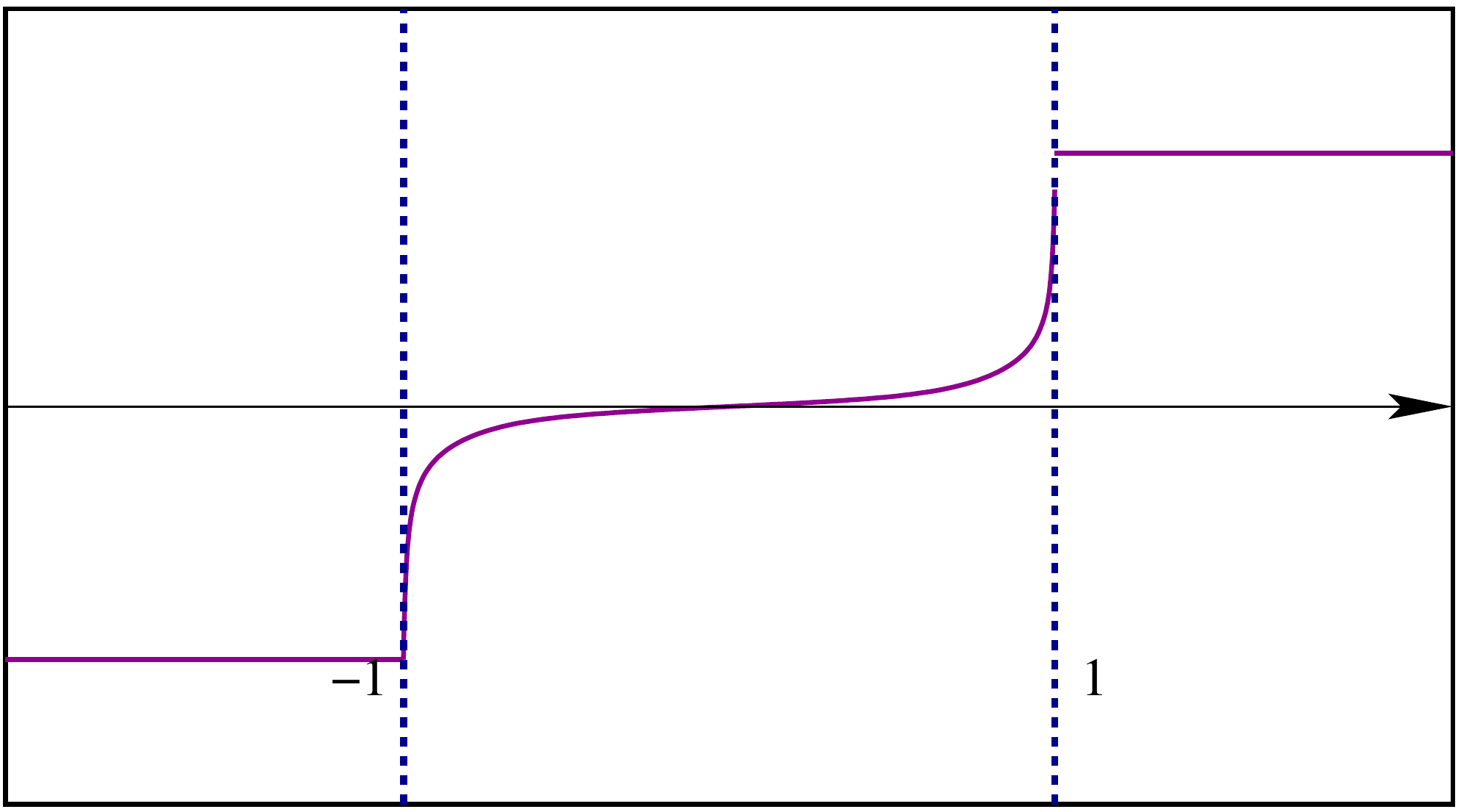}\hspace{0.3cm}
\includegraphics[width=6.9cm]{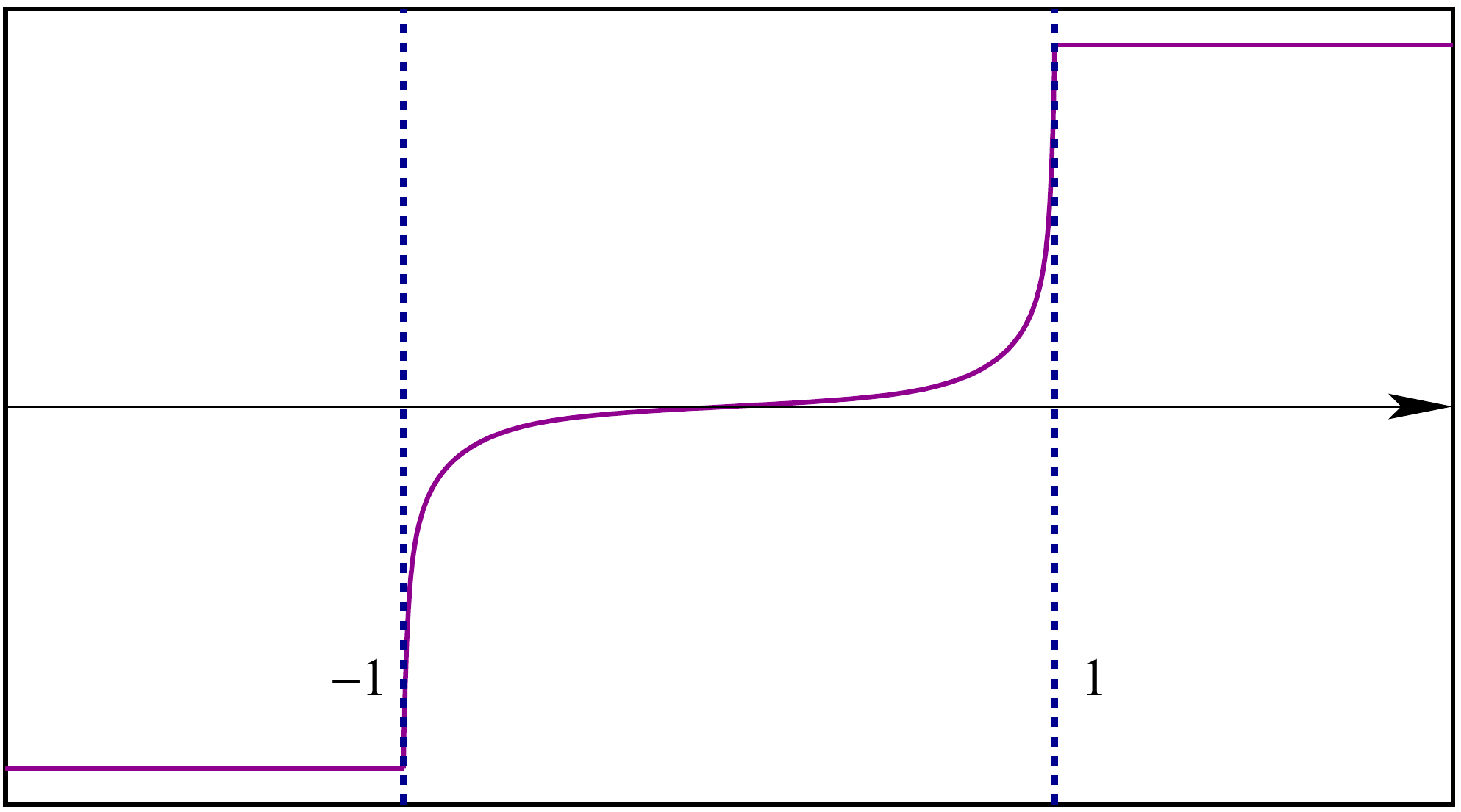}
\caption{\sl Stickiness 
coming from large oscillations of the data
with the oscillation progressively larger.}
\label{M:SR:F}
\end{figure}

\medskip

We believe that the stickiness phenomenon is rather common
among nonlocal minimal surfaces. Indeed, it may occur even
under small modifications of boundary data for which the nonlocal
minimal surfaces cut the boundary in a transversal way.

A typical, and rather striking, example of this
situation happens for perturbation of halfplanes in~$\R^2$.
That is, an arbitrarily small perturbation of the data
corresponding to halfplanes is sufficient for the stickiness
phenomenon to occur. Of course, the smaller the perturbation,
the smaller the stickiness: nevertheless, small perturbations
are enough to cause the fact that the boundary data of nonlocal
minimal surfaces are not attained in a continuous way,
and indeed they may exhibit jumps (notice that this lack
of boundary regularity for $s$-minimal surfaces is
rather surprising, especially after the interior regularity
results discussed in Theorem~\ref{REG:A} and~\ref{REG:B}
and it shows that the boundary behavior of the halfplanes
is rather unstable).

A detailed result goes as follows:

\begin{theorem}[Stickiness arising from perturbation of halfplanes]\label{PER}
There exists~$\delta_0 > 0$
such that for any~$\delta\ in(0, \delta_0]$ 
the following statement holds true.

Let~$\Omega:=(-1,1)\times\R$. Let also
$$ F_- := [-3,-2]\times[0,\delta],\qquad
F_+ := [2,3]\times[0,\delta],\qquad H:=\R\times(-\infty,0).$$
Assume that~$F\subseteq\R^2$, with
$$ F\supseteq H\cup F_-\cup F_+.$$

Let~$E$ be an $s$-minimal set in~$\Omega$
among all the sets which coincide with~$F$ outside~$\Omega$.

Then, 
$$ E\supseteq (-1,1)\times[0,\delta^\gamma],$$
for a suitable~$\gamma>1$.
\end{theorem}

The result of Theorem~\ref{PER} is depicted in Figure~\ref{FPER}.

\begin{figure}[ht]
\includegraphics[width=6.9cm]{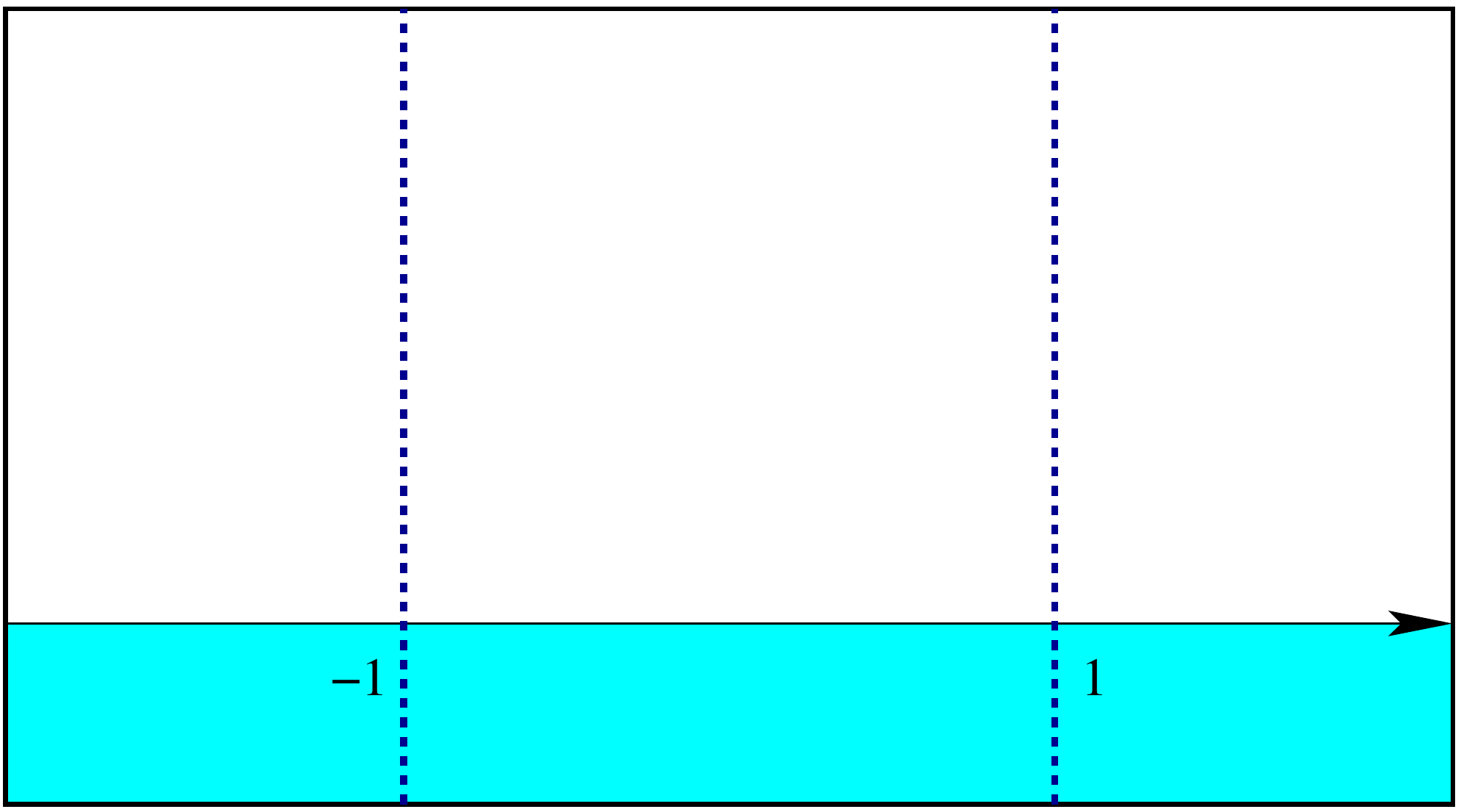}\hspace{0.3cm}
\includegraphics[width=6.9cm]{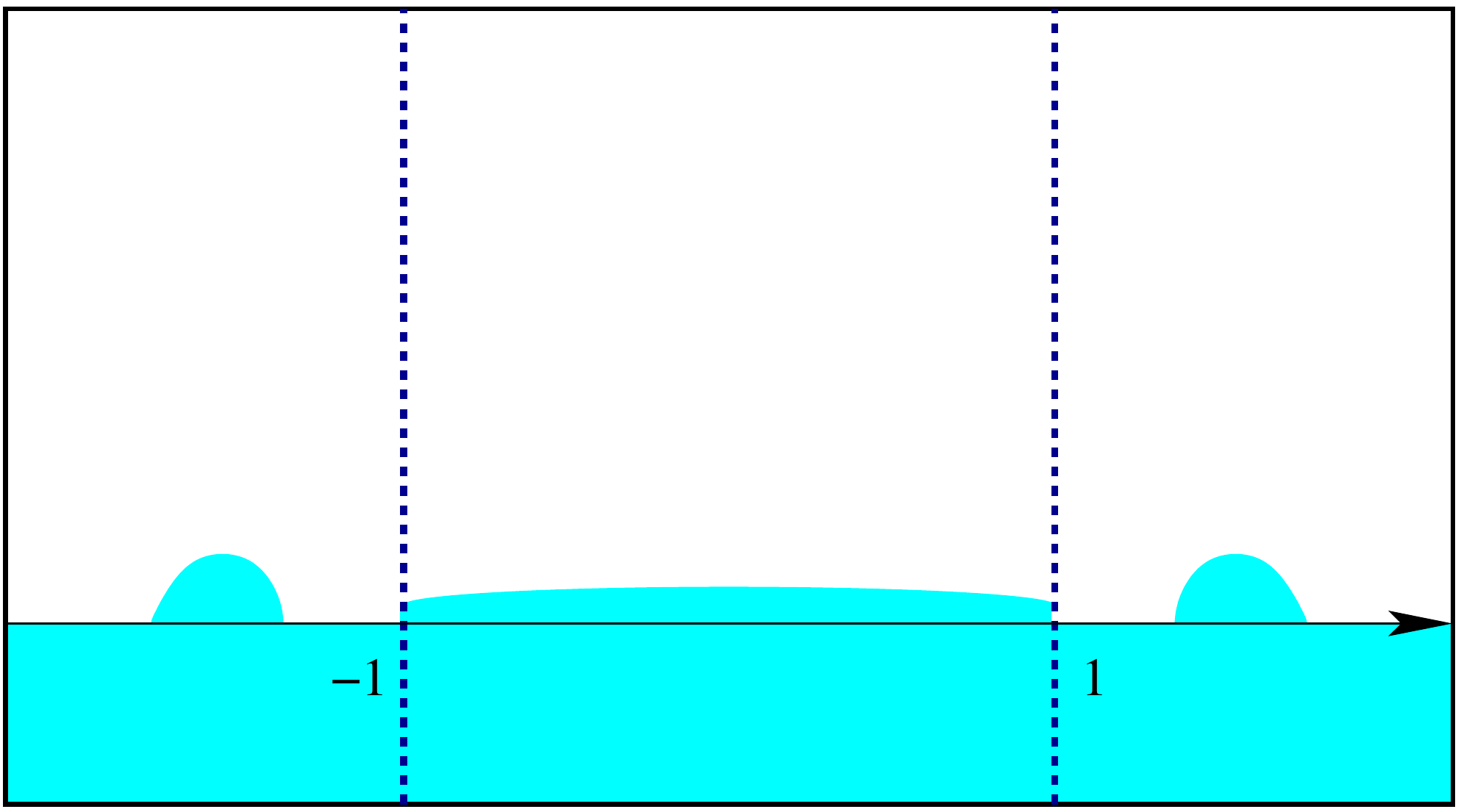}\vspace{0.3cm}
\includegraphics[width=6.9cm]{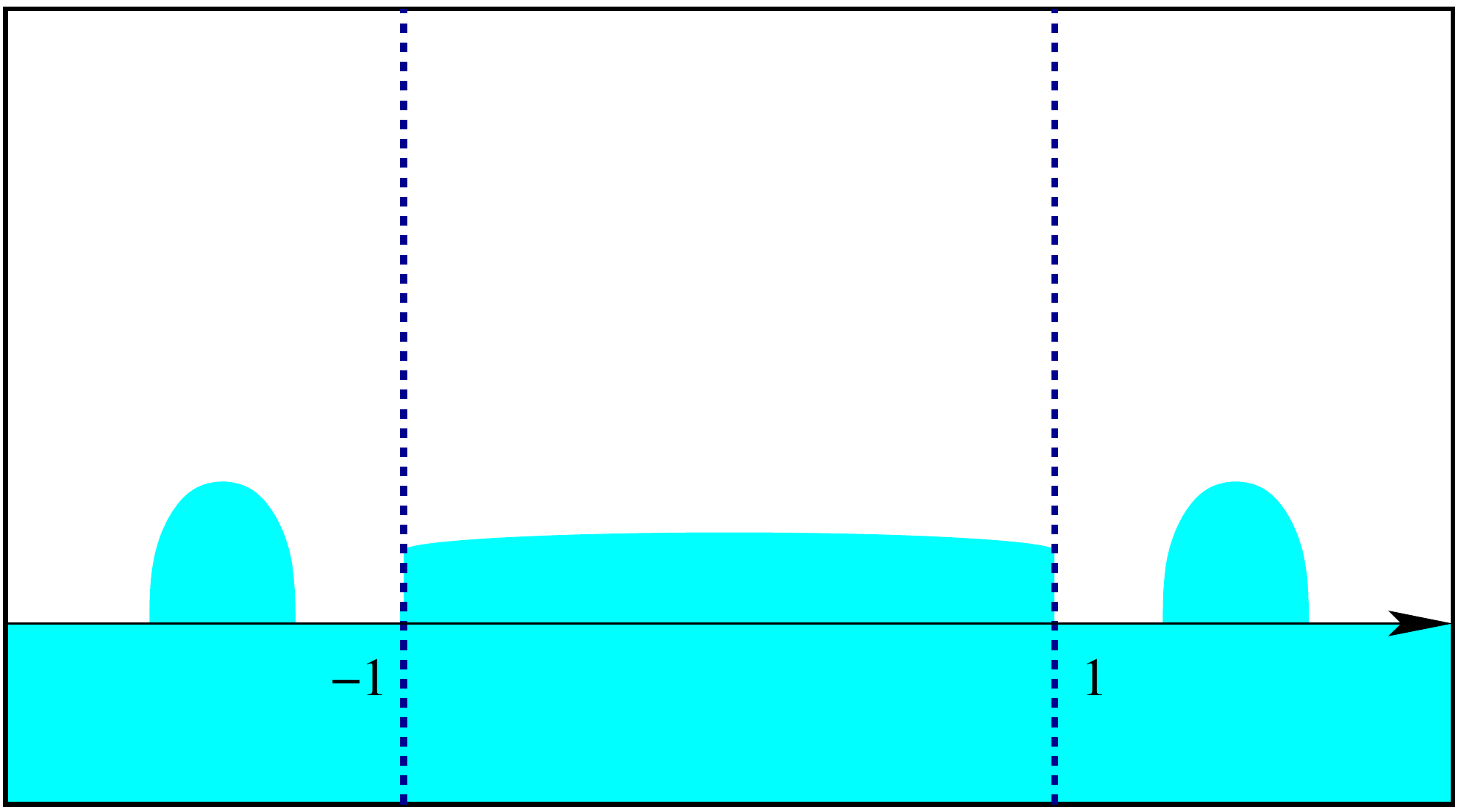}\hspace{0.3cm}
\includegraphics[width=6.9cm]{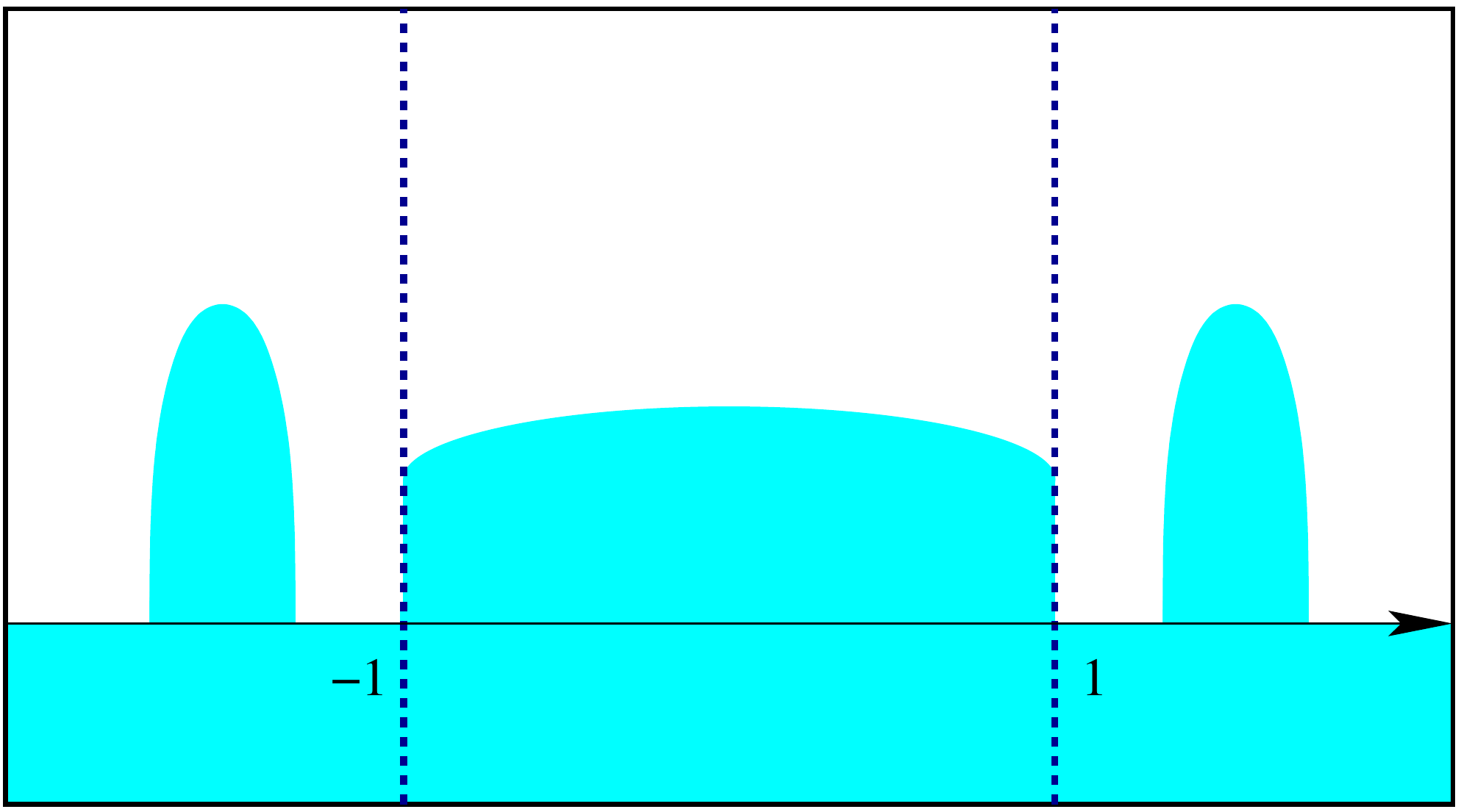}
\caption{\sl Stickiness arising from perturbation of halfplanes,
with the perturbation progressively larger.}
\label{FPER}
\end{figure}

Let us briefly give some further comments on the stickiness
phenomena discussed above. First of all, we would like to
convince the reader (as well as ourselves) that
this type of behaviors indeed occurs in the nonlocal case.

To this end, let us make an investigation to find
how the $s$-minimal set~$E_\alpha$ in~$\Omega:=(-1,1)\times\R\subset\R^2$
with datum
$$ C_\alpha := \{ (x,y)\in\R^2 {\mbox{ s.t. }} y< \alpha|x|\}$$
looks like.

When~$\alpha=0$, then~$E_\alpha=C_\alpha$ is the halfplane,
so the interesting case is when~$\alpha\ne0$; say, up to symmetries,~$\alpha>0$.
Now, we know how an investigation works:
we need to place all the usual suspects in a row
and try to find the culprit.

The line of suspect is on Figure~\ref{SOSPETTI}
(remember that we have to find
the $s$-minimal set among them). Some of the suspects
resemble our prejudices on how the culprit should look like.
For instance, for what we saw on TV,
we have the prejudice that serial killers always wear
black gloves and raincoats. Similarly, for what we learnt
from the hyperplanes, we may have the prejudice that
$s$-minimal surfaces meet the boundary data in a smooth
fashion (this prejudice will turn out to be wrong, as we will see).
In this sense, the usual suspects number~1 and~2
in Figure~\ref{SOSPETTI}
are the ones who look like the serial killers.

\begin{figure}[ht]
{\bf \#1}\includegraphics[width=6.9cm]{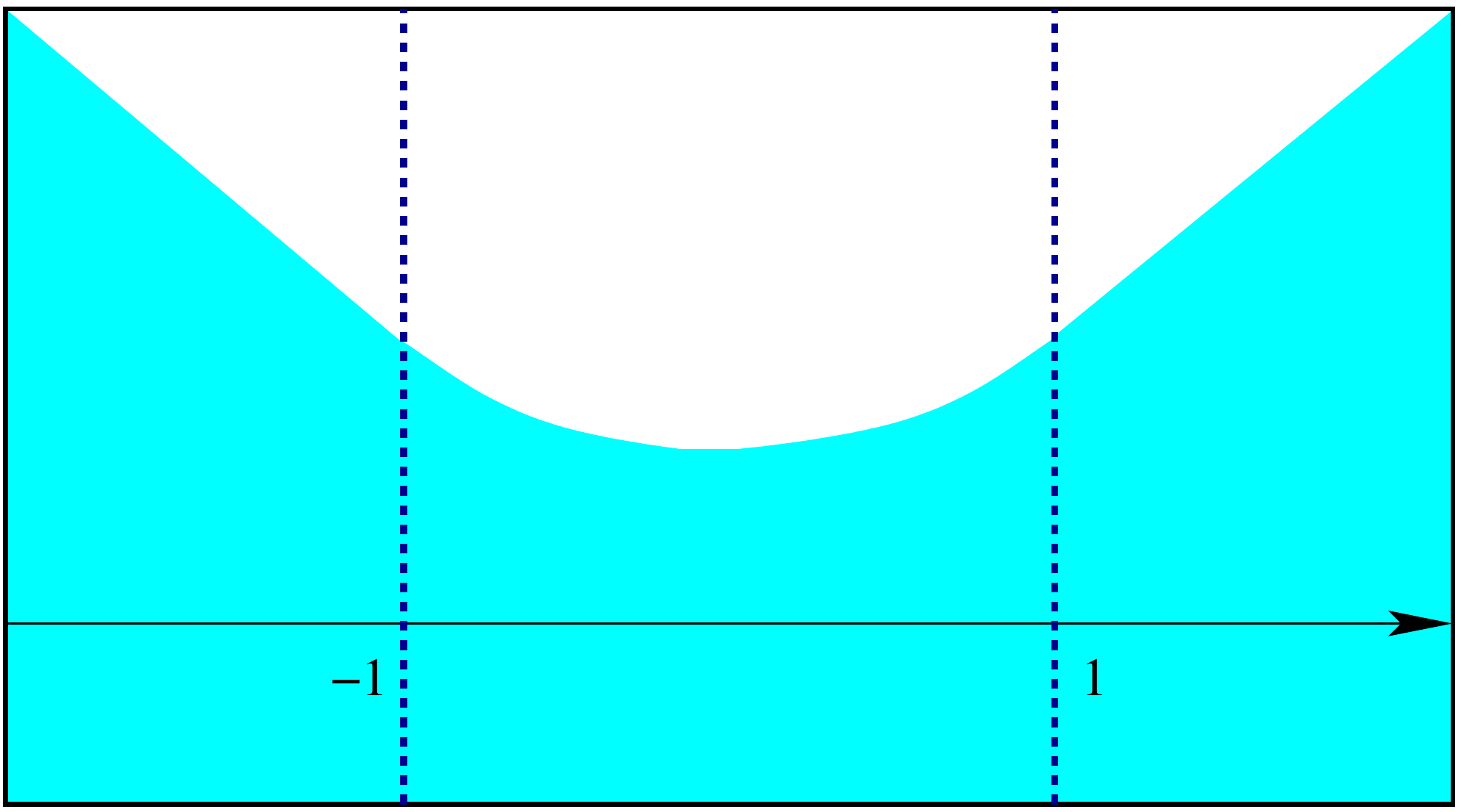}\hspace{0.3cm} 
{\bf \#2}\includegraphics[width=6.9cm]{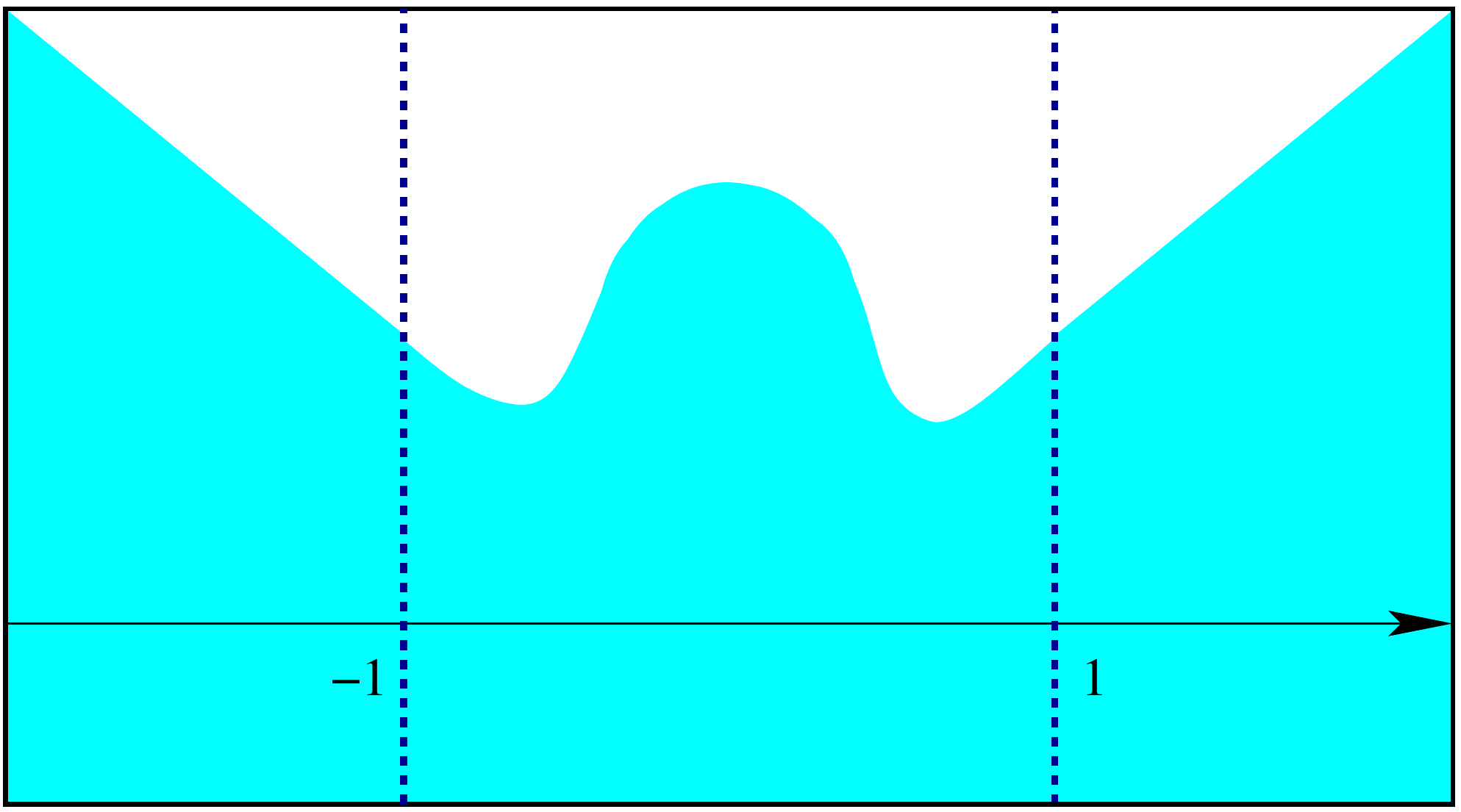}\vspace{0.3cm}
{\bf \#3}\includegraphics[width=6.9cm]{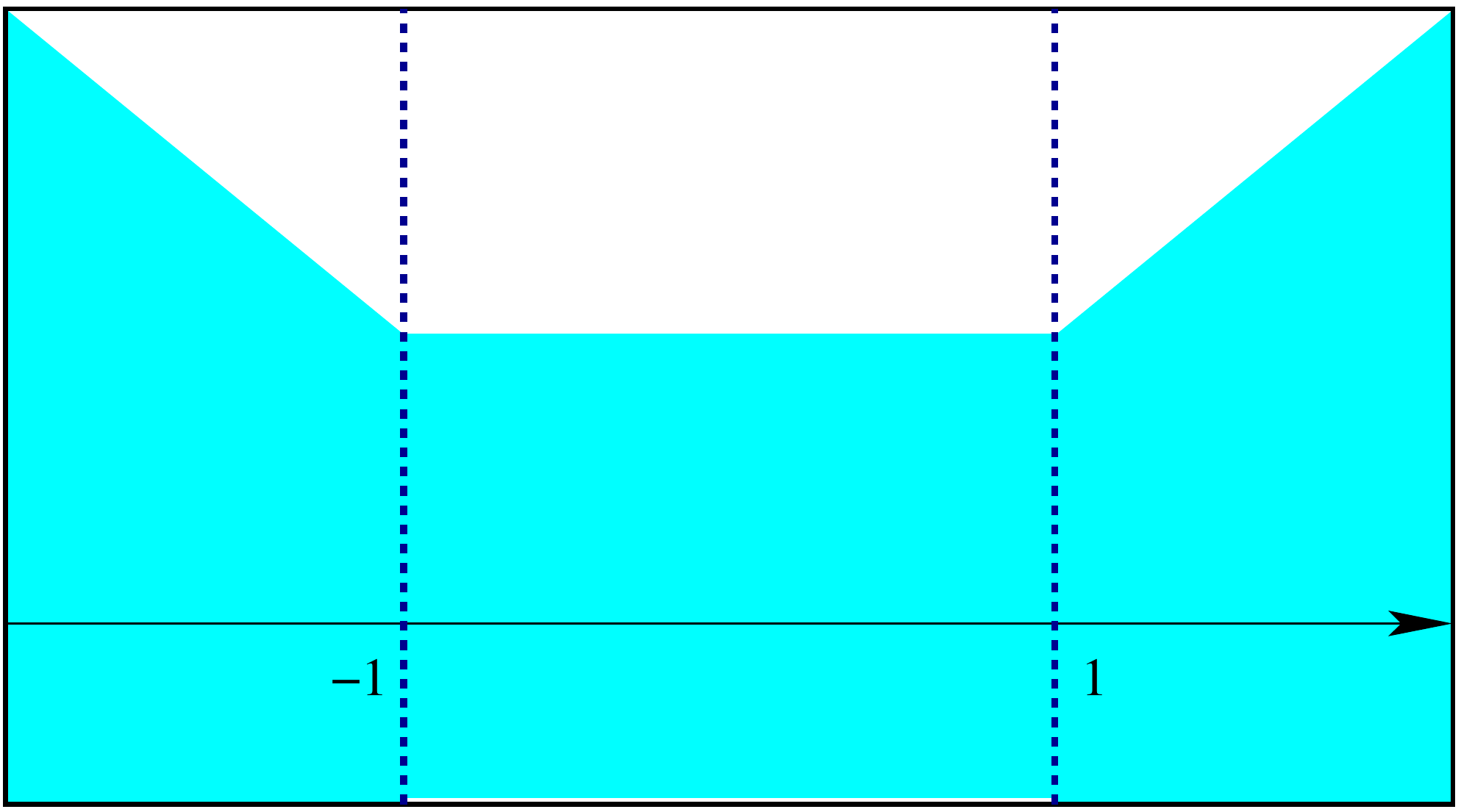}\hspace{0.3cm} 
{\bf \#4}\includegraphics[width=6.9cm]{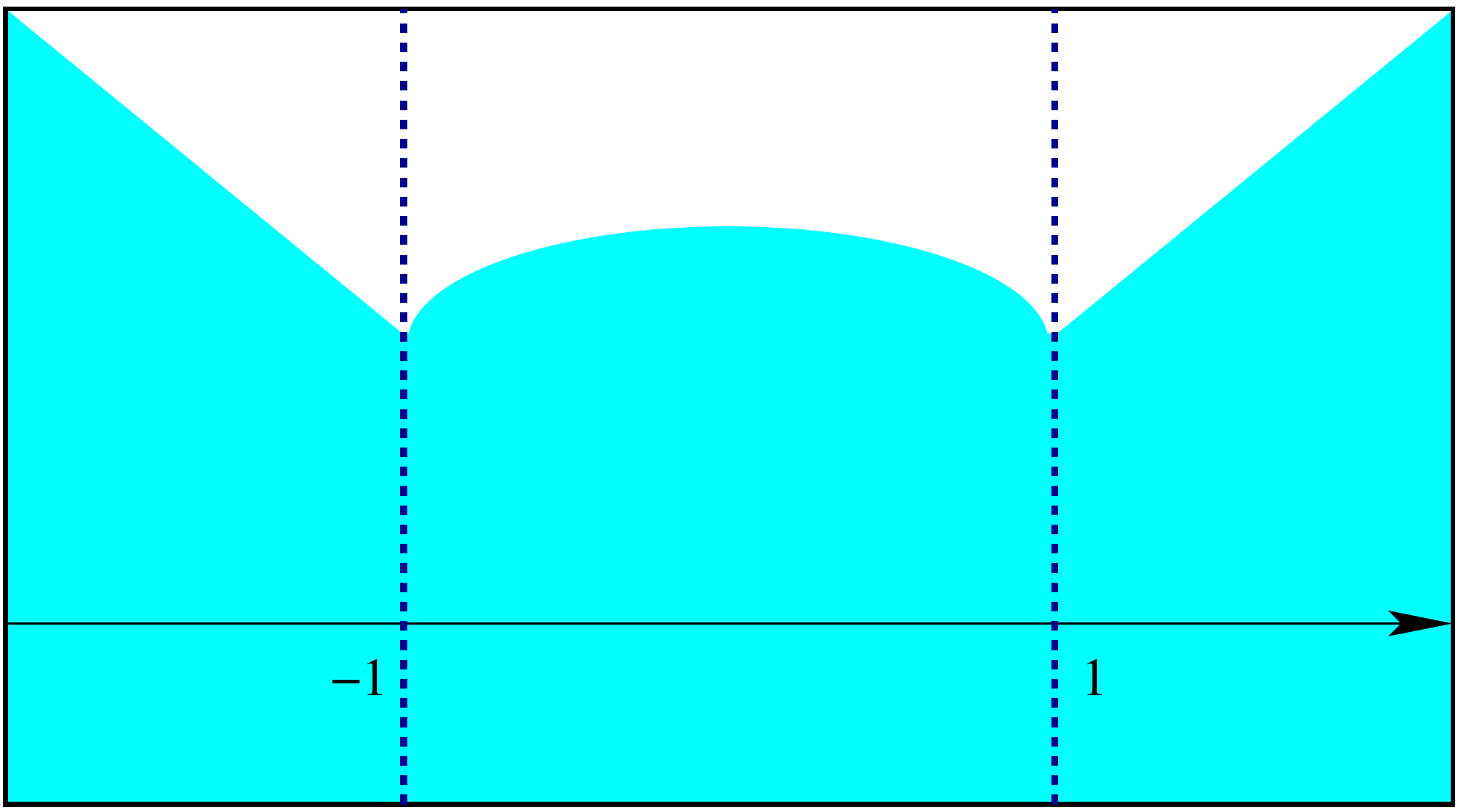}\vspace{0.3cm} 
{\bf \#5}\includegraphics[width=6.9cm]{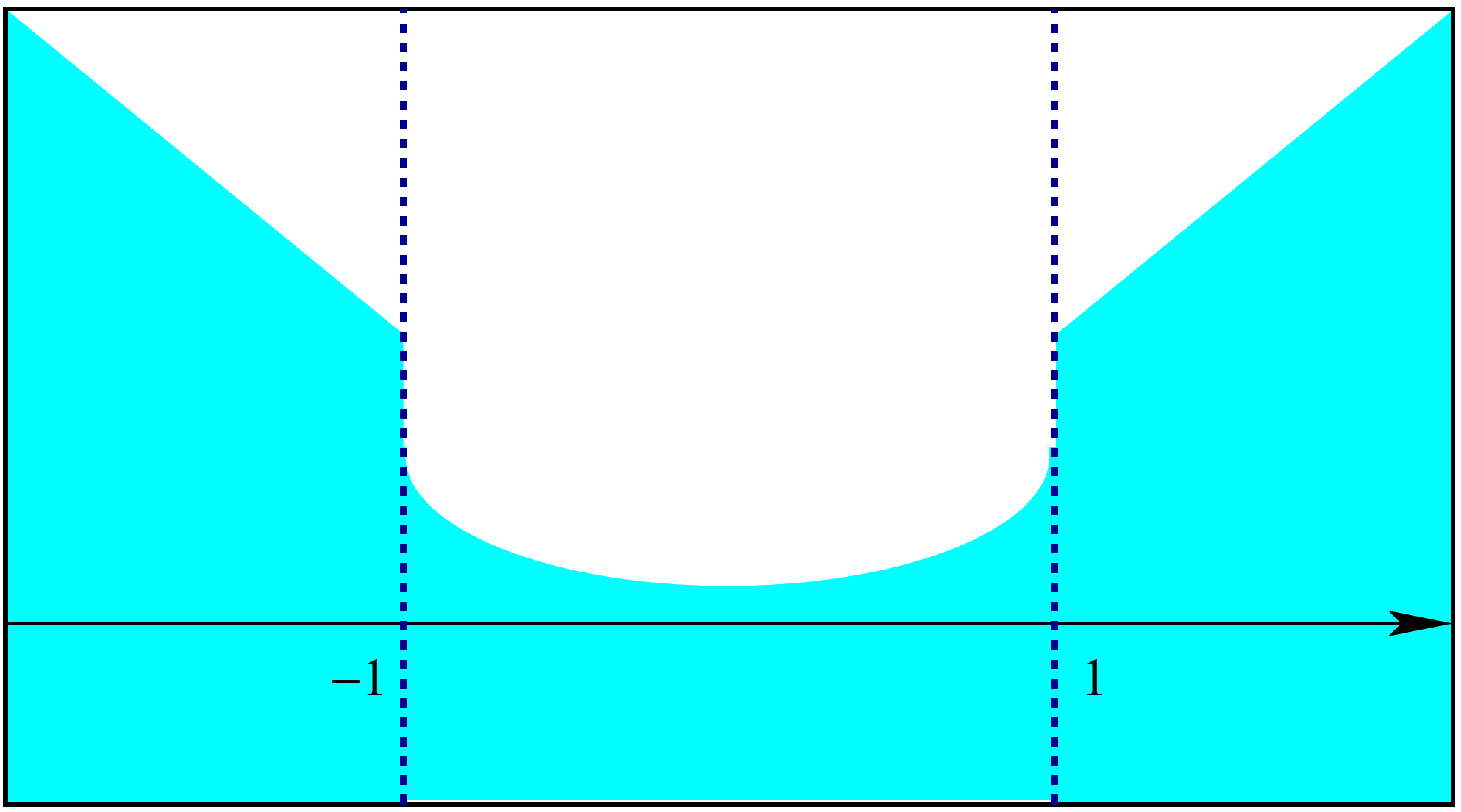}\hspace{0.3cm} 
{\bf \#6}\includegraphics[width=6.9cm]{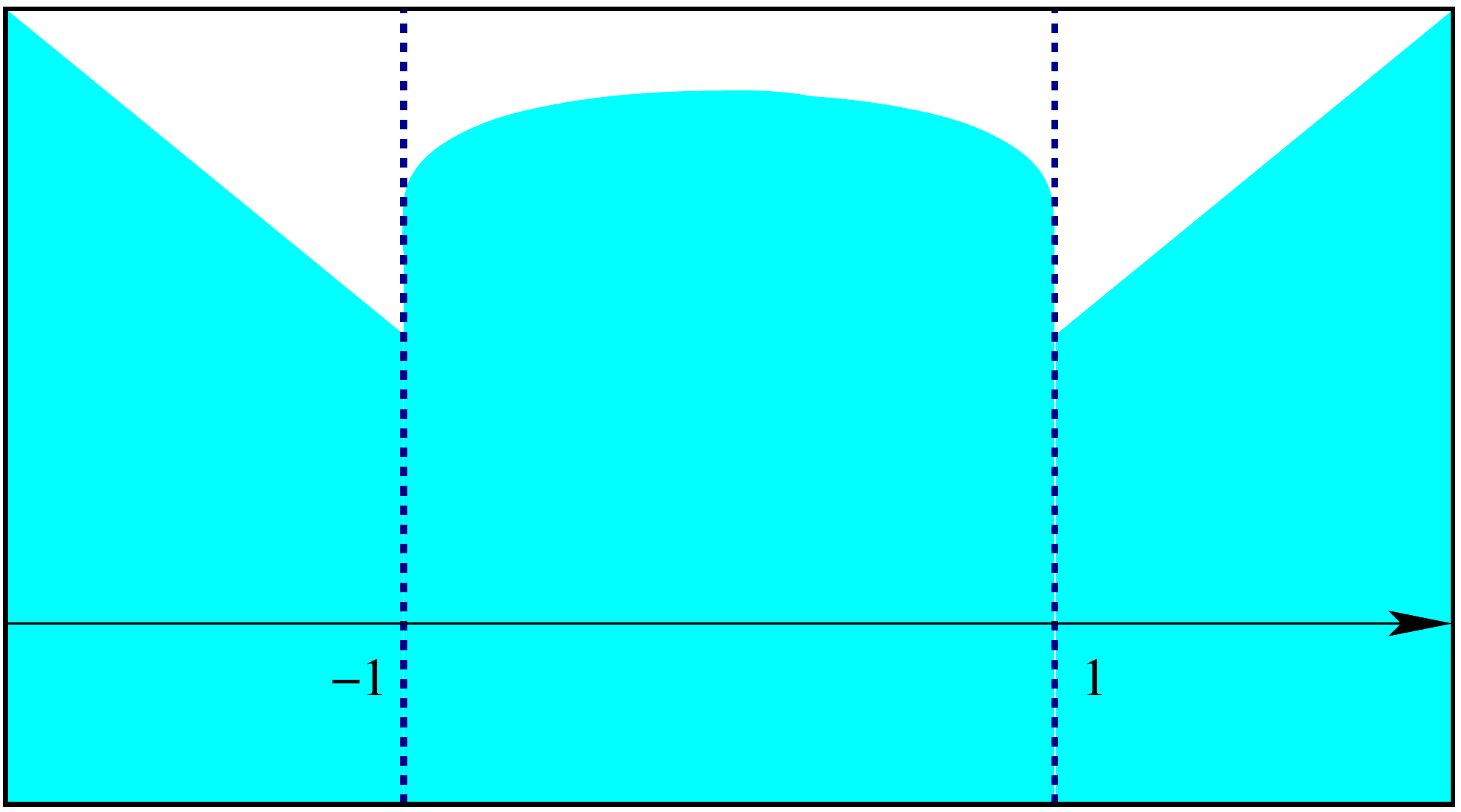}
\caption{\sl Confrontation between the suspects.}
\label{SOSPETTI}
\end{figure}

Then, we have the regular guys with some strange hobbies,
we know from TV that they are also quite plausible candidates
for being guilty; in our analogy, these are the usual suspects
number~3 and~4, which meet the boundary data in a Lipschitz or
H\"older fashion (and one may also observe that number~3
is the minimal set in the local case).\medskip

Then, we have the candidates which look above suspicion,
the ones to which nobody ever consider to be guilty, usually
the postman or the butler. In our analogy, these
are the suspects number~5 and~6, which are discontinuous
at the boundary.\medskip

Now, we know from TV how we should proceed:
if a suspect has a strong and verified alibi, we can rule him or her out
of the list. In our case, an alibi can be offered by
the necessary condition for $s$-minimality
given in Lemma~\ref{iu678HHJKA}. Indeed, if one of our suspects~$E$
does not satisfy that~$H^s_E=0$ along~$(\partial E)\cap\Omega$,
then~$E$ cannot be $s$-minimal and we can cross out~$E$
from our list of suspects ($E$ has an alibi!).

Now, it is easily seen that all the suspects number~1, 2, 3, 4 and~5
have an alibi: indeed, from Figure~\ref{SOSPETTIa} we see that~$H^s_E(p)\ne0$,
since the set~$E$ occupies (in measure, weighted
by the kernel in~\eqref{HS-DEF})
more than a halfplane\footnote{Indeed, in view of~\eqref{HS-DEF},
we know that an $s$-minimal set, seen from
any point of the boundary, satisfies a perfect balance between the
weighted measure of the set itself and the weighted measure
of its complement (here, weighted is intended with
respect to the kernel in~\eqref{HS-DEF}). Since a halfplane
also satisfies such perfect balance when seen from
any point of its boundary (due to odd symmetry),
one can say that
a set is $s$-minimal
when,
at any point of its boundary, 
the weighted contributions of
the set and its complement produce the same result as
the ones of a hyperplane passing through such point.
This geometric trick often allows us to ``subtract
the tangent halfplane'' from a set without modifying its
fractional curvature (and this is often convenient to
observe cancellations).}
passing through~$p$:
in Figure~\ref{SOSPETTIa} the point~$p$ is the big dot
and the halfplane is marked by the line passing through it, so
a quick inspection confirms that the  alibis
of number~1, 2, 3, 4 and~5 check out, hence their nonlocal
mean curvature does not vanish at~$p$ and consequently they are not $s$-minimal sets.

On the other hand, the alibi of number~6 doesn't hold water.
Indeed, near~$p$, the set~$E$ is confined below the horizontal line,
but at infinity the set~$E$ go well beyond such line: these effects
might compensate each other and produce a vanishing
mean curvature.

\begin{figure}[ht]
{\bf \#1}\includegraphics[width=6.9cm]{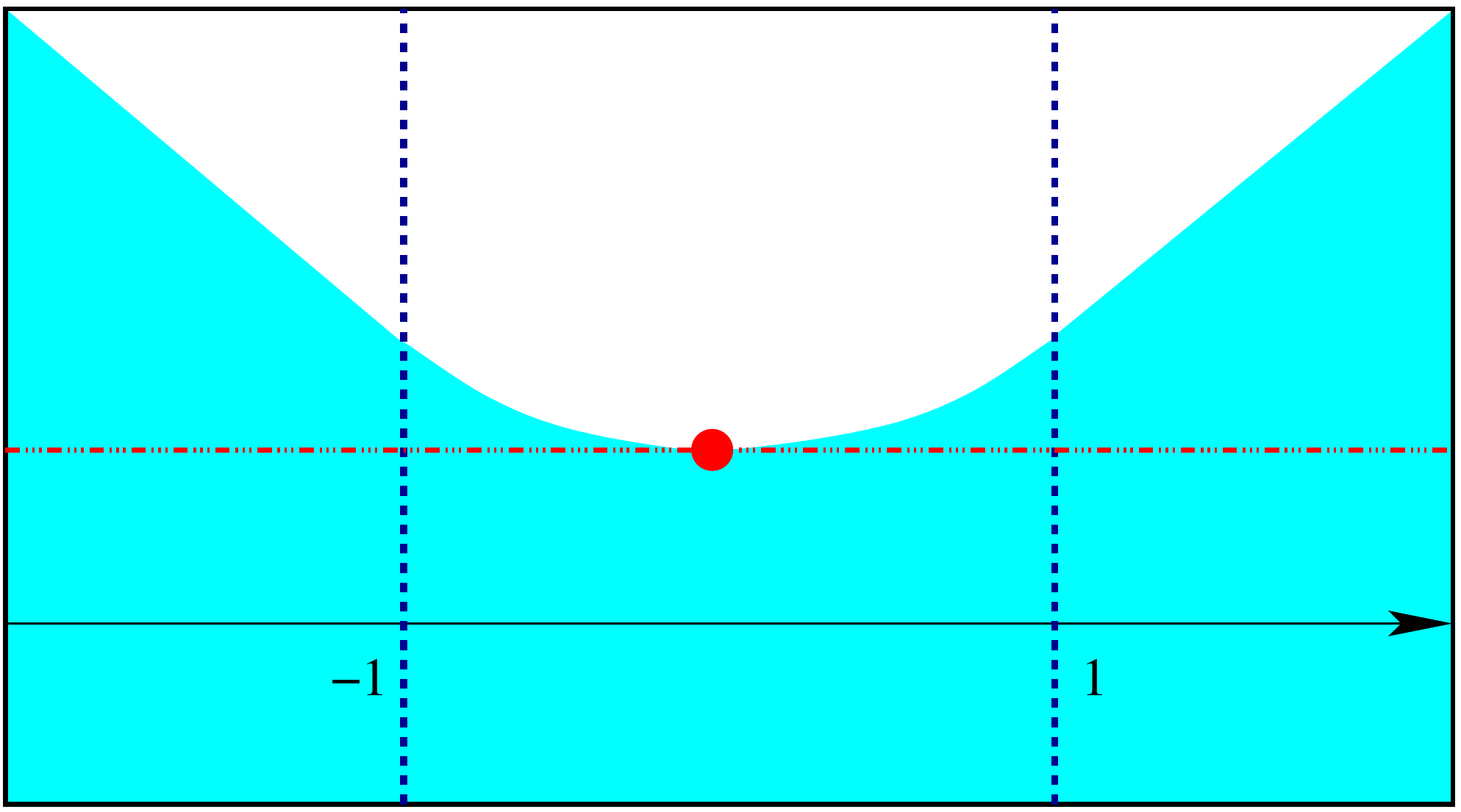}\hspace{0.3cm}
{\bf \#2}\includegraphics[width=6.9cm]{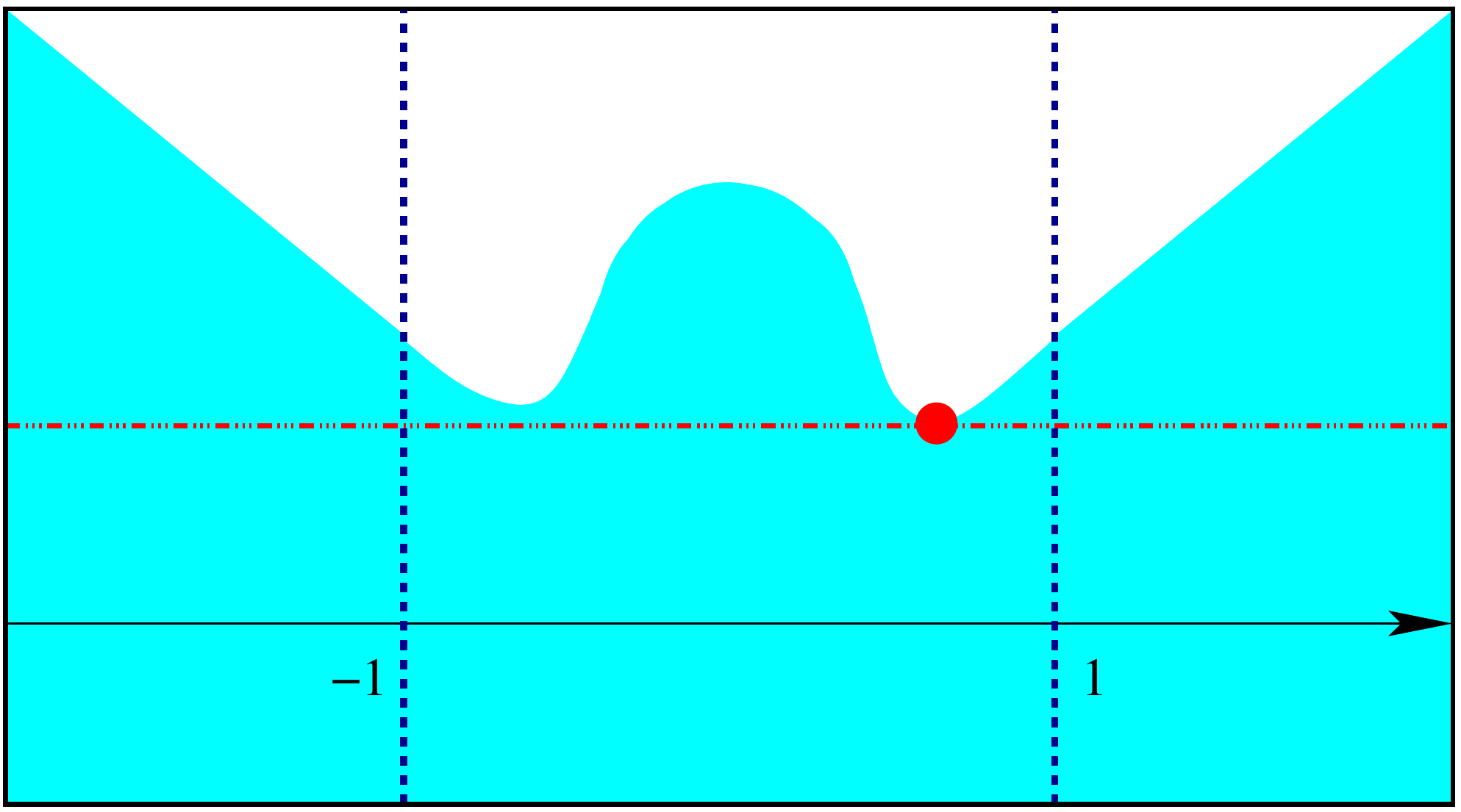}\vspace{0.3cm}
{\bf \#3}\includegraphics[width=6.9cm]{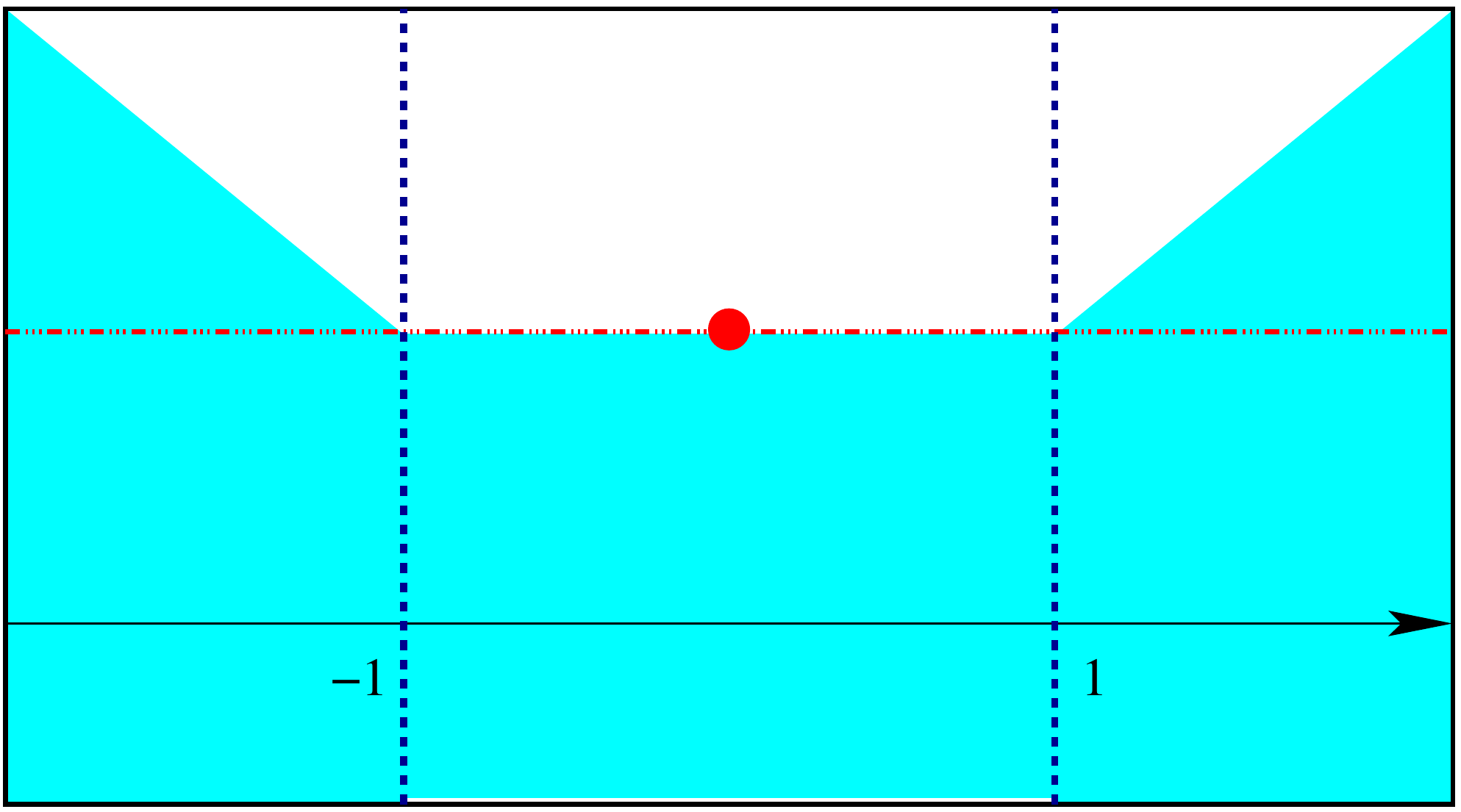}\hspace{0.3cm}
{\bf \#4}\includegraphics[width=6.9cm]{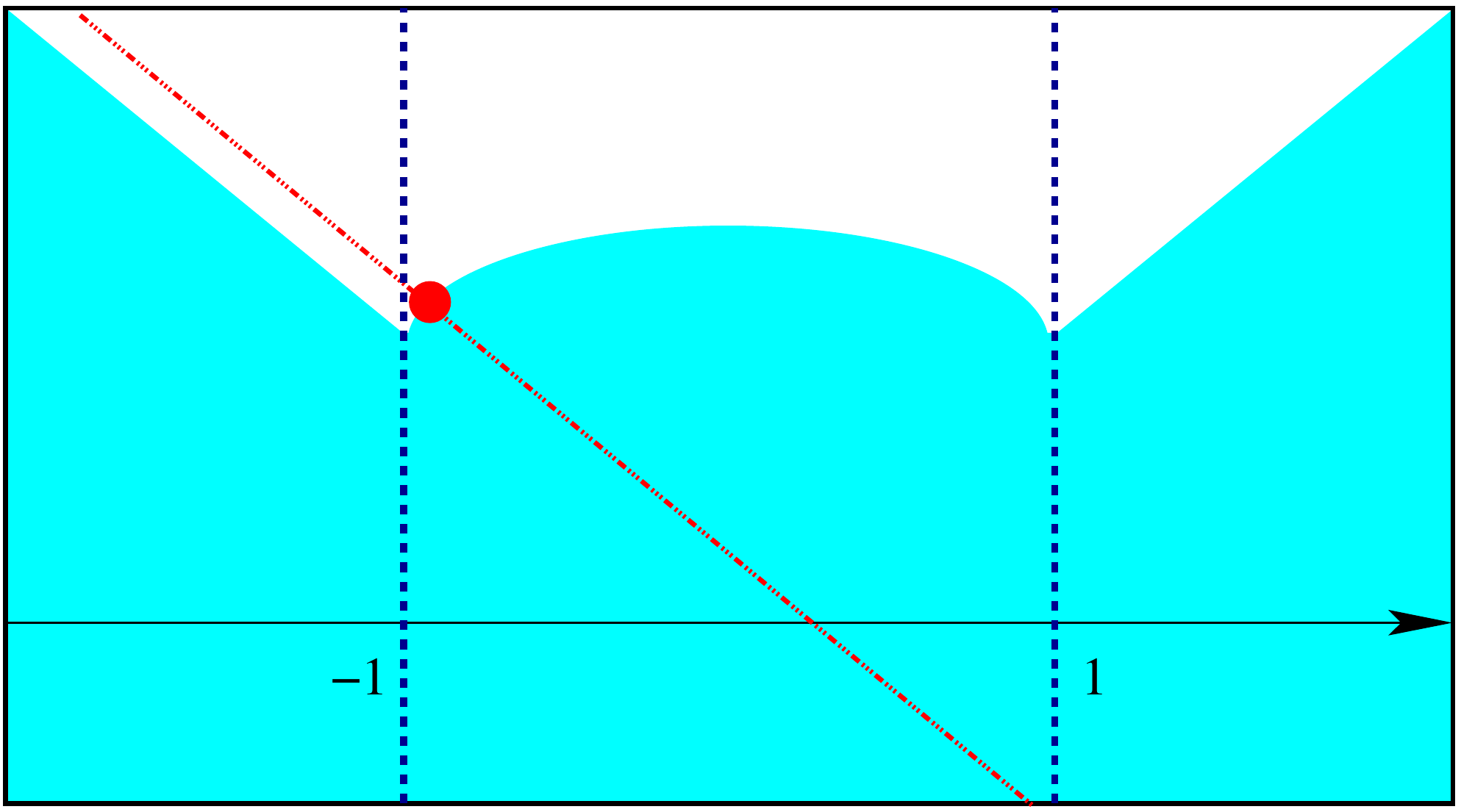}\vspace{0.3cm}
{\bf \#5}\includegraphics[width=6.9cm]{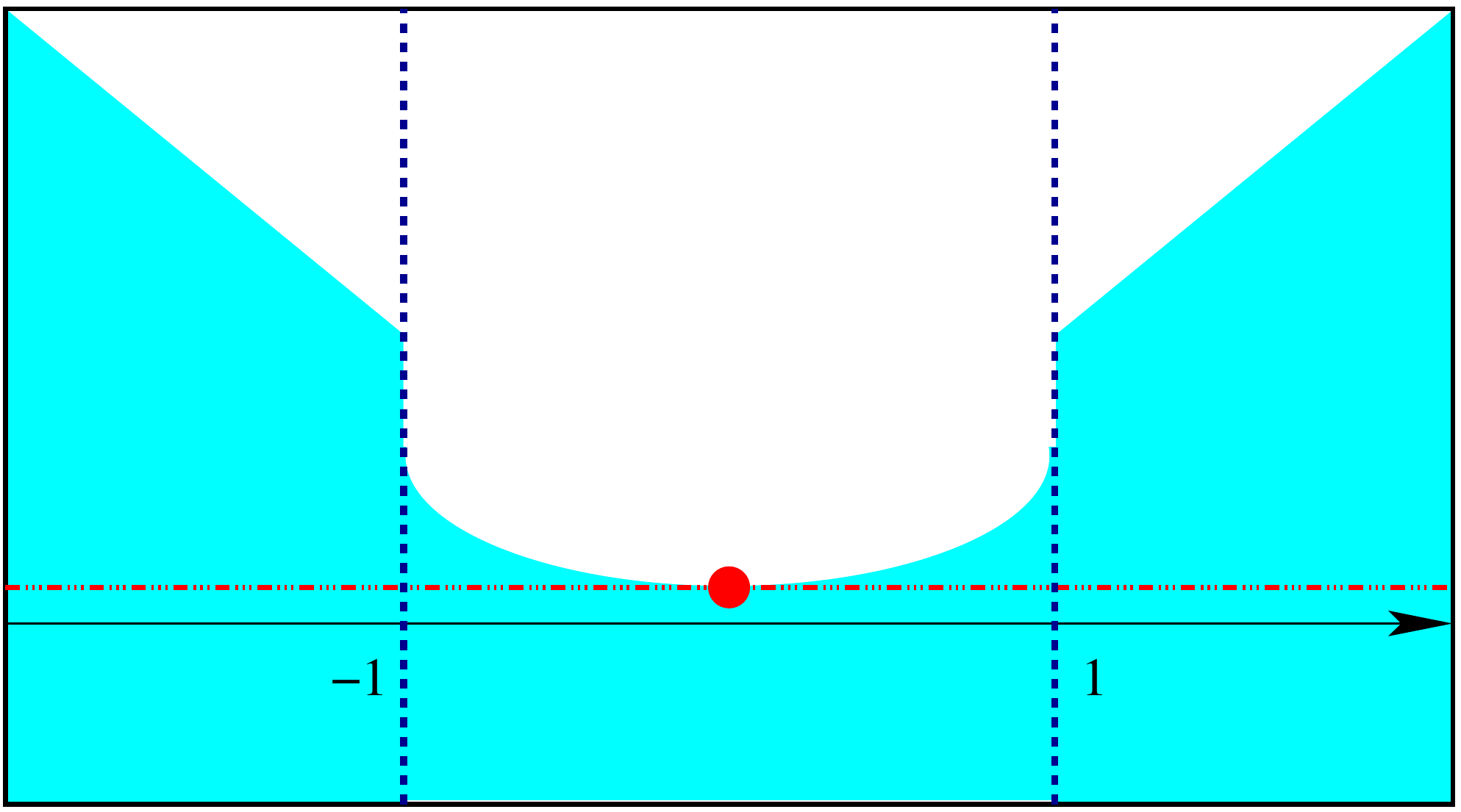}\hspace{0.3cm}
{\bf \#6}\includegraphics[width=6.9cm]{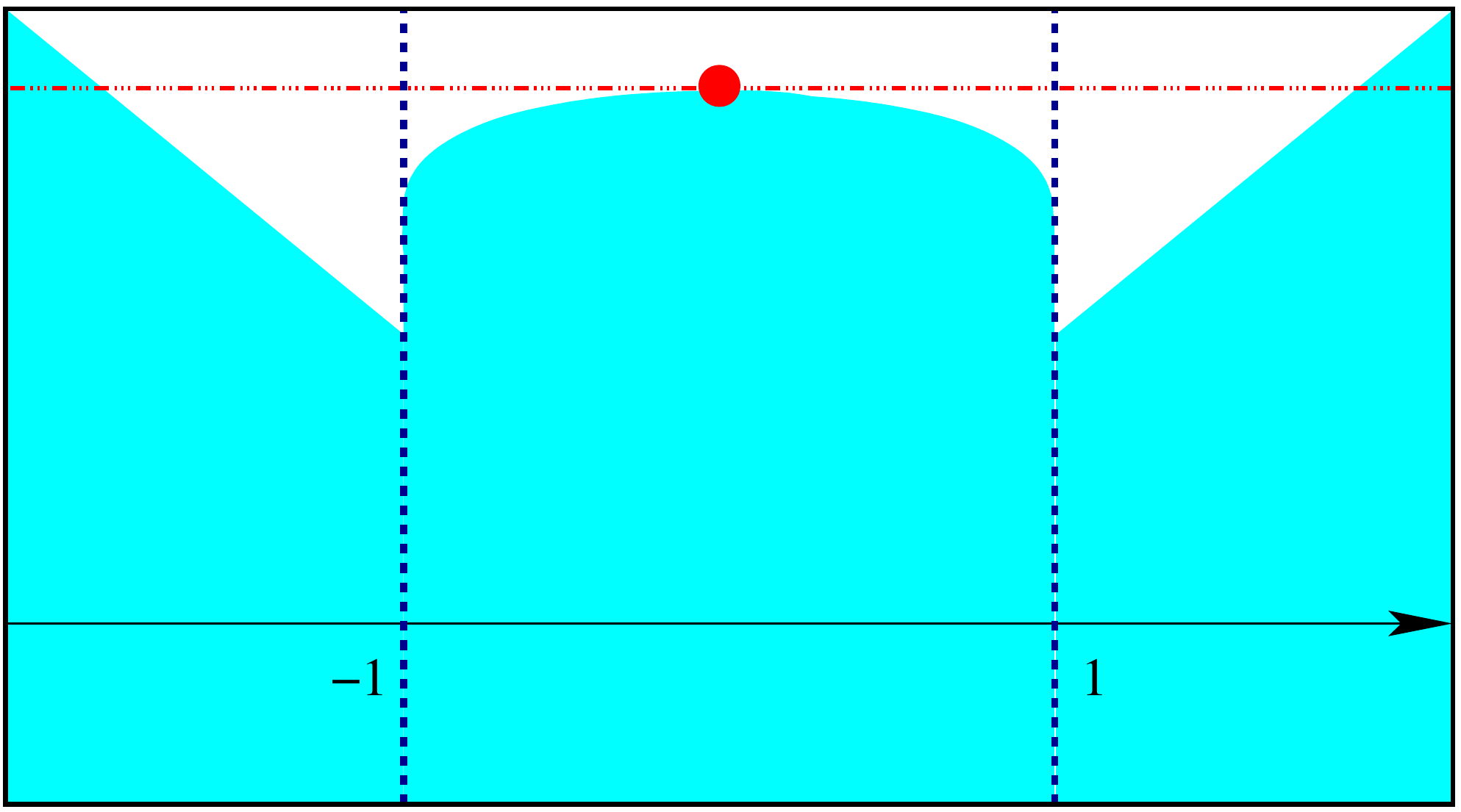}
\caption{\sl The alibis of the suspects.}
\label{SOSPETTIa}
\end{figure}

So, having ruled out all the suspects but number~6,
we have only to remember what the old investigators have taught us
(e.g., ``When you have eliminated the impossible, whatever remains,
however improbable, must be the truth''), to find that
the only possible (though, in principle, rather improbable)
culprit is number~6.\medskip

Of course, once that we know that 
the butler did it, i.e. that number~6 is $s$-minimal,
it is our duty to prove it beyond any reasonable doubt.
Many pieces of evidence, and a complete proof, is
given in~\cite{stick} (where indeed the more general version
given in Theorem~\ref{PER} is established). Here,
we provide some ideas towards the proof of Theorem~\ref{PER}
in Section~\ref{PER:S}.

\medskip

This set of notes is organized as follows. In Section~\ref{iu678HHJKA:PF} we present the proof 
of Lemma~\ref{iu678HHJKA}. Sections~\ref{sec:primo} and~\ref{sec:secondo} 
are devoted to the proofs of the quantitative estimates in Theorems~\ref{B V T} and~\ref{1123THJ}, 
respectively. Then, Section~\ref{PER:S} is dedicated to a sketch of the proof of 
Theorem~\ref{PER}. We also provide Appendix~\ref{APP12}
to discuss briefly the asymptotics of the $s$-perimeter as $s\nearrow1/2$
and as $s\searrow0$ and Appendix~\ref{APP1W} to discuss
the asymptotic expansion of the nonlocal mean curvature as $s\searrow0$.
Finally, in Appendix~\ref{9ojknAAsaw} we discuss the second variation 
of the fractional perimeter functional. 

\section{Proof of Lemma~\ref{iu678HHJKA}}\label{iu678HHJKA:PF}

\begin{proof}[Proof of Lemma~\ref{iu678HHJKA}]
We consider a diffeomorphism~$T_\epsilon(x):=x+\epsilon v(x)$,
with~$v\in C^\infty_0(\Omega,\R^n)$ and we take~$E_\epsilon:=T_\epsilon(E)$.
By minimality, we know that~${\rm Per}_s\,(E_\epsilon,\Omega)\ge
{\rm Per}_s\,(E,\Omega)$ for every~$\epsilon\in (-\epsilon_0,\epsilon_0)$, 
with~$\epsilon_0>0$ sufficiently small, hence
\begin{equation}\label{89iokGHA:IK}
{\rm Per}_s\,(E_\epsilon,\Omega) - {\rm Per}_s\,(E,\Omega) = o(\epsilon). 
\end{equation}
Suppose, for simplicity, that~$I(E\setminus\Omega ,E^c\setminus\Omega)<+\infty$,
so that we can write
\begin{equation*}
{\rm Per}_s\,(E_\epsilon,\Omega) - {\rm Per}_s\,(E,\Omega) = 
I(E_\epsilon,E_\epsilon^c)-I(E,E^c).
\end{equation*}
Moreover, if we use the notation~$X:=T_\epsilon^{-1}(x)$, we have that
$$ dx = |\det DT_\epsilon(X)|\,dX
= \big(1+\epsilon\,{\rm div}\, v(X)+o(\epsilon)\big)\,dX.$$
Similarly, if~$Y:=T_\epsilon^{-1}(y)$, we find that
\begin{eqnarray*}
&& |x-y|^{-n-2s}\\ &=& |T_\epsilon(X)-T_\epsilon(Y)|^{-n-2s}
\\&=&\big|X-Y +\epsilon\big(v(X)-v(Y)\big)|^{-n-2s} \\
&=&\big|X-Y|^{-n-2s} -(n+2s)\,\epsilon\,|X-Y|^{-n-2s-2}
(X-Y)\cdot \big(v(X)-v(Y)\big)+o(\epsilon).
\end{eqnarray*}
As a consequence,
\begin{eqnarray*}
&& {\rm Per}_s\,(E_\epsilon,\Omega) - {\rm Per}_s\,(E,\Omega)\\
&=&
\iint_{E_\epsilon\times E_\epsilon^c} \frac{dx\,dy}{|x-y|^{n+2s}}
-
\iint_{E\times E^c} \frac{dx\,dy}{|x-y|^{n+2s}} \\
&=&
\iint_{E\times E^c} 
\Big[\big|X-Y|^{-n-2s} -(n+2s)\,\epsilon\,|X-Y|^{-n-2s-2}
(X-Y)\cdot \big(v(X)-v(Y)\big)\Big]
\\ &&\qquad\cdot\big(1+\epsilon\,{\rm div}\, v(X)\big)
\big(1+\epsilon\,{\rm div}\, v(Y)\big)\,dX\,dY
\\ &&\qquad-
\iint_{E\times E^c} \frac{dx\,dy}{|x-y|^{n+2s}}
+o(\epsilon)
\\ &=& -(n+2s)\,\epsilon\, \iint_{E\times E^c} \frac{ (x-y)\cdot 
\big(v(x)-v(y)\big) }{|x-y|^{n+2s+2}}\,dx\,dy\\&&\qquad+
\epsilon\, \iint_{E\times E^c} \frac{ {\rm div}\, v(x)+{\rm div}\, v(y)
}{|x-y|^{n+2s}}\,dx\,dy+o(\epsilon).
\end{eqnarray*}
Now we point out that
$$ {\rm div}_x\, \frac{v(x)}{|x-y|^{n+2s}}= -(n+2s)\frac{v(x)\cdot(x-y)}{|x-y|^{n+2s+2}}+
\frac{{\rm div}_x\,v(x)}{|x-y|^{n+2s}}$$
and so, interchanging the names of the variables,
$$ {\rm div}_y\, \frac{v(y)}{|x-y|^{n+2s}}
= (n+2s)\frac{v(y)\cdot(x-y)}{|x-y|^{n+2s+2}}+
\frac{{\rm div}_y\,v(y)}{|x-y|^{n+2s}}
.$$
Consequently,
\begin{eqnarray*}
&&{\rm Per}_s\,(E_\epsilon,\Omega) - {\rm Per}_s\,(E,\Omega)\\
&=& \epsilon\,\iint_{E\times E^c}
\left[{\rm div}_x\, \frac{v(x)}{|x-y|^{n+2s}} +
{\rm div}_y\, \frac{v(y)}{|x-y|^{n+2s}}\right]\,dx\,dy+o(\epsilon)
.\end{eqnarray*}
Now, using
the Divergence Theorem 
and changing the names of the variables
we have that
\begin{eqnarray*}
\iint_{E\times E^c} {\rm div}_x\, 
\frac{v(x)}{|x-y|^{n+2s}}\,dx\,dy
&=&\int_{E^c}\,dy \left[
\int_{\partial E} \frac{v(x)\cdot\nu(x)}{|x-y|^{n+2s}}\,d{\mathcal{H}}^{n-1}(x)
\right]\\
&=& \int_{E^c}\,dx \left[
\int_{\partial E} \frac{v(y)\cdot\nu(y)}{|x-y|^{n+2s}}\,d{\mathcal{H}}^{n-1}(y)
\right]\end{eqnarray*}
and
$$ \iint_{E\times E^c} {\rm div}_y\,
\frac{v(y)}{|x-y|^{n+2s}}\,dx\,dy
=-\int_{E}\,dx \left[
\int_{\partial E} \frac{v(y)\cdot\nu(y)}{|x-y|^{n+2s}}\,d{\mathcal{H}}^{n-1}(y)
\right].$$
Accordingly, we find that
\begin{eqnarray*}
&&{\rm Per}_s\,(E_\epsilon,\Omega) - {\rm Per}_s\,(E,\Omega)\\
&=& \epsilon\,\int_{\partial E} \,d{\mathcal{H}}^{n-1}(y) v(y)\cdot\nu(y)
\left[ \int_{E^c} \frac{dx}{|x-y|^{n+2s}} - \int_{E} \frac{dx}{|x-y|^{n+2s}}
\right] +o(\epsilon)\\
&=& \epsilon\,\int_{\partial E} 
v(y)\cdot\nu(y) \, H_E^s(y)
\,d{\mathcal{H}}^{n-1}(y)+o(\epsilon).
\end{eqnarray*}
Comparing with~\eqref{89iokGHA:IK}, we see that
$$ \int_{\partial E}
v(y)\cdot\nu(y) \, H_E^s(y)
\,d{\mathcal{H}}^{n-1}(y)=0$$
and so, since~$v$ is an arbitrary vector field supported in~$\Omega$,
the desired result
follows.
\end{proof}

\section{Proof of Theorem~\ref{B V T}}\label{sec:primo}

The basic idea goes as follows. 
One uses
the appropriate combination of two general facts: on the one 
hand, one can perturb a given set by a smooth flow and compare the 
energy at time~$ t$ with the one at time~$-t$, thus obtaining a second order 
estimate; on the other hand, the nonlocal interaction always charges a 
mass on points that are sufficiently close, thus providing a natural 
measure for the discrepancy between the original set and its flow. One 
can appropriately combine these two facts with the minimality (or more 
generally, the stability) property of a set. Indeed, by choosing as 
smooth flow a translation near the origin, the above arguments lead to 
an integral estimate of the discrepancy between the set and its 
translations, which in turn implies a perimeter estimate.
\medskip

We now give the details of the proof of Theorem~\ref{B V T}.
To do this, we fix~$R\ge1$, a direction~$v\in S^{n-1}$,
a function~$\varphi\in C^\infty_0(B_{9/10})$
with~$\varphi=1$ in~$B_{3/4}$, and a
small scalar quantity~$t\in\left( -\frac{1}{100},\,\frac{1}{100}\right)$,
and we consider the diffeomorphism~$\Phi^t\in C^\infty_0(B_{9/10})$
given by~$\Phi^t(x):=x + t\varphi(x/R)\,v$.
Notice that
\begin{equation}\label{phi t}
{\mbox{$\Phi^t(x)=x+tv$ for any~$x\in B_{3R/4}$.}}\end{equation} 
We also define~$E_t:=
\Phi^t(E)$.
We have the following useful auxiliary estimates
(that will be used in the proofs of both 
Theorem~\ref{B V T} and Theorem~\ref{1123THJ}):

\begin{lemma}\label{USEFUL LEMMA}
Let~$E$ 
be a minimizer for the $s$-perimeter in~$B_R$. Then
\begin{eqnarray}
&& \label{Yh:002}
{\rm Per}_s\,(E_t,B_{R})+{\rm Per}_s\,(E_{-t},B_{R})
-2{\rm Per}_s\,(E,B_{R}) \le CR^{n-2s-2}\,t^2,\\
&& \label{YH:09}
2 I(E_{t}\setminus E, E\setminus E_{t})\le CR^{n-2s-2}\,t^2,\\
\label{GH:2TXT}&&
\min\Big\{
\big| ( (E+ tv)\setminus E)\cap B_{R/2}\big|,\;
\big| (E\setminus (E+ tv))\cap B_{R/2}\big|
\Big\} \le C\,R^{\frac{n-2s-2}{2}}\,|t|,\end{eqnarray}
and
\begin{equation}\label{O9056}
\min\left\{
\int_{B_{R/2}}
\big( \chi_{E}(x+tv)-\chi_E(x)\big)_+\,dx, \;
\int_{B_{R/2}}
\big( \chi_{E}(x+tv)-\chi_E(x)\big)_-\,dx
\right\}\le C\,R^{\frac{n-2s-2}{2}}\,|t|,
\end{equation}
for some~$C>0$.
\end{lemma}

\begin{proof} First we observe that
\begin{equation}\label{Yh:001}
{\rm Per}_s\,(E_t,B_{R})+{\rm Per}_s\,(E_{-t},B_{R})
-2{\rm Per}_s\,(E,B_{R}) \le \frac{Ct^2}{R^2} {\rm Per}_s\,(E,B_{R}),
\end{equation}
for some~$C>0$.
This is indeed
a general estimate, which does not use
minimality,
and which follows
by changing variable in the integrals of the fractional perimeter
(and noticing that the linear term in~$t$ simplifies).
We provide some details of the proof of~\eqref{Yh:001}
for the facility of the reader.
To this aim, we observe that
$$ |\det D\Phi^t(X)| = |\det ( {\bf 1} + tR^{-1}
\nabla \varphi(X/R)\otimes v)|
= 1 +tR^{-1} \nabla \varphi(X/R)\cdot v + O(t^2R^{-2}).$$
Moreover, if, for any~$\xi$, $\eta\in\R^n$, we set
$$ g(\xi,\eta):= \frac{\big(\varphi(\xi)-\varphi(\eta)\big)\,v}{|\xi-\eta|},$$
we have that~$g$ is bounded and
\begin{eqnarray*}
|\Phi^t(X)-\Phi^t(Y)| &=&
\left| X-Y+ t\big(\varphi(X/R)-\varphi(Y/R)\big)\,v \right| \\
&=& |X-Y|\, 
\left| \frac{X-Y}{|X-Y|}+ tR^{-1}
\frac{\big(\varphi(X/R)-\varphi(Y/R)\big)\,v}{|(X/R)-(Y/R)|} \right|
\\
&=& |X-Y|\, 
\left| \frac{X-Y}{|X-Y|}+ tR^{-1} g(X/R,Y/R)
\right|.
\end{eqnarray*}
Therefore
\begin{eqnarray*}
|\Phi^t(X)-\Phi^t(Y)|^{-n-2s} &=&
|X-Y|^{-n-2s}\,
\left| \frac{X-Y}{|X-Y|}+ tR^{-1} g(X/R,Y/R)
\right|^{-n-2s}
\\ &=& |X-Y|^{-n-2s} \left(
1-(n+2s) t R^{-1} \frac{X-Y}{|X-Y|}\cdot g(X/R,Y/R)
+O(t^2R^{-2})
\right).\end{eqnarray*}
Now we observe that~$\Phi^t$ is the identity outside~$B_R$
and therefore if~$A\in \{ B_R, B_R^c,\R^n\}$
then~$E_t\cap A = \Phi^t(E\cap A)$.
Accordingly, for any~$A$, $B\in \{ B_R, B_R^c,\R^n\}$, 
a change of variables~$x:=\Phi^t(X)$ and~$y:=\Phi^t(Y)$
gives that
\begin{eqnarray*}&&
I(E_t\cap A, E_t^c\cap B)\\ &=&
\int_{\Phi^t(E\cap A)} \int_{\Phi^t(E\cap B)}
|x-y|^{-n-2s}\,dx\,dy \\
&=& 
\int_{E\cap A} \int_{E\cap B}
|\Phi^t(X)-\Phi^t(Y)|^{-n-2s}\,
|\det D\Phi^t(X)|\,
|\det D\Phi^t(Y)|\,dX\,dY
\\ &=&
\int_{E\cap A} \int_{E\cap B}
|X-Y|^{-n-2s} \left(
1-(n+2s) t R^{-1} \frac{X-Y}{|X-Y|}\cdot g(X/R,Y/R)
+O(t^2R^{-2})
\right) \\
&&\quad\cdot\left(
1 +tR^{-1} \nabla \varphi(X/R)\cdot v + O(t^2R^{-2})
\right)\left(
1 +tR^{-1} \nabla \varphi(Y/R)\cdot v + O(t^2R^{-2})
\right)\,dX\,dY
\\ &=&
\int_{E\cap A} \int_{E\cap B}
|X-Y|^{-n-2s} \left(
1-(n+2s) t R^{-1} \tilde g(X/R,Y/R)
+O(t^2R^{-2}) \right)\,dX\,dY,
\end{eqnarray*}
for a suitable scalar function~$\tilde g$.

Then, replacing $t$ with~$-t$ and summing up,
the linear term in~$t$ simplifies and we obtain
$$ I(E_t\cap A, E_t^c\cap B)+
I(E_{-t}\cap A, E_{-t}^c\cap B)
=
\big(2+O(t^2R^{-2})\big)\int_{E\cap A} \int_{E\cap B}
\frac{dX\,dY}{|X-Y|^{n+2s}}
.$$
This, choosing~$A$ and~$B$ appropriately, establishes~\eqref{Yh:001}.

On the other hand, the $s$-minimality of~$E$ gives that~$
{\rm Per}_s\,(E,B_{R})\le{\rm Per}_s\,(E\cup B_R,B_{R})$,
which, in turn, is bounded from above
by the interaction between~$B_R$
and~$ B_R^c$, namely~$
I(B_R, B_R^c)$,
which is a constant (only depending on~$n$
and~$s$) times~$R^{n-2s}$, due to scale invariance
of the fractional perimeter. 
That is, we have that~${\rm Per}_s\,(E,B_{R})\le CR^{n-2s}$,
for some~$C>0$, and then we can make the right hand side
of~\eqref{Yh:001} uniform in~$E$
and obtain~\eqref{Yh:002},
up to renaming $C>0$.

The next step is to charge mass in a ball.
Namely, one defines~$E_t^\cup := E\cup E_t$ and~$E_t^\cap:=E\cap E_t$.
By counting the interactions of the different sets, one sees that
\begin{equation}\label{Yh:003}
{\rm Per}_s\,(E,B_{R})
+{\rm Per}_s\,(E_t,B_{R})
-{\rm Per}_s\,(E_t^\cup,B_{R})-{\rm Per}_s\,(E_{t}^\cap,B_{R})
= 2 I(E_t\setminus E, E\setminus E_t).
\end{equation}
To check this, one observes indeed that
the set~$E_t\setminus E$ interacts with~$E\setminus E_t$
in the computations of~${\rm Per}_s\,(E,B_{R})$
and~${\rm Per}_s\,(E_t,B_{R})$,
while these two sets do not interact
in the computations of~${\rm Per}_s\,(E_t^\cup,B_{R})$
and~${\rm Per}_s\,(E_{t}^\cap,B_{R})$ (the interactions
of the other sets simplify). This proves~\eqref{Yh:003}.
We remark that,
again, formula~\eqref{Yh:003}
is a general fact and is not based on minimality.
Changing~$t$ with~$-t$, we also obtain from~\eqref{Yh:003} that
\begin{equation*}
{\rm Per}_s\,(E,B_{R})
+{\rm Per}_s\,(E_{-t},B_{R}) 
-{\rm Per}_s\,(E_{-t}^\cup,B_{R})-{\rm Per}_s\,(E_{-t}^\cap,B_{R})
= 2 I(E_{-t}\setminus E, E\setminus E_{-t}).
\end{equation*}
This and~\eqref{Yh:003} give that
\begin{eqnarray*}&&
{\rm Per}_s\,(E_{t},B_{R})+
{\rm Per}_s\,(E_{-t},B_{R}) -2{\rm Per}_s\,(E,B_{R})\\
&=& 
{\rm Per}_s\,(E_{t}^\cup,B_{R})+{\rm Per}_s\,(E_{t}^\cap,B_{R})
+{\rm Per}_s\,(E_{-t}^\cup,B_{R})+{\rm Per}_s\,(E_{-t}^\cap,B_{R})
\\ &&\qquad-4{\rm Per}_s\,(E,B_{R})
+2 I(E_{t}\setminus E, E\setminus E_{t})
+2 I(E_{-t}\setminus E, E\setminus E_{-t})\\
&\ge& 2 I(E_{t}\setminus E, E\setminus E_{t})
+2 I(E_{-t}\setminus E, E\setminus E_{-t}),
\end{eqnarray*}
thanks to the $s$-minimality of~$E$. In particular,
$$ {\rm Per}_s\,(E_{t},B_{R})+
{\rm Per}_s\,(E_{-t},B_{R}) -2{\rm Per}_s\,(E,B_{R})\ge
2 I(E_{t}\setminus E, E\setminus E_{t}).$$
This and~\eqref{Yh:002} imply~\eqref{YH:09}.

Now, the interaction kernel is bounded away from zero in~$B_{R/2}$,
and so
$$ I(E_{t}\setminus E, E\setminus E_{t})\ge
\big| (E_{t}\setminus E)\cap B_{R/2}\big|\cdot
\big| (E\setminus E_t)\cap B_{R/2}\big|.$$
This is again a general fact, not
depending on
minimality.
By plugging this into~\eqref{YH:09}, we conclude that
\begin{eqnarray*} 
CR^{n-2s-2}\,t^2 &\ge& 
\big| (E_{t}\setminus E)\cap B_{R/2}\big|\cdot
\big| (E\setminus E_t)\cap B_{R/2}\big|\\
&\ge& \min\Big\{
\big| (E_{t}\setminus E)\cap B_{R/2}\big|^2,\;
\big| (E\setminus E_t)\cap B_{R/2}\big|^2
\Big\}\end{eqnarray*}
and so, again up to renaming~$C$,
\begin{equation}\label{phi tt}
\min\Big\{
\big| (E_{t}\setminus E)\cap B_{R/2}\big|,\;
\big| (E\setminus E_t)\cap B_{R/2}\big|
\Big\} \le CR^{\frac{n-2s-2}{2}}\,t.\end{equation}
Now, we recall~\eqref{phi t} and we observe that~$E_t\cap B_{R/2}
=(E+tv)\cap B_{R/2}$. Hence, the estimate in~\eqref{phi tt}
becomes
\begin{equation}\label{GH:1}
\min\Big\{
\big| ( (E+tv)\setminus E)\cap B_{R/2}\big|,\;
\big| (E\setminus (E+tv))\cap B_{R/2}\big|
\Big\} \le CR^{\frac{n-2s-2}2}\,t.\end{equation}
Since this is valid for any~$v\in S^{n-1}$,
we may also switch the sign of~$v$ and obtain that
\begin{equation}\label{GH:2}             
\min\Big\{
\big| ( (E-tv)\setminus E)\cap B_{R/2}\big|,\;
\big| (E\setminus (E-tv))\cap B_{R/2}\big|
\Big\} \le CR^{\frac{n-2s-2}2}\,t.\end{equation}
{F}rom~\eqref{GH:1} and~\eqref{GH:2} we obtain~\eqref{GH:2TXT}.

Now we observe that, for any sets~$A$ and~$B$,
\begin{equation}\label{KA:901}
\chi_{A\setminus B}(x) \ge \chi_A(x)-\chi_B(x)
.\end{equation}
Indeed, this formula is clearly true if~$x\in B$,
since in this case the right hand side is nonpositive.
The formula is also true if~$x\in A\setminus B$,
since in this case the left hand side is~$1$ and the
right hand side is less or equal than~$1$.
It remains to consider the case in which~$x\not\in
A\cup B$. 
In this case, $\chi_A(x)=0$,
hence the right hand side is nonpositive, which gives
that~\eqref{KA:901} holds true.

By~\eqref{KA:901},
$$ \chi_{A\setminus B}(x) \ge \big( \chi_A(x)-\chi_B(x)\big)_+. $$
As a consequence,
\begin{eqnarray*}
&& \big| ( (E-tv)\setminus E)\cap B_{R/2}\big|
=\int_{B_{R/2}} \chi_{ (E-tv)\setminus E }(x)\,dx
\\ &&\quad\ge \int_{B_{R/2}}
\big( \chi_{E-tv}(x)-\chi_E(x)\big)_+\,dx
= \int_{B_{R/2}}
\big( \chi_{E}(x+tv)-\chi_E(x)\big)_+\,dx \\
{\mbox{and }}&&
\big| (E\setminus (E-tv))\cap B_{R/2}\big|
=\int_{B_{R/2}} \chi_{ E\setminus (E-tv) }(x)\,dx
\\ &&\quad\ge \int_{B_{R/2}}
\big( \chi_{E}(x)-\chi_{E-tv}(x)\big)_+\,dx
= \int_{B_{R/2}}
\big( \chi_{E}(x)-\chi_E(x+tv)\big)_+\,dx 
\\ &&\quad= \int_{B_{R/2}}
\big( \chi_{E}(x+tv)-\chi_E(x)\big)_-\,dx.
\end{eqnarray*}
This and~\eqref{GH:2} give that
\begin{equation*}
CR^{\frac{n-2s-2}2}\,t\ge 
\min\left\{
\int_{B_{R/2}}
\big( \chi_{E}(x+tv)-\chi_E(x)\big)_+\,dx, \;
\int_{B_{R/2}}
\big( \chi_{E}(x+tv)-\chi_E(x)\big)_-\,dx
\right\},
\end{equation*}
which is~\eqref{O9056}. This ends the proof of Lemma~\ref{USEFUL LEMMA}.\end{proof}

With the preliminary work done in
Lemma~\ref{USEFUL LEMMA} (to be used here
with~$R=1$), we can now complete the proof
of Theorem~\ref{B V T}. To this end, we observe that
\begin{equation}\label{hj987654rgKA:1}
\begin{split}
& \left|\int_{B_{1/2}}
\big( \chi_{E}(x+tv)-\chi_E(x)\big)_+\,dx -
\int_{B_{1/2}}
\big( \chi_{E}(x+tv)-\chi_E(x)\big)_-\,dx\right|
\\ =\;& \left|\int_{B_{1/2}}
\big( \chi_{E}(x+tv)-\chi_E(x)\big)\,dx\right| \\=\;&\left|
\int_{B_{1/2}-tv} \chi_{E}(x)\,dx-
\int_{B_{1/2}}\chi_E(x)\,dx \right| \\ \le\;&
\big| (B_{1/2}-tv)\Delta B_{1/2}\big|\\ 
\le\;& C\,t,
\end{split}\end{equation}
for some~$C>0$.

Also, we observe that, for any~$a$, $b\in\R$,
\begin{equation}\label{hj987654rgKA:2}
a+b \le |a-b| + 2\min\{a,b\}
.\end{equation}
Indeed, up to exchanging~$a$ and~$b$, we may suppose that~$a\ge b$; thus
$$ a+b = a-b +2b = |a-b| + 2\min\{a,b\},$$
which proves~\eqref{hj987654rgKA:2}.

Using~\eqref{O9056}, \eqref{hj987654rgKA:1} and~\eqref{hj987654rgKA:2},
we obtain that
\begin{eqnarray*}
&& \int_{B_{1/2}}
\big| \chi_{E}(x+tv)-\chi_E(x)\big|\,dx\\
&=& \int_{B_{1/2}}
\big( \chi_{E}(x+tv)-\chi_E(x)\big)_+\,dx +
\int_{B_{1/2}}
\big( \chi_{E}(x+tv)-\chi_E(x)\big)_-\,dx \\ &\le&\left|
\int_{B_{1/2}}
\big( \chi_{E}(x+tv)-\chi_E(x)\big)_+\,dx -
\int_{B_{1/2}}
\big( \chi_{E}(x+tv)-\chi_E(x)\big)_-\,dx\right|\\ &&\quad+
2\min\left\{
\int_{B_{1/2}}
\big( \chi_{E}(x+tv)-\chi_E(x)\big)_+\,dx ,\;
\int_{B_{1/2}}
\big( \chi_{E}(x+tv)-\chi_E(x)\big)_-\,dx\right\}
\\ &\le& Ct,
\end{eqnarray*}
up to renaming~$C$. Dividing by~$t$ and sending~$t\searrow0$
(up to subsequences), one finds that
$$ \int_{B_{1/2}} |\partial_v \chi_E(x)|\,dx\le C,$$
for any~$v\in S^{n-1}$, in the bounded variation sense.
Since the direction~$v$ is arbitrary,
this proves that
$$ {\rm Per}\,(E,B_{1/2})= \int_{B_{1/2}} |\nabla \chi_E(x)|\,dx
\le C.$$
This proves Theorem~\ref{B V T} with~$R=1$,
and the general case follows from scaling.

\section{Proof of Theorem \ref{1123THJ}}\label{sec:secondo}

In this part, we will make use
of some integral geometric formulas which compute the perimeter
of a set by averaging the number of intersections
of straight lines with the boundary of a set.

For this,
we recall the notation of the positive and negative
part of a function~$u$, namely
$$ u_+(x):= \max\{ u(x),\;0\} \qquad{\mbox{
and }}\qquad u_-(x):= \max\{ -u(x),\;0\}.$$
Notice that~$u_\pm\ge0$, that~$|u|=u_+ + u_-$ and that~$u=u_+ - u_-$.

Also, if~$v\in \partial B_1$ and~$p\in\R^n$, we define
\begin{eqnarray*}
&& v^\perp := \{ y\in\R^n {\mbox{ s.t. }} y\cdot v=0 \}\\
{\mbox{and }}
&& p+\R v:= \{ p+tv {\mbox{ s.t. }} t\in\R\}.
\end{eqnarray*}
That is, $v^\perp$ is the orthogonal linear space to~$v$
and~$p+\R v$ is the line passing through~$p$ with direction~$v$.

Now, given a Caccioppoli set~$E\subseteq\R^n$ with exterior normal~$\nu$
(and reduced boundary denoted by~$\partial^* E$),
and~$v\in \partial B_1$, we set
%% \begin{equation}\label{JK:001:BIS}
%% \int_{B_1\cap(\partial^* E)} |v\cdot\nu(x)|\,d{\mathcal{H}}^{n-1}(x)=
%% \int_{y\in v^\perp} {\mathcal{H}}^0 \big( B_1\cap (\partial^* E)\cap (y+\R v)\big)
%% \,d{\mathcal{H}}^{n-1}(y)
%% \end{equation}
%% and
%% \begin{equation}\label{JK:001}
%% \int_{B_1\cap(\partial^* E)} (v\cdot\nu(x))_\pm\,d{\mathcal{H}}^{n-1}(x)=
%% \int_{y\in v^\perp} I_{v,\pm}(y)\,d{\mathcal{H}}^{n-1}(y)
%% ,\end{equation}
%% where
\begin{equation}\label{JK:002}
I_{v,\pm}(y) := \sup \mp
\int_{y+\R v} \chi_E(x)\, \phi'(x)
\,d{\mathcal{H}}^1(x),
\end{equation}
with the~$\sup$ taken
over all smooth~$\phi$ supported in
the segment~$B_1\cap( {y+\R v} )$ with image in~$[0,1]$.
%% See for instance Theorem~1 in~\cite{MR758909}
%% for this and other formulas in integral geometry.
We have (see e.g. Proposition~4.4 in~\cite{joaq}) that
one can compute the directional
derivative in the sense of bounded variation by the formula
\begin{equation} \label{PAKJ087}
\int_{B_1} (\partial_v \chi_E)_\pm (x)\,dx
=\int_{y\in v^\perp} I_{v,\pm}(y)\,d{\mathcal{H}}^{n-1}(y)\end{equation}
and we also have that~$I_{v,\pm}(y)$ is the number of
points~$x$ that lie in~$B_1\cap (\partial^* E)\cap (y+\R v)$
and such that~$\mp v\cdot\nu(x)>0$.
That is, the quantity~$I_{v,+}(y)$ (resp.,~$I_{v,-}(y)$)
counts the number of intersections in the ball~$B_1$
between the line~$y+\R v$ and the (reduced) boundary of~$E$
that occur at points~$x$ in which~$v\cdot\nu(x)$ is negative
(resp., positive).
In particular,
\begin{equation}\label{JK:003}
I_{v,\pm}(y) \in \Z\cap [0,+\infty) =\{0,1,2,3,\dots\}.
\end{equation}
Furthermore, the vanishing of~$I_{v,+}(y)$ (resp., $I_{v,-}(y)$)
is related
to the fact that, moving along the segment~$B_1\cap( y+\R v )$,
one can only exit (resp., enter) the set~$E$, according
to the following result:

\begin{lemma}\label{NONE}
If~$I_{v,+}(y)=0$, then the map~$ B_1\cap( y+\R v )\ni x\mapsto
\chi_E(x)$ is nonincreasing.
\end{lemma}

\begin{proof} For any
smooth~$\phi$ supported in
the segment~$B_1\cap( {y+\R v} )$ with image in~$[0,1]$,
$$ 0=I_{v,+}(y)\ge
-\int_{y+\R v} \chi_E(x)\, \phi'(x)
\,d{\mathcal{H}}^1(x),$$
that is
$$ \int_{y+\R v} \chi_E(x)\, \phi'(x)
\,d{\mathcal{H}}^1(x)\ge0,$$
which gives the desired result.
\end{proof}

Now we define
\begin{equation}\label{BO:000}
\Phi_\pm (v):=\int_{y\in v^\perp} I_{v,\pm}(y)\,d{\mathcal{H}}^{n-1}(y)
.\end{equation}
By~\eqref{PAKJ087},
\begin{equation}\label{BO:000-II}
\Phi_\pm (v)=
\int_{B_1} (\partial_v \chi_E)_\pm (x)\,dx.
\end{equation}

We observe that

\begin{lemma}\label{12345888LAJ}
Let~${\rm Per}\,(E,B_1)<+\infty$ and~$n\ge2$.
Then the functions~$\Phi_\pm$ are continuous on~$S^{n-1}$.
Moreover, there exists~$v_\star$ such that
\begin{equation}\label{vstar}
\Phi_+(v_\star)=\Phi_-(v_\star).\end{equation}\end{lemma}

\begin{proof} Let $v$, $w\in S^{n-1}$.
By~\eqref{BO:000-II},
\begin{eqnarray*}
&& \big| \Phi_+ (v)-\Phi_+(w)\big| \le
\int_{B_1} \big| (\partial_v \chi_E)_+ (x)
-(\partial_w \chi_E)_+ (x)
\big|\,dx\\
&&\qquad \le 
\int_{B_1} \big| \partial_v \chi_E (x)
-\partial_w \chi_E (x)
\big|\,dx \le |v-w|\,\int_{B_1}|\nabla\chi_E(x)|\,dx
=|v-w|\,{\rm Per}\,(E,B_1).\end{eqnarray*}
This shows that~$\Phi_+$ is continuous. Similarly, one sees that~$\Phi_-$
is continuous.

Now we prove~\eqref{vstar}.
For this, let~$\Psi(v):=\Phi_+(v) -\Phi_-(v)$.
By~\eqref{BO:000-II},
\begin{equation*}
\Phi_\pm (-v)=\Phi_\mp(v). 
\end{equation*}
Therefore
\begin{equation}\label{iao}
\Psi(-v)= \Phi_+(-v) -\Phi_-(-v)
=\Phi_-(v) -\Phi_+(v) = -\Psi(v).\end{equation}
Now, if~$\Psi(e_1)=0$, we can take~$v_\star:=e_1$
and~\eqref{vstar} is proved. So we can assume that~$\Psi(e_1)>0$
(the case~$\Psi(e_1)<0$ is analogous). By~\eqref{iao}, we obtain that~$\Psi(-e_1)<0$.
Hence, since~$\Psi$ is continuous, it must have a zero on any path
joining~$e_1$ to~$-e_1$, and this proves~\eqref{vstar}.
\end{proof}

A control on the function~$\Phi_\pm$ implies a quantitative flatness
bound on the set~$E$, as stated here below:

\begin{lemma}\label{GHA98} Let~$n=2$.
There exists~$\mu_o>0$ such that for any~$\mu\in(0,\mu_o]$
the following statement holds.

Assume that
\begin{equation}\label{BO:001}
\Phi_-(e_2)\le \mu
\end{equation}
and that
\begin{equation}\label{BO:002}
\max\{ \Phi_+(e_1),\;\Phi_-(e_1)\}\le\mu 
.\end{equation}
Then, there exists a horizontal halfplane~$h\subset\R^2$ such that
\begin{equation}\label{JHAU}
\big| (E\setminus h)\cap B_1\big|+\big|(h\setminus E)\cap B_1\big|\le C\mu,\end{equation}
for some~$C>0$.
\end{lemma}

\begin{proof} 
Given~$v\in\partial B_1$,
we take into account the sets of $y\in v^\perp$
which give a positive contribution to~$I_{v,\pm}(y)$. For this,
we define
$$ {\mathcal{B}}_\pm(v):=
\{ y\in v^\perp {\mbox{ s.t. }} I_{v,\pm}(y)\ne0\}.$$
{F}rom~\eqref{JK:003}, we know that if~$y\in {\mathcal{B}}_\pm(v)$, then
$I_{v,\pm}(y)\ge1$. As a consequence
of this and of~\eqref{BO:000}, we have that
$$ \Phi_\pm (v)\ge \int_{ {\mathcal{B}}_\pm (v) } I_{v,\pm}(y)
\,d{\mathcal{H}}^{1}(y)
\ge {\mathcal{H}}^{1} ( {\mathcal{B}}_\pm(v) ).$$
Accordingly, by~\eqref{BO:001}
and~\eqref{BO:002}, we see that
\begin{equation}\label{BO:003A}
{\mathcal{H}}^{1} ( {\mathcal{B}}_-(e_2) )\le
\mu 
\end{equation}
and
\begin{equation}\label{BO:003}
{\mathcal{H}}^{1} ( {\mathcal{B}}_\pm(e_1) )\le
\mu 
.\end{equation}
Furthermore, for any~$y\in v^\perp\setminus{\mathcal{B}}_+(v)$
(resp.~$y\in v^\perp\setminus{\mathcal{B}}_-(v)$),
we have that~$I_{v,+}(y)=0$ (resp.,~$I_{v,-}(y)=0)$ and thus, by
Lemma~\ref{NONE},
the map~$B_1\cap( {y+\R v} )\ni x\mapsto
\chi_E(x)$ is nonincreasing (resp., nondecreasing).

Therefore,
by~\eqref{BO:003},
we have that for any vertical coordinate~$y\in e_1^\perp$
outside the small set~${\mathcal{B}}_-(e_1)\cup {\mathcal{B}}_+(e_1)$
(which has total length of size~$2\mu$),
the vertical line~$y+\R e_1$ is either all contained in~$E$
or in its complement (see Figure~\ref{INI}).

\begin{figure}[ht]
\includegraphics[width=10.4cm]{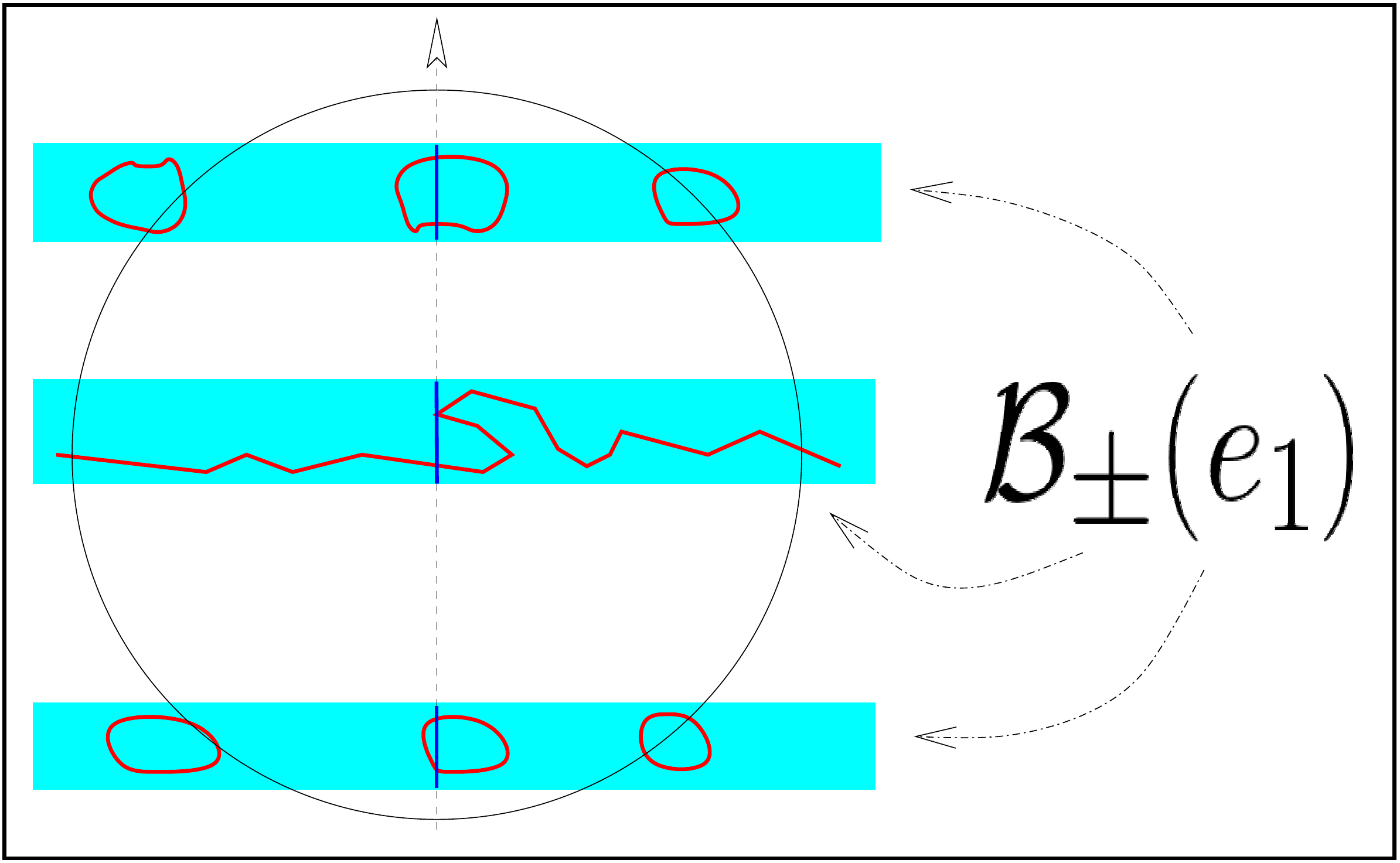}
\caption{\sl Horizontal lines do not meet the boundary of~$E$,
with the exception of a small set~${\mathcal{B}}_\pm(e_1)$.}
\label{INI}
\end{figure}

That is, we can denote by~${\mathcal{G}}_E$
the set of vertical coordinates~$y$
for which the portion in~$B_1$
of the horizontal line passing through~$y$ lies in~$E$
and, similarly, by~${\mathcal{G}}_{E^c}$
the set of vertical coordinates~$y$
for which the portion in~$B_1$
of the horizontal line passing through~$y$ lies in~$E^c$
and we obtain that~${\mathcal{G}}_E\cup {\mathcal{G}}_{E^c}$
exhaust the whole of~$(-1,1)$, up to a set of size at most~$2\mu$.

We also remark that~${\mathcal{G}}_{E}$ lies below~${\mathcal{G}}_{E^c}$:
indeed, by~\eqref{BO:003A},
we have that vertical lines can only exit the set~$E$
(possibly with the exception of a small set of size~$\mu$).
The situation is depicted in Figure~\ref{INI2}.

\begin{figure}[ht]
\includegraphics[width=10.4cm]{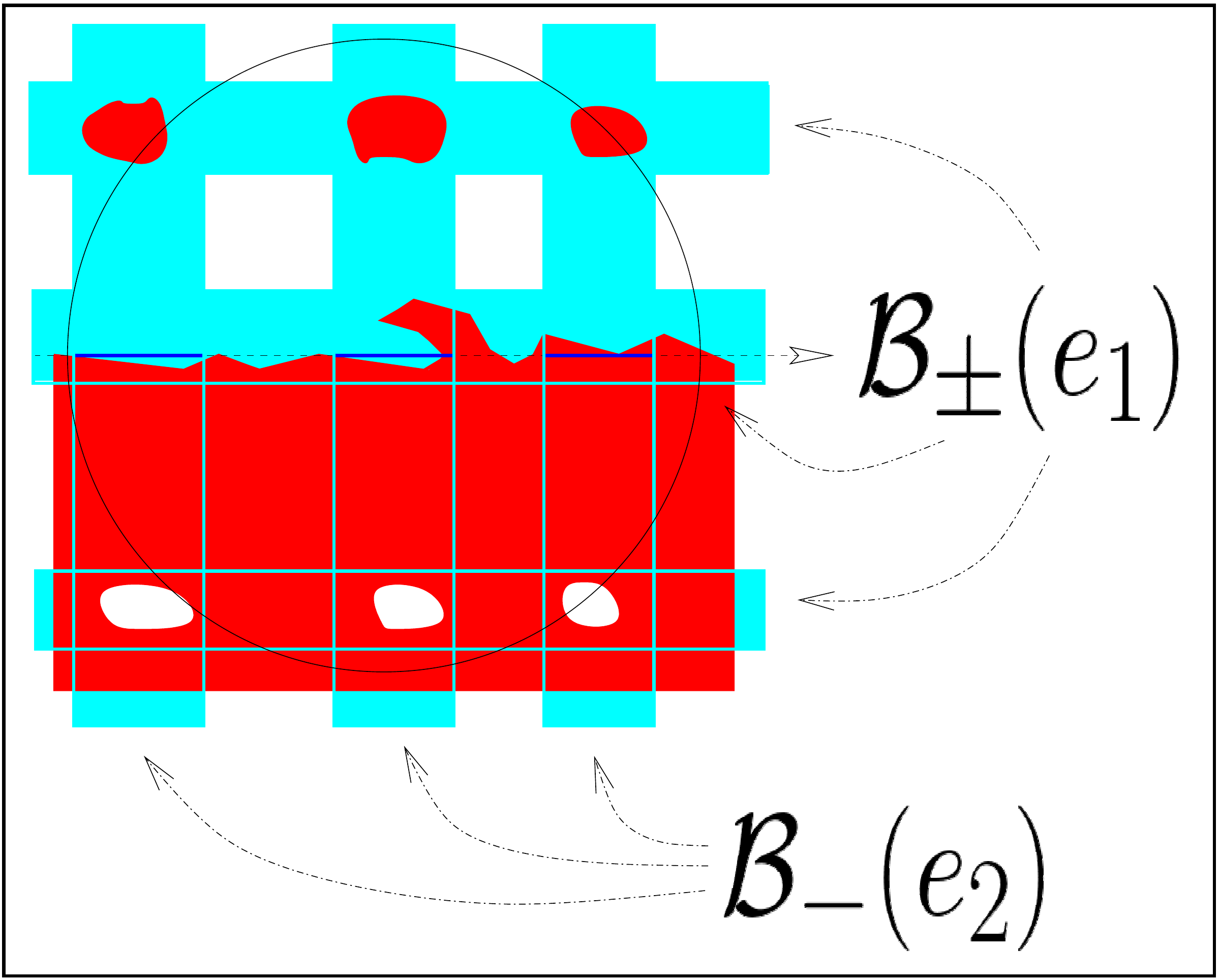}
\caption{\sl Vertical lines do not meet the boundary of~$E$,
with the exception of a small set~${\mathcal{B}}_-(e_2)$.}
\label{INI2}
\end{figure}

Hence, if we take~$h$ to be a horizontal 
halfplane which separates~${\mathcal{G}}_{E}$
and~${\mathcal{G}}_{E^c}$,
we obtain~\eqref{JHAU}.
\end{proof}

With this, we can now complete the proof of Theorem \ref{1123THJ}.
The main tool for this goal is Lemma~\ref{GHA98}.
In order to apply it, we need to check that~\eqref{BO:001}
and~\eqref{BO:002} are satisfied.
To this end,
we argue as follows. First of all, fixed a large~$R>2$,
we consider, as in Section~\ref{sec:primo},
a diffeomorphism~$\Phi^t$
such that~$\Phi^t(x)=x$
for any~$x\in\R^n\setminus B_{9R/10}$,
and
$\Phi^t(x)=x+tv$ for any~$x\in B_{3R/4}$, and we set~$E_t:=
\Phi^t(E)$. {F}rom~\eqref{O9056} (recall that here~$n=2$), we have that
\begin{equation*}
\min\left\{
\int_{B_{R/2}}
\big( \chi_{E}(x+tv)-\chi_E(x)\big)_+\,dx, \;
\int_{B_{R/2}}
\big( \chi_{E}(x+tv)-\chi_E(x)\big)_-\,dx
\right\}\le \frac{Ct}{R^s},
\end{equation*}
for some~$C>0$. Thus, dividing by~$t$ and sending~$t\searrow0$,
$$ \min\left\{
\int_{B_{R/2}}
\big(\partial_v \chi_{E}(x)\big)_+\,dx, \;
\int_{B_{R/2}}
\big( \partial_v\chi_{E}(x)\big)_-\,dx
\right\}\le \frac{C}{R^s}.$$
That is, recalling~\eqref{BO:000-II},
\begin{equation} \label{IO:co1}
\min\left\{ \Phi_+(v),\;\Phi_-(v)
\right\}\le \frac{C}{R^s}.\end{equation}
We also observe that~$E$ has finite perimeter in~$B_1$,
thanks to Theorem~\ref{B V T}, and so we can make use of
Lemma~\ref{12345888LAJ}. In particular, by~\eqref{vstar},
after a rotation of coordinates, we may assume that~$
\Phi_+(e_1)=\Phi_-(e_1)$. Hence~\eqref{IO:co1} says that
\begin{equation} \label{IO:co1:X67X}
\max\left\{ \Phi_+(e_1),\;\Phi_-(e_1)
\right\}
=\min\left\{ \Phi_+(e_1),\;\Phi_-(e_1)
\right\}\le \frac{C}{R^s}.\end{equation}
Also,
up to a change of orientation, we may suppose that~$\Phi_-(e_2)\le
\Phi_+(e_2)$, hence in this case~\eqref{IO:co1} says that
\begin{equation*} 
\Phi_-(e_2)\le \frac{C}{R^s}.\end{equation*}
{F}rom this and~\eqref{IO:co1:X67X},
we see that~\eqref{BO:001}
and~\eqref{BO:002} are satisfied (with~$\mu=C/R^s$)
and so by Lemma~\ref{GHA98} we conclude that
$$ \big| (E\setminus h)\cap B_1\big|+\big|(h\setminus E)\cap B_1\big|\le \frac{C}{R^s},$$
for some halfplane~$h$.
This completes the proof of Theorem~\ref{1123THJ}:
as a matter of fact, the result proven 
is even stronger, since it says that, after removing
horizontal and vertical slabs of size~$C/R^s$,
we have that~$\partial E$ in~$B_1$
is a graph of oscillation bounded by~$C/R^s$,
see Figure~\ref{INI2}
(in fact, more general statements and proofs can be find in~\cite{joaq}).

\section{Sketch of the proof of Theorem~\ref{PER}}\label{PER:S}

The core of the proof of Theorem~\ref{PER}
consists in constructing a suitable barrier that can be slided
``from below'' and which exhibits the desired stickiness
phenomenon: if this is possible, since the $s$-minimal surface
cannot touch the barrier, it has to stay above the barrier
and stick at the boundary as well.

So, the barrier we are looking for should have
negative fractional mean curvature,
coincide with~$F$ outside~$(-1,1)\times\R$
and contain~$(-1,1)\times(-\infty,\delta^\gamma)$.

Such barrier is constructed in~\cite{stick}
in an iterative way, that we now try to describe.\medskip

\noindent{\it Step 1.} Let us start by looking at
the subgraph of the function~$y=\frac{x_+}{\ell}$, given~$\ell\ge0$.
Then,
at all the boundary points~$X=(x,y)$ with
positive abscissa~$x>0$, the fractional mean curvature is
at most
\begin{equation}\label{OA:PAI9}
-\frac{c}{\max\{1,\ell\}\,|X|^{2s}},\end{equation}
for some~$c>0$.
The full computation is given in Lemma~5.1
of~\cite{stick}, but we can give a heuristic justification of
it, by saying that for small~$X$ the boundary point gets close
to the origin, where there is a corner and the curvature blows up
(with a negative sign, since there is ``more than a hyperplane''
contained in the set), see Figure~\ref{BARR1}. Also, the power~$2s$
in~\eqref{OA:PAI9} follows by scaling.

\begin{figure}[ht]
\includegraphics[width=9.4cm]{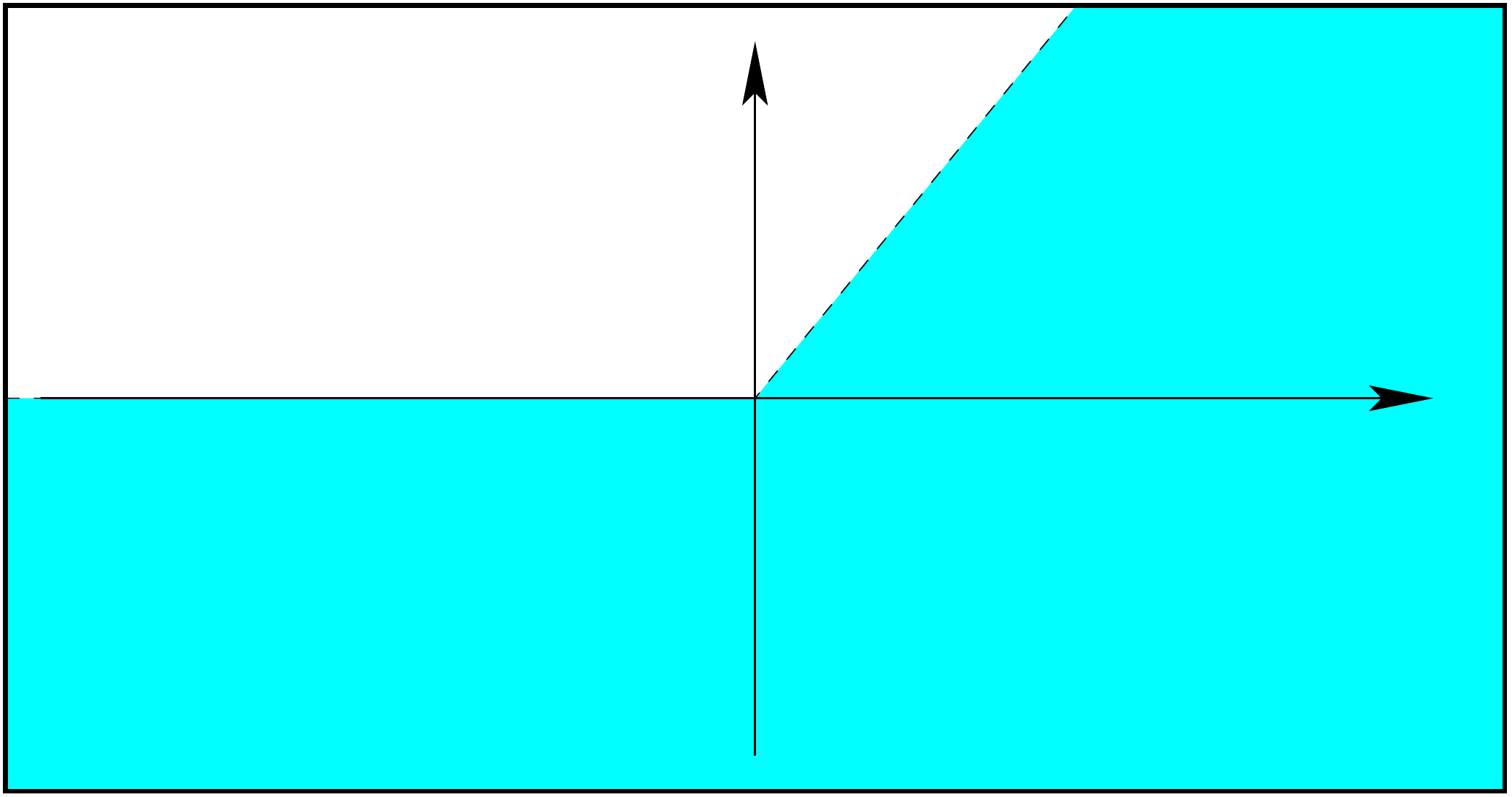}
\caption{\sl Description of Step 1.}
\label{BARR1}
\end{figure}

In addition, if~$\ell$ is close to~$0$, this first barrier
is close to a ninety degree angle, 
while if~$\ell$ is large it is close to a flat line, and these considerations
are also in agreement with~\eqref{OA:PAI9}.
\medskip

\noindent{\it Step 2.} Having understood in Step~1 what happens
for the ``angles'', now we would like to ``shift iteratively 
in a smooth way from one slope
to another'', see Figure~\ref{BARR2}.

\begin{figure}[ht]
\includegraphics[width=9.4cm]{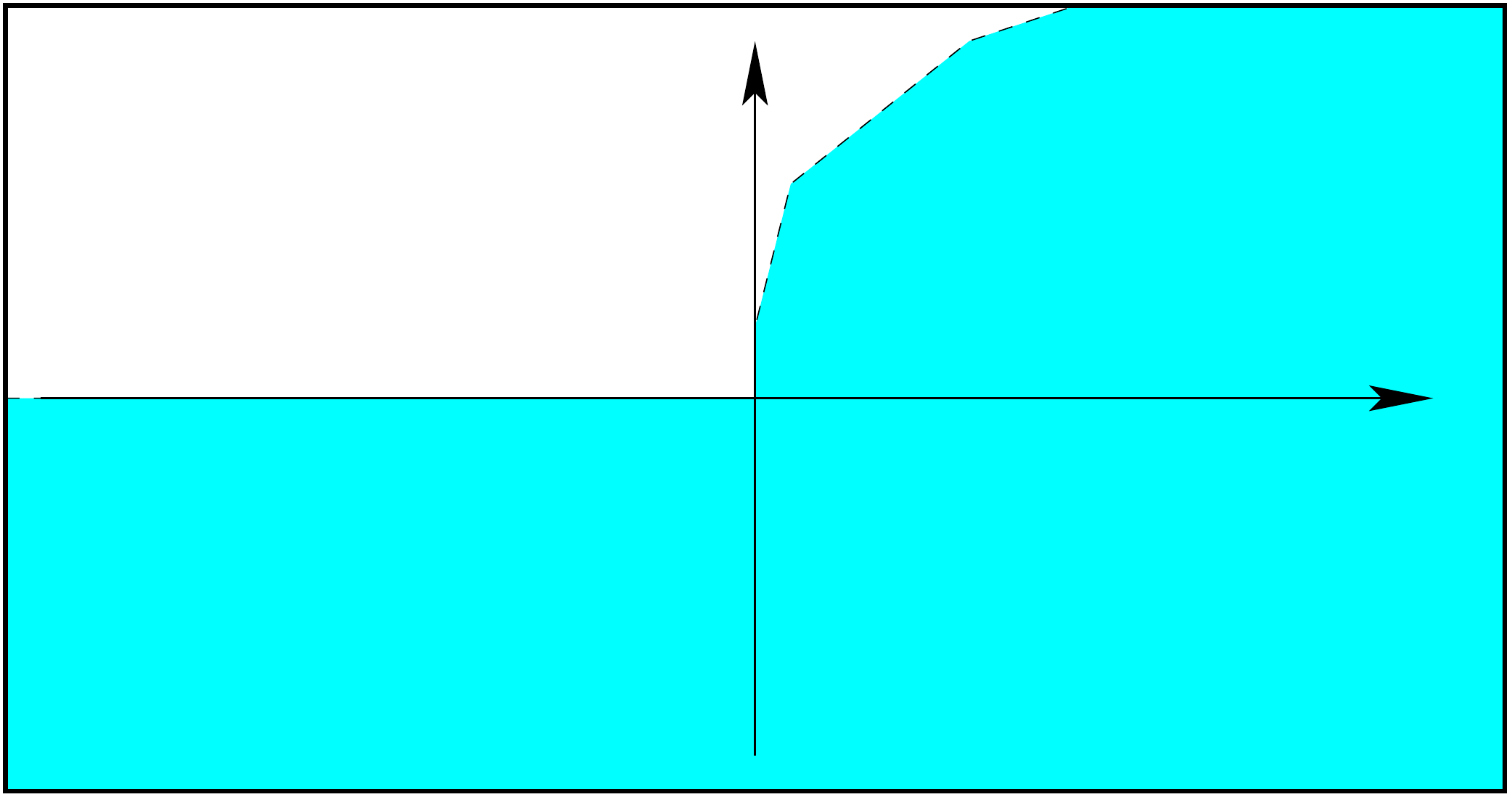}
\caption{\sl Description of Step 2.}
\label{BARR2}
\end{figure}

The detailed statement is given in Proposition~5.3
in~\cite{stick}, but the idea is as follows.
For any~$K\in\N$, $K\ge1$, one looks at the subgraph of
a nonnegative function~$v_K$ such that
\begin{itemize}
\item $v_K(x)=0$ if~$x<0$,
\item $v_K(x)\ge a_K$ if~$x>0$, for some~$a_K>0$,
\item $v_K(x)=\frac{x+q_K}{\ell_K}$ for any~$x\ge\ell_K-q_K$, for some~$\ell_K\ge K$ and~$q_K\in
\left[0,\frac1K\right]$,
\item at all the boundary points~$X=(x,y)$ with
positive abscissa~$x>0$, the fractional mean curvature is
at most~$-\frac{c}{\ell_K\,|X|^{2s}}$, for some~$c>0$.
\end{itemize}
\medskip

\noindent{\it Step 3.} If~$K$ is sufficiently large in Step~2,
the final slope is almost horizontal. In this case, one
can smoothly glue such barrier with a power like function like~$x^{
\frac12+s+\epsilon_0}$. Here, $\epsilon_0$ is any fixed
positive exponent (the power~$\gamma$ in the statement of
Theorem~\ref{PER} is related to~$\epsilon_0$, since~$\gamma:=\frac{
2+\epsilon_0}{1-2s}$). The details of the barrier constructed
in this way are given in Proposition~6.3
of~\cite{stick}.
In this case, one can still control the fractional
mean curvature
at all the boundary points~$X=(x,y)$ with
positive abscissa~$x>0$, but the estimate is of the type either~$|X|^{-2s}$,
for small~$|X|$, or~$|X|^{-\frac12-s+\epsilon_0}$, for large~$|X|$.
A sketch of such barrier is given in
Figure~\ref{BARR3}.

\begin{figure}[ht]
\includegraphics[width=9.4cm]{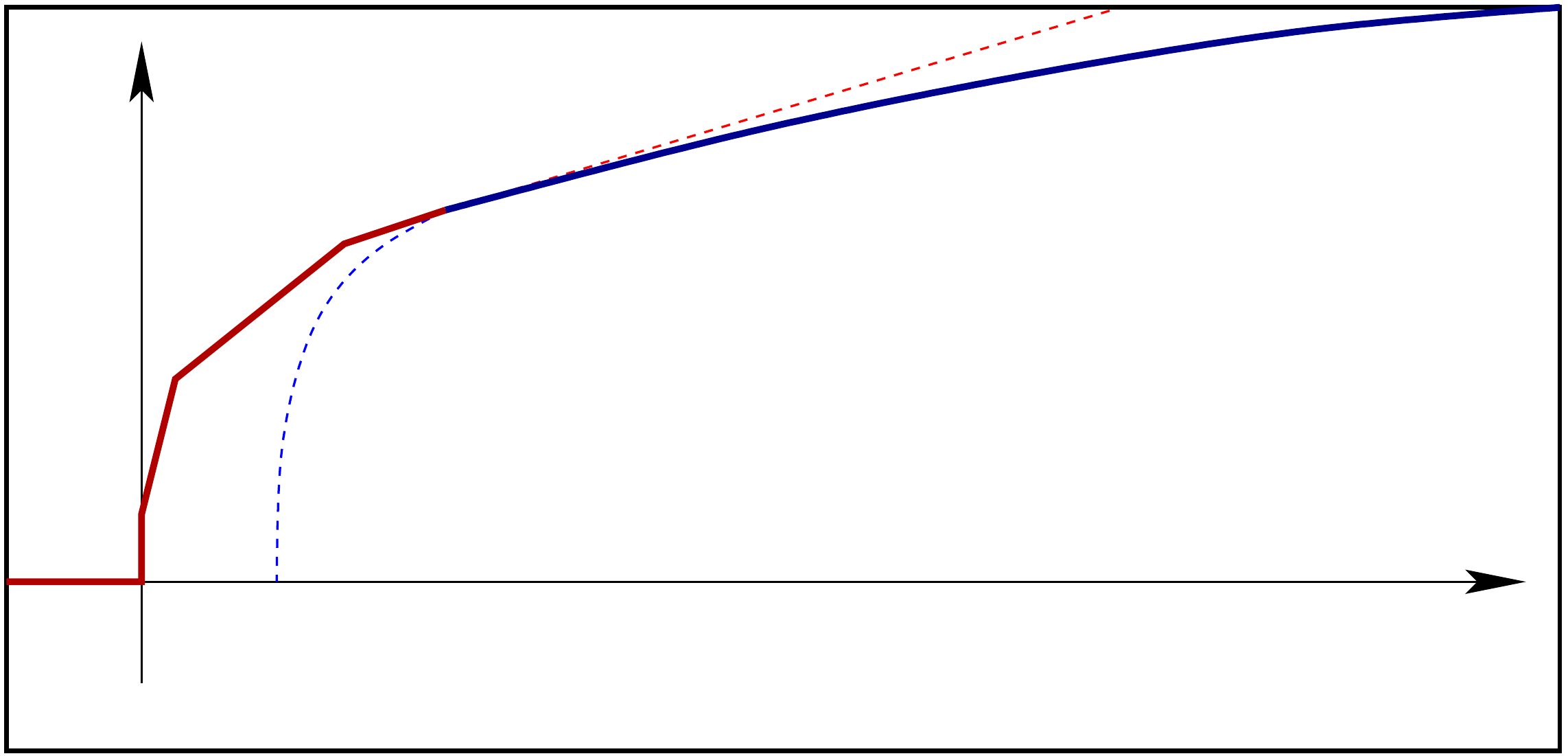}
\caption{\sl Description of Step 3.}
\label{BARR3}
\end{figure}

\medskip

\noindent{\it Step 4.} Now we use the barrier of Step~3
to construct a compactly supported object. The idea is to take such
barrier, to reflect it and to glue it
at a ``horizontal level'', see
Figure~\ref{BARR4}.

\begin{figure}[ht]
\includegraphics[height=6.8cm]{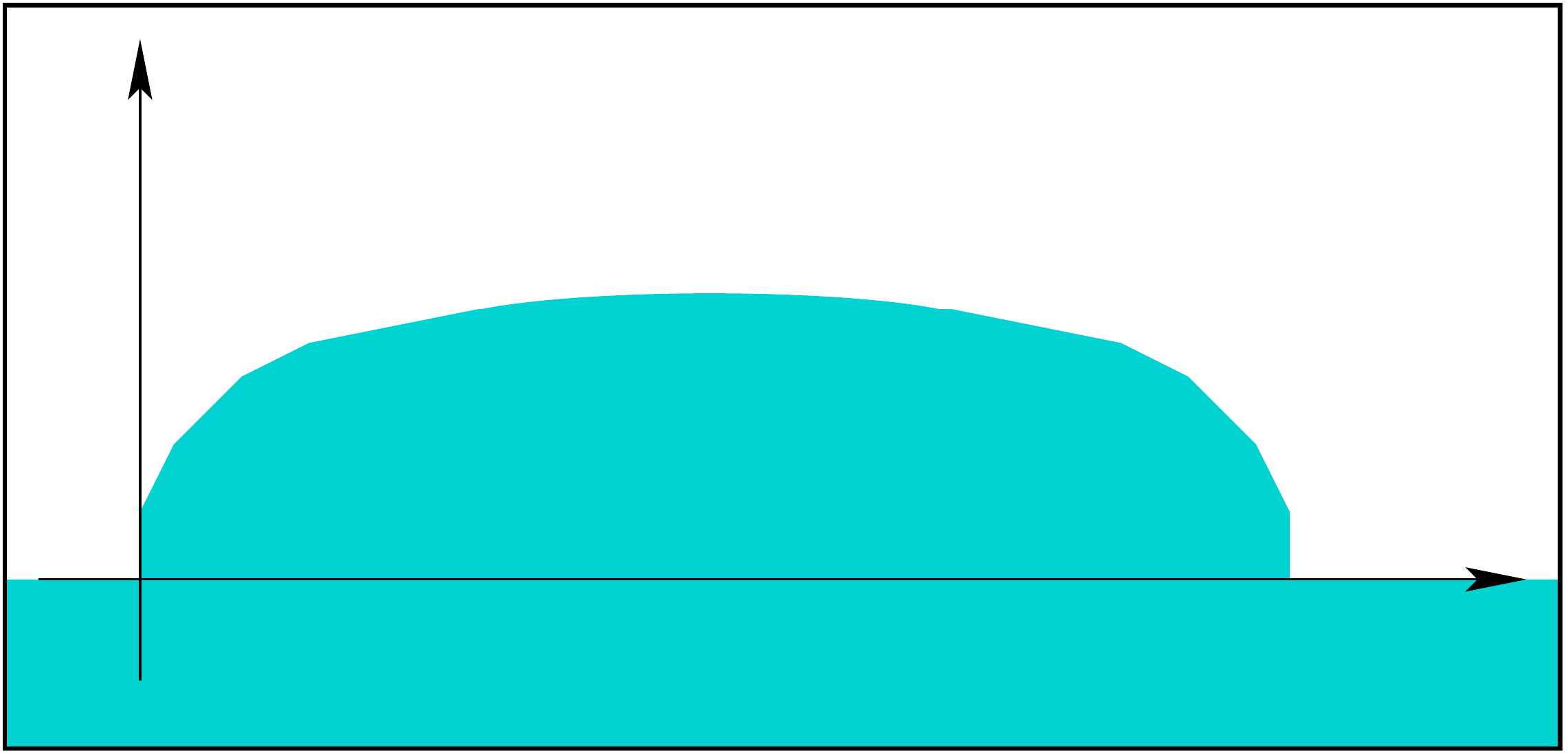}
\caption{\sl Description of Step 4.}
\label{BARR4}
\end{figure}

We remark that such barrier 
has a vertical portion at the origin and one can control its fractional
mean curvature from above with a negative quantity for
the boundary points~$X=(x,y)$ with
positive, but not too large, abscissa.

Of course, this type of estimate cannot hold at the maximal point
of the barrier, where ``more than a hyperplane'' is contained
in the complement of the set, and therefore the fractional mean curvature
is positive (the precise quantitative estimate
is given in
Proposition~7.1. of~\cite{stick}).\medskip

\noindent{\it Step 5.}
Nevertheless, we can now compensate this error in the fractional
mean curvature near the maximal point
of the barrier by adding two suitably large domains on the sides
of the barriers, see Figure~\ref{BARR5}.

\begin{figure}[ht]
\includegraphics[height=6.8cm]{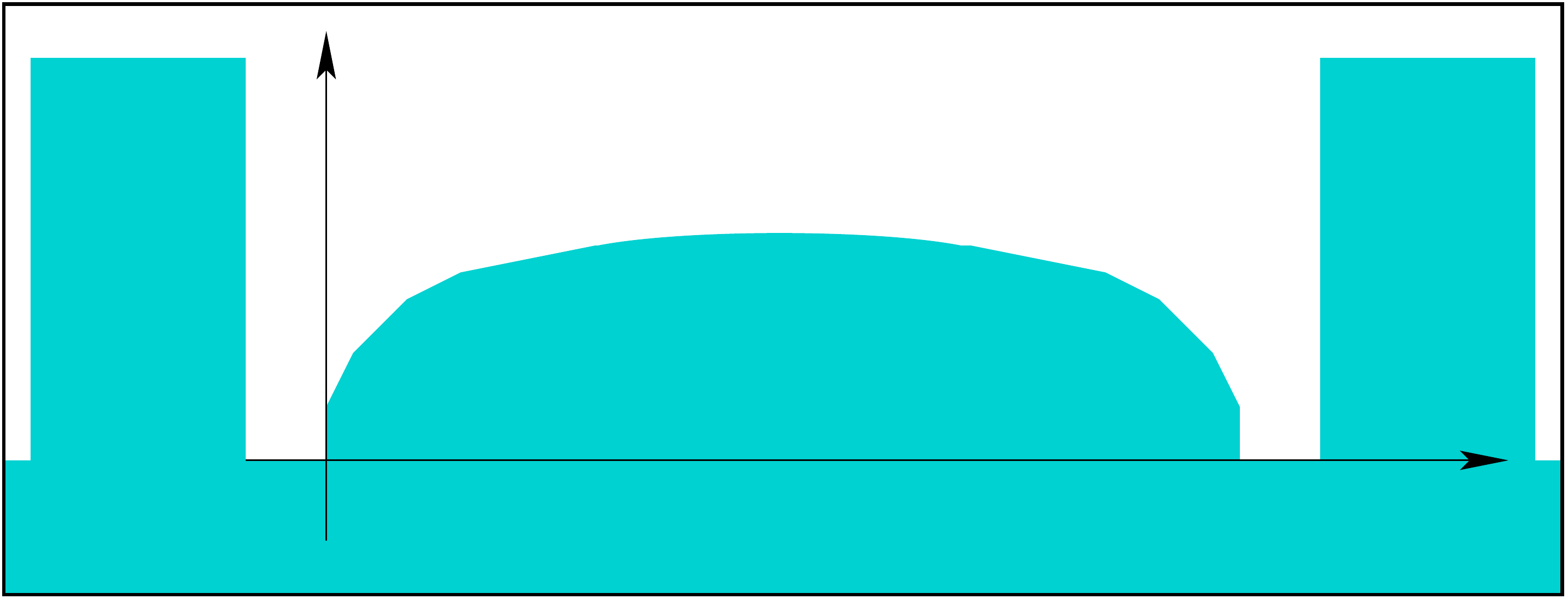}
\caption{\sl Description of Step 5.}
\label{BARR5}
\end{figure}

The barrier constructed in this way is described in details
in Proposition~7.3 of~\cite{stick} and its basic feature is to
possess a vertical portion near the origin and to possess negative
fractional mean curvature. 

By keeping good track of the quantitative estimates on the bumps
of the barriers and on their fractional mean curvatures,
one can now scale the latter barrier and slide it from below,
in order to prove
Theorem~\ref{PER}. The full details are given in Section~8
of~\cite{stick}.

\begin{appendix}

\section{A sketchy discussion on the asymptotics of the $s$-perimeter}\label{APP12}

In this appendix, we would like to emphasize the fact that,
as $s\nearrow1/2$, the $s$-perimeter recovers
(under different perspectives) the classical perimeter, while,
as $s\searrow0$, the nonlocal features become predominant and
the problem produces
the Lebesgue measure ???- or, better to say, convex combinations of Lebesgue
measures by interpolation parameters of nonlocal type.

First of all,
we show that if $E$ is a bounded set with smooth boundary, then
\begin{equation}\label{TccP01}
\lim_{s\nearrow1/2} (1-2s)\,{\rm Per}_s\,(E,\R^n)=
\kappa_{n-1}\,{\rm Per}\,(E,\R^n)
,\end{equation}
where we denoted by $\kappa_n$ 
the $n$-dimensional volume of the $n$-dimensional
unit ball.

For further convenience, we also use the notation $$ \varpi_n := 
{\mathcal{H}}^{n-1}(S^{n-1}).$$ 
Notice that, by polar coordinates, 
\begin{equation}\label{GJB1OUJLsssaAA}
\kappa_n = \int_{S^{n-1}}\left[\int_0^1 \rho^{n-1}\,d\rho\right]\,d{\mathcal{H}}^{n-1}(x)=\frac{\varpi_n}{n}.
\end{equation}
We point out that formula \eqref{TccP01}
is indeed a simple version of more general
approximation
results, for which we refer to \cite{MR1945278, MR1942130,
MR2765717, MR2033060, MR2782803} and to \cite{CV} for the regularity
results that can be achieved by approximation methods.
See also~\cite{MR3506705} for further comments and
examples.
\medskip

The proof of \eqref{TccP01} can be performed by different methods;
here we give a simple argument which uses formula \eqref{8uj56789:A}.
To this aim, we fix $x\in\partial E$ and $\delta>0$.
If $y\in (\partial E)\cap B_\delta(x)$
and $\delta$ is sufficiently small, then $\nu(y)=\nu(x)+O(\delta)$.
Moreover, for any $\varrho\in(0,\delta]$,
the $(n-2)$-dimensional contribution of $\partial E$
in $\partial B_\varrho(x)$ coincides,
up to higher orders in $\delta$, with the one
of the $(n-2)$-dimensional sphere, that is
$\varpi_{n-1}\,\varrho^{n-2}$,
see Figure \ref{de45678gsagdu}.

\begin{figure}[ht]
\includegraphics[width=9.4cm]{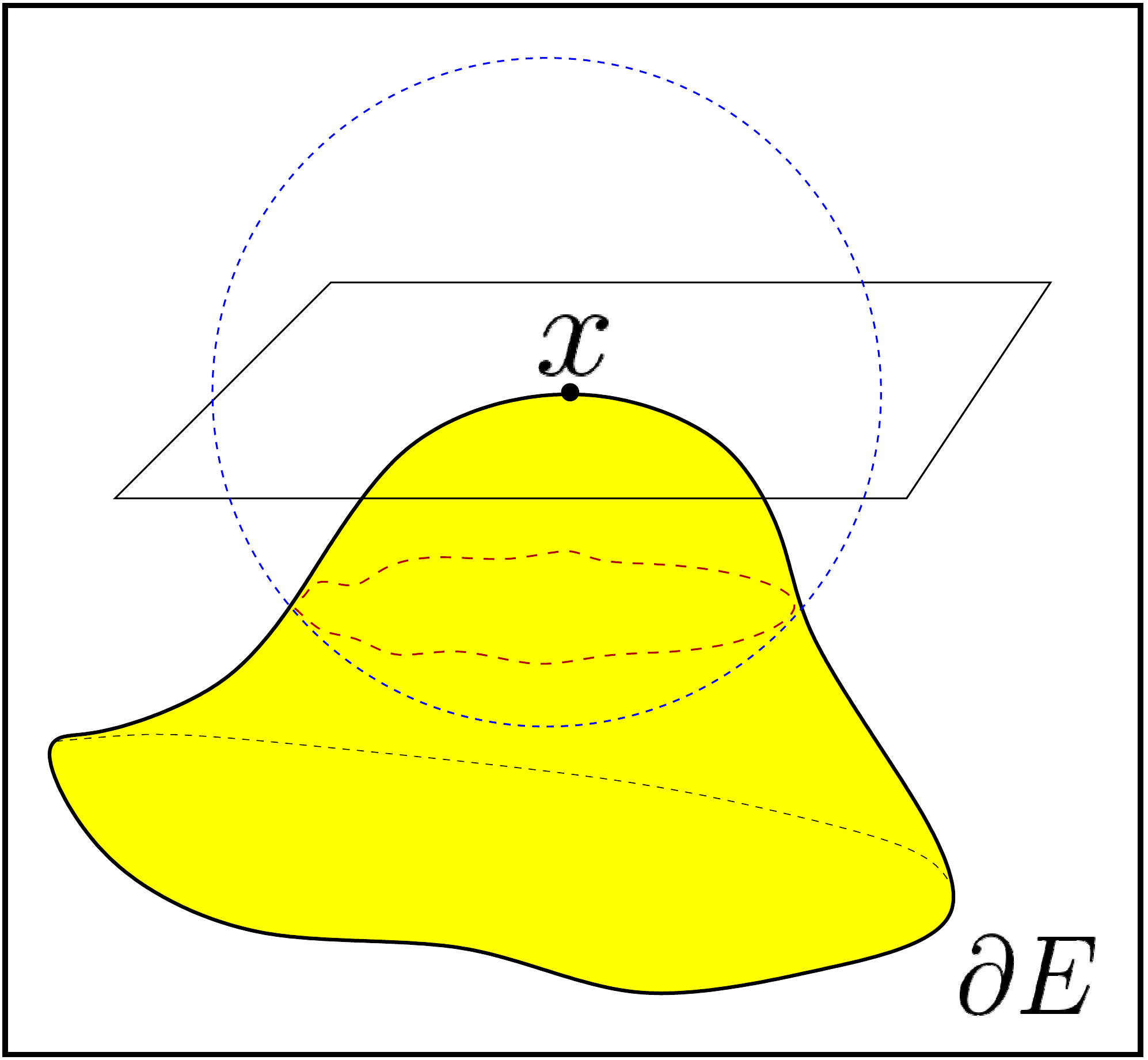}
\caption{\sl ${\mathcal{H}}^{n-2}\big( (\partial E)\cap \partial B_\varrho(x) \big)$
(in the picture, $n=3$).}
\label{de45678gsagdu}
\end{figure}

As a consequence of these observations, we have that
\begin{eqnarray*} 
&&\int_{(\partial E)\cap B_\delta(x)} \frac{ \nu(x)\cdot\nu(y) }{
|x-y|^{n+2s-2}}\,d{\mathcal{H}}^{n-1}(y)
=
\int_{(\partial E)\cap B_\delta(x)} \frac{ 1+O(\delta) }{
|x-y|^{n+2s-2}}\,d{\mathcal{H}}^{n-1}(y) \\
&&\qquad= \big( 1+O(\delta) \big)\,
\int_0^\delta \frac{
{\mathcal{H}}^{n-2} \big( (\partial E)\cap (\partial B_\rho)\big) 
}{\varrho^{n+2s-2}}
\,d\varrho
\\&&\qquad= 
\big( 1+O(\delta) \big)\,\varpi_{n-1}\,
\int_0^\delta \frac{ \varrho^{n-2}
}{\varrho^{n+2s-2}}
\,d\varrho\\
&&\qquad= \frac{ \big( 1+O(\delta) \big)\,\varpi_{n-1}\, \delta^{1-2s} }{
1-2s}.\end{eqnarray*}
On the other hand,
\begin{eqnarray*}
\int_{(\partial E)\setminus B_\delta(x)} \frac{ \nu(x)\cdot\nu(y) }{
|x-y|^{n+2s-2}}\,d{\mathcal{H}}^{n-1}(y)\le
\frac{ {\mathcal{H}}^{n-1} (\partial E)}{\delta^{n+2s-2}}.
\end{eqnarray*}
Therefore
$$ \int_{\partial E} \frac{ \nu(x)\cdot\nu(y) }{
|x-y|^{n+2s-2}}\,d{\mathcal{H}}^{n-1}(y) =
\frac{ \big( 1+O(\delta) \big)\,\varpi_{n-1}\, \delta^{1-2s} }{
1-2s}+O(\delta^{-n-2s+2}).$$
Accordingly, recalling \eqref{8uj56789:A},
\begin{eqnarray*}
&&\lim_{s\nearrow1/2} (1-2s)\,{\rm Per}_s\,(E,\R^n)\\
&&\qquad=
\lim_{s\nearrow1/2} \frac{1-2s}{2s\,(n+2s-2)}\,
\int_{\partial E}\,\left[
\int_{\partial E} \frac{ \nu(x)\cdot\nu(y) }{
|x-y|^{n+2s-2}}\,d{\mathcal{H}}^{n-1}(y)\,\right]
\,d{\mathcal{H}}^{n-1}(x)\\
&&\qquad=
\lim_{s\nearrow1/2} \frac{1-2s}{n-1}\,
\int_{\partial E}\,\left[
\frac{ \big( 1+O(\delta) \big)\,\varpi_{n-1}\, \delta^{1-2s} }{
1-2s}+O(\delta^{-n-2s+2})
\right]\,d{\mathcal{H}}^{n-1}(x) \\
&&\qquad
= \lim_{s\nearrow1/2} \frac{\big( 1+O(\delta) \big)\,\varpi_{n-1}\, \delta^{1-2s}
+(1-2s)\,O(\delta^{-n-2s+2})}{n-1}\,{\mathcal{H}}^{n-1}({\partial E})\\
&&\qquad=
\frac{\big( 1+O(\delta) \big)\,\varpi_{n-1}}{n-1}
\,{\mathcal{H}}^{n-1}({\partial E}).
\end{eqnarray*}
Hence, by taking $\delta$ arbitrarily small,
$$ \lim_{s\nearrow1/2} (1-2s)\,{\rm Per}_s\,(E,\R^n)=
\frac{\varpi_{n-1}}{n-1}
\,{\mathcal{H}}^{n-1}({\partial E}),$$
which gives \eqref{TccP01}, in view of \eqref{GJB1OUJLsssaAA}.\medskip

Now we show that, if
$n\ge3$ and $E$ is a bounded set with smooth boundary,
\begin{equation}\label{TccP02}
\lim_{s\searrow0} s\,{\rm Per}_s\,(E,\R^n)=\frac{\varpi_n}{2}
\,|E|.\end{equation}
Once again, more general (and subtle) statements hold true, see
\cite{MR1940355, MR3007726} for details.
\medskip

To prove \eqref{TccP02}, we denote by
$$ \Gamma(x):= \frac1{(n-2)\,\varpi_n\,|x|^{n-2}}$$
the fundamental solution\footnote{It is interesting to understand how the fundamental solution of the Laplacian also 
occurs when $n=2$. In this case, we observe that if $c_E:=\int_{\partial E}\nu(y)\,
d{\mathcal{H}}^{n-1}(y)$, then of course 
$$ \iint_{(\partial E)\times (\partial E)} \nu(x)\cdot\nu(y)\,d{\mathcal{H}}^{n-1}(x)
\,d{\mathcal{H}}^{n-1}(y) =\int_{\partial E} \nu(x)\cdot c_E \,d{\mathcal{H}}^{n-1}(x)
=\int_E {\rm div}_x c_E\,dx =\int_E 0\,dx=0.$$
Hence, we write
$$ \frac{1}{|x-y|^{2s}}= \exp\left( -2s\log |x-y|\right)=
1-2s\log |x-y|+O(s^2),$$
thus
$$ \frac{1}{2s}\iint_{(\partial E)\times (\partial E)}
\frac{\nu(x)\cdot\nu(y)}{|x-y|^{2s}}\,d{\mathcal{H}}^{n-1}(x)
\,d{\mathcal{H}}^{n-1}(y)
= - \iint_{(\partial E)\times (\partial E)}
\nu(x)\cdot\nu(y)\,\log |x-y|\,d{\mathcal{H}}^{n-1}(x)\,d{\mathcal{H}}^{n-1}(y)+O(s)$$
and one can use the same fundamental solution trick
as in the case $n\ge3$. }
of the Laplace operator when $n\ge3$, that is
$$ -\Delta \Gamma (x)=\delta_0(x),$$
where $\delta_0$ is the Dirac's Delta centered at
the origin. Then, from \eqref{8uj56789:A},
\begin{eqnarray*}
&& \lim_{s\searrow0} s\,{\rm Per}_s\,(E,\R^n)=
\frac{1}{2\,(n-2)}\,
\int_{\partial E}\,\left[
\int_{\partial E} \frac{ \nu(x)\cdot\nu(y) }{
|x-y|^{n-2}}\,d{\mathcal{H}}^{n-1}(y)
\right]\,d{\mathcal{H}}^{n-1}(x) \\
&&\qquad =\frac{ \varpi_n }{2}\,
\int_{\partial E}\,\left[
\int_{\partial E} 
\nu(y)\cdot \big( 
\nu(x)
\Gamma(x-y)\big)\,d{\mathcal{H}}^{n-1}(y)
\right]\,d{\mathcal{H}}^{n-1}(x)
\\&&\qquad=\frac{\varpi_n}{2}\,
\int_{\partial E}\,\left[
\int_{E} \,{\rm div}_y \big( 
\nu(x) \Gamma(x-y) \big)\,dy
\right]\,d{\mathcal{H}}^{n-1}(x)\\
&&\qquad=\frac{\varpi_n}{2}\,
\int_{\partial E}\,\left[
\int_{E}
\nu(x) \cdot\nabla_y\Gamma(x-y) \,dy
\right]\,d{\mathcal{H}}^{n-1}(x)\\
&&\qquad=\frac{\varpi_n}{2}\,
\int_{E}\,\left[
\int_{\partial E}                    
\nu(x) \cdot\nabla_y\Gamma(x-y) \,d{\mathcal{H}}^{n-1}(x)
\right]\,dy
\\&&\qquad=\frac{\varpi_n}{2}\,
\int_{E}\,\left[
\int_{E}
{\rm div}_x\,\big(\nabla_y\Gamma(x-y)\big) \,dx
\right]\,dy
\\&&\qquad= 
-\frac{\varpi_n}{2}\,
\iint_{E\times E}
\Delta \Gamma(x-y)\,dx\,dy
\\&&\qquad= 
\frac{\varpi_n}{2}\,
\iint_{E\times E}
\delta_0(x-y)\,dx\,dy \\&&\qquad=\frac{\varpi_n}{2}\,\int_{E}
1\,dy\\&&\qquad
=\frac{\varpi_n}{2}\,|E|,\end{eqnarray*}
that is \eqref{TccP02}.

We remark that formula \eqref{TccP02}
is actually a particular case of a more general phenomenon,
described in \cite{MR3007726}. For instance, if
the following limit exists
$$ a(E) :=\lim_{s\searrow0} \frac{2s}{\varpi_n} \,\int_{E\setminus B_1}
\frac{dx}{|x|^{n+2s}},$$
then
\begin{equation}
\label{TccP03}
\lim_{s\searrow0} \frac{2s}{\varpi_n}\,{\rm Per}_s\,(E,\Omega)=
(1-a(E))\,|E\cap\Omega|+a(E)\,|\Omega\setminus E|.\end{equation}
Notice indeed that \eqref{TccP02} is a particular case of
\eqref{TccP03}, since when $E$ is bounded, then $a(E)=0$.
Equation \eqref{TccP03} has also a suggestive interpretation,
since it says that, in a sense, as $s\searrow0$, the fractional
perimeter is a convex interpolation of measure contributions
inside the reference set $\Omega$: namely it weights the measures of
two contributions of $E$ and the complement of $E$ inside $\Omega$
by a convex parameter $a(E)\in[0,1]$ which in turn takes into account
the behavior of $E$ at infinity.

\section{A sketchy discussion on the asymptotics of the $s$-mean curvature}\label{APP1W}

As $s\nearrow1/2$, the $s$-mean curvature
recovers the classical mean curvature (see \cite{abatangelo} for details).

A very natural question raised to us by Jun-Cheng Wei dealt with the
asymptotics as $s\searrow0$ of the $s$-mean curvature.
Notice that, by \eqref{TccP02}, we know that $2s$ times the $s$-perimeter
approaches $\varpi_n$ times the volume. Since the variation of the volume along
normal deformations is $1$, if one is allowed to ``exchange the limits''
(i.e. to identify the limit of the variation with the variation of the limit),
then she or he may guess that $2s$ times the $s$-mean curvature should approach $\varpi_n$.

This is indeed the case, and higher orders can be computed as well,
according to the following observation: if $E$ has smooth boundary,
$p\in\partial E$ and $E\subseteq B_R(p)$ for some $R>0$, then
\begin{equation}\label{JCW}
2s\,H^s_E(p)={\varpi_n}+2s\left(
\int_{B_R(p)}\frac{\chi_{E^c}(x)-\chi_E(x)}{|x-p|^{n}}\,dx-
{\varpi_n}\,\log R
\right)+o(s),
\end{equation}
as $s\searrow0$.
To prove this, we first observe that, up to a translation, we can take $p=0$.
Moreover, since $E$ lies inside $B_R$,
\begin{equation}\label{G56:AKAlap}\begin{split}
& \int_{\R^n\setminus B_R} \frac{\chi_{E^c}(x)-\chi_E(x)}{|x|^{n+2s}}\,dx
=\int_{\R^n\setminus B_R} \frac{dx}{|x|^{n+2s}}
\\ &\qquad=\frac{ {\varpi_n} }{2s\,R^{2s}}
=\frac{ {\varpi_n} }{2s}\,\exp (-2s\,\log R)
\\ &\qquad=\frac{ {\varpi_n} }{2s}\,\big( 1-2s\,\log R+o(s)\big)
.\end{split}\end{equation}
In addition, since $\partial E$ is smooth, we have that (possibly after
a rotation) there exists $\delta_o\in\big(0,\,\min\{1,R\}\big)$
such that, for any $\delta\in(0,\delta_o]$,
$E\cap B_\delta$ contains $\{x_n\le -M|x'|^2\}$ and is
contained in $\{x_n \le M|x'|^2\}$. Here, $M>0$ only depends on the curvatures
of $E$ and we are using the notation~$x=(x',x_n)\in\R^{n-1}\times\R$
(notice also that since
we took~$p=0$, the ball~$B_\delta$ is actually centered at~$p$). 

Therefore,
we have that
$\chi_{E^c}(x)-\chi_{E}(x)=-1$
for any $x\in B_\delta\cap \{x_n\le -M|x'|^2\}$
and
$\chi_{E^c}(x)-\chi_{E}(x)=1$
for any $x\in B_\delta\cap \{x_n\ge M|x'|^2\}$.
In this way, a cancellation gives that
$$ \int_{B_\delta\cap \{|x_n| \ge M|x'|^2\}} \frac{
\chi_{E^c}(x)-\chi_{E}(x)
}{|x|^{n+2s}}\,dx=0.$$
As a consequence, for any $\sigma\in[0,s]$, if $s\in(0,\,1/4)$,
\begin{eqnarray*}
&& \left|\int_{B_\delta} \frac{\chi_{E^c}(x)-\chi_E(x)}{|x|^{n+2\sigma}}\,dx\right|
\le \int_{ \{|x'|\le\delta\}} \,dx'
\int_{ \{|x_n|\le M\,|x'|^2\}} \,dx_n\, \frac{1}{|x|^{n+2\sigma}}\\
\\ &&\qquad\le 2M\int_{ \{|x'|\le\delta\}} \frac{|x'|^2}{|x'|^{n+2\sigma}}\,dx'
\le \frac{ 2M\,{\varpi_n} \,\delta^{1-2\sigma} }{1-2\sigma}
\le 4M\,{\varpi_n} \,\delta^{1/2} .\end{eqnarray*}
Therefore, we use this inequality with $\sigma:=0$
and $\sigma:=s$ and the Dominated Convergence Theorem, to find that
\begin{eqnarray*}&& \lim_{s\searrow0}
\left|
\int_{B_R}\frac{\chi_{E^c}(x)-\chi_E(x)}{|x|^{n}}\,dx-
\int_{B_R}\frac{\chi_{E^c}(x)-\chi_E(x)}{|x|^{n+2s}}\,dx
\right| \\
&\le&
\lim_{s\searrow0}
\left|
\int_{B_R\setminus B_\delta}\frac{\chi_{E^c}(x)-\chi_E(x)}{|x|^{n}}\,dx-
\int_{B_R\setminus B_\delta}\frac{\chi_{E^c}(x)-\chi_E(x)}{|x|^{n+2s}}\,dx
\right| +8M\,{\varpi_n} \,\delta^{1/2}
\\ &=& 8M\,{\varpi_n} \,\delta^{1/2}.
\end{eqnarray*}
Hence, since we can now take $\delta$ arbitrarily small, we conclude that
$$ \lim_{s\searrow0}
\left|
\int_{B_R}\frac{\chi_{E^c}(x)-\chi_E(x)}{|x|^{n}}\,dx-
\int_{B_R}\frac{\chi_{E^c}(x)-\chi_E(x)}{|x|^{n+2s}}\,dx
\right|=0.$$
In view of this, and recalling \eqref{HS-DEF} and \eqref{G56:AKAlap}, we find that
\begin{eqnarray*}
&& \lim_{s\searrow0} \frac1{s}\,\left|
2s\,H^s_E(0)-{\varpi_n}-2s\left(
\int_{B_R}\frac{\chi_{E^c}(x)-\chi_E(x)}{|x|^{n}}\,dx-
{\varpi_n}\,\log R
\right) \right| \\
&\le& \lim_{s\searrow0} \frac1{s}\,\left|
2s\,
\int_{\R^n} \frac{\chi_{E^c}(x)-\chi_E(x)}{|x|^{n+2s}}\,dx
-{\varpi_n}-2s\left(
\int_{B_R}\frac{\chi_{E^c}(x)-\chi_E(x)}{|x|^{n+2s}}\,dx-
{\varpi_n}\,\log R
\right) \right| 
\\ &&\qquad+2\,\left|
\int_{B_R}\frac{\chi_{E^c}(x)-\chi_E(x)}{|x|^{n}}\,dx-
\int_{B_R}\frac{\chi_{E^c}(x)-\chi_E(x)}{|x|^{n+2s}}\,dx
\right|
\\
&=& \lim_{s\searrow0} \frac1{s}\,\Big|
{ {\varpi_n} }\,\big( 1-2s\,\log R+o(s)\big)
-{\varpi_n}
+2s{\varpi_n}\,\log R
\Big| 
\\&=&0.\end{eqnarray*}
This proves \eqref{JCW}.

\section{Second variation formulas and graphs of zero nonlocal mean curvature}
\label{9ojknAAsaw}

In this appendix, we show that the second variation (say, with respect to
a normal perturbation~$\eta$)
of the fractional perimeter of surfaces 
with vanishing mean curvature is given by \begin{equation*} -2 \int_{\partial E} 
\frac{ \eta(y) -\eta({{x}})}{
|{{x}}-y|^{n+2s}}\,d{\mathcal{H}}^{n-1}(y)
+
\int_{\partial E}
\frac{
\eta({{x}})\,\big[1-\nu({{x}})\cdot\nu(y)\big]}{
|{{x}}-y|^{n+2s}}\,d{\mathcal{H}}^{n-1}(y).\end{equation*}
A rigorous statement for this claim will be given in the forthcoming
Lemma~\ref{78yuOOP}: for the moment, we remark that
the expression
above
is related with
the Jacobi field along surfaces of vanishing nonlocal mean curvature.
We refer to \cite{del pino} for full details about this type of formulas.
See in particular formula (1.6) there, which gives the details of this formula,
Lemma A.2 there, which shows that, as $s\nearrow1/2$, the first integral
approaches the Laplace-Beltrami operator and Lemma A.4 there,
which shows that the latter integral produces, as $s\nearrow1/2$,
the norm squared of the second fundamental form, in agreement with the classical case.

Here, for simplicity, we reduce to the case in which $E$ is a graph
and
we consider a small normal deformation of its boundary,
plus an additional small translation, and we write
the resulting manifold as an appropriate normal deformation. The details go as follows:

\begin{lemma}\label{78yuOOP}
Let~$\Sigma\subset\R^n$ be a graph of class $C^2$, and let~$E$ be the corresponding
epigraph. Let~$\nu=(\nu_1,\dots,\nu_n)$
be the exterior normal of~$\Sigma=\partial E$.

Given~$\eps>0$ and~${\bar{x}}\in\Sigma$, we set
\begin{equation}\label{BA1}
\Sigma_\eps^* :=\{ x+\eps\eta(x)\,\nu(x)-\eps\eta({\bar{x}})\,\nu({\bar{x}})
,\; x\in\Sigma\}.\end{equation}
Then, if~$\eps$ is sufficiently small,~$\Sigma_\eps^*$ is a graph, with
epigraph a suitable~$E_\eps^*$, with~${\bar{x}}\in \partial E_\eps^*$,
and
\begin{eqnarray*}
\lim_{\eps\to0} \frac{1}{2\eps}
\big( H^s_E({\bar{x}}) -H^s_{E_\eps^*}({\bar{x}})\big)&=&
\int_{\Sigma}
\frac{
\eta(y)
-\eta({\bar{x}})\,\nu({\bar{x}})\cdot\nu(y)}{
|{\bar{x}}-y|^{n+2s}}\,d{\mathcal{H}}^{n-1}(y)
\\ &=&
\int_{\Sigma}
\frac{
\eta(y)
-\eta({\bar{x}})}{
|{\bar{x}}-y|^{n+2s}}\,d{\mathcal{H}}^{n-1}(y)
+\int_{\Sigma}
\frac{
\eta({\bar{x}})\,\big[1-\nu({\bar{x}})\cdot\nu(y)\big]}{
|{\bar{x}}-y|^{n+2s}}\,d{\mathcal{H}}^{n-1}(y)
.\end{eqnarray*}
\end{lemma}

\begin{proof}
We denote by~$\gamma:\R^{n-1}\to\R$ the graph of class $C^2$
that describes~$\Sigma$. In this way, we can write~$E=\{x_n <\gamma(x')\}$
and
$$ \nu(x) = \nu\big(x',\gamma(x')\big)=
\frac{\big( -\nabla \gamma(x'),1\big)}{\sqrt{ 1+|\nabla\gamma(x')|^2 }}.$$
We also write~$\kappa=(\kappa',\kappa_n):=\eta({\bar{x}})\,\nu({\bar{x}})$.
Then
\begin{eqnarray*}
\Sigma_\eps^* &=&\left\{ \big(x',\gamma(x')\big)+
\eps\eta\big(x',\gamma(x')\big)
\,\frac{\big( -\nabla \gamma(x'),1\big)}{\sqrt{ 1+|\nabla\gamma(x')|^2 }}
-\eps\kappa
,\; x'\in\R^{n-1}\right\}
\\ &=& 
\left\{ \left( x'-\eps \kappa'-
\frac{\eps\eta\big(x',\gamma(x')\big)\nabla \gamma(x')}{
\sqrt{ 1+|\nabla\gamma(x')|^2} },\;
\gamma(x')-\eps\kappa_n
+\,\frac{\eps\eta\big(x',\gamma(x')\big)}{\sqrt{ 1+|\nabla\gamma(x')|^2 }}
\right)
,\; x'\in\R^{n-1}\right\}.\end{eqnarray*}
So we define
\begin{equation}\label{DF:09:09}
y'=y'(x'):=
x'-\eps \kappa'-
\frac{\eps\eta\big(x',\gamma(x')\big)\nabla \gamma(x')}{
\sqrt{ 1+|\nabla\gamma(x')|^2} }.\end{equation}
Notice that, if~$\eps$ is sufficiently small
$$ \det \frac{\partial y'(x')}{\partial x'}\ne0.$$
Moreover, $|\nabla\gamma(x')|\le 1+|\nabla\gamma(x')|^2$ and therefore
$$ |y'(x)|\ge |x'| -\eps\kappa'-\eps \to +\infty \,{\mbox{ as }}\,
|x|\to+\infty.$$
Hence, by the Global Inverse Function Theorem (see e.g.
Corollary~4.3 in~\cite{MR0116352}),
we have that~$y'$ is a global diffeomorphism of class $C^2$ of~$\R^{n-1}$,
with inverse diffeomorphism~$x'=x'(y')$.
Thus, we obtain
$$ \Sigma_\eps^* =
\left\{ \left( y',\;
\gamma\big(x'(y')\big)-\eps\kappa_n
+\,\frac{\eps\eta\big(x'(y'),\gamma\big(x'(y')\big)\big)}{
\sqrt{ 1+\big|\nabla\gamma\big(x'(y')\big)\big|^2 }}
\right)
,\; y'\in\R^{n-1}\right\}.$$
This is clearly a graph, whose corresponding epigraph can be written as~$E_\eps^*=\{y_n<\gamma_\eps^*(y')\}$, with
$$ \gamma_\eps^*(y'):=
\gamma\big(x'(y')\big)-\eps\kappa_n
+\,\frac{\eps\eta\big(x'(y'),\gamma\big(x'(y')\big)\big)}{
\sqrt{ 1+\big|\nabla\gamma\big(x'(y')\big)\big|^2 }}
.$$
By~\eqref{DF:09:09}, we have that~$y'(x'_0)=x'_0$,
therefore~$\gamma_\eps^*({\bar{x}}')=\gamma({\bar{x}}')$
and so~${\bar{x}}\in \partial E^*_\eps$.
We also notice that
\begin{eqnarray*}
\gamma_\eps^*(y')&=& \gamma(y')+\nabla\gamma(y')\cdot\big(x'(y)-y'\big)
-\eps\kappa_n
+\,\frac{\eps\eta(y',\gamma(y'))}{
\sqrt{ 1+|\nabla\gamma(y')|^2 }} +\eps^2 R(y')\\
&=&
\gamma(y')+\nabla\gamma(y')\cdot\left(
\eps \kappa'+
\frac{\eps\eta(y',\gamma(y'))\nabla \gamma(y')}{
\sqrt{ 1+|\nabla\gamma(y')|^2} }\right)
-\eps\kappa_n
+\,\frac{\eps\eta(y',\gamma(y'))}{
\sqrt{ 1+|\nabla\gamma(y')|^2 }}
+\eps^2 R(y')\\
&=&
\gamma(y')+ \eps\sqrt{ 1+|\nabla\gamma(y')|^2 }\,
\left(
\eta(y',\gamma(y'))
-\kappa\cdot\frac{\big( 
\nabla\gamma(y'),-1\big)}{\sqrt{ 1+|\nabla\gamma(y')|^2 }}
\right)+\eps^2 R(y')
\end{eqnarray*}
for a suitable remainder
functions~$R$ (possibly varying from line
to line), that are bounded if so is~$|D^2\gamma|$.

Accordingly,
\begin{eqnarray*}
E^*_\eps \setminus E &=&
\big\{ \gamma(y')\le y_n<\gamma_\eps^*(y')\big\}\\
&=&
\left\{ \gamma(y')\le y_n<
\gamma(y')+ \eps\big(\Xi(y') +\eps^2 R(y')\big)^+
\right\},
\end{eqnarray*}
where \begin{eqnarray*}
&&\Xi(y'):=
\sqrt{ 1+|\nabla\gamma(y')|^2 }\,
\left(
\eta(y',\gamma(y'))
-\kappa\cdot\tilde\nu(y')
\right)\\
{\mbox{and }}&&\tilde\nu(y'):=
\frac{\big(
\nabla\gamma(y'),-1\big)}{\sqrt{ 1+|\nabla\gamma(y')|^2 }}.\end{eqnarray*}
Notice that~$\tilde\nu(y')=\nu\big(y',\gamma(y')\big)$.
Similarly,
$$E\setminus E^*_\eps\,\subseteq\,
\left\{ \gamma(y')-
\eps\big(\Xi(y') +\eps^2 R(y')\big)^-
\le y_n<\gamma(y')\right\}.$$
Therefore
\begin{eqnarray*}
\lim_{\eps\to0} 
\frac{1}{\eps}
\int_{E^*_\eps \setminus E}\frac{dy}{|{\bar{x}}-y|^{n+2s}}
&=& \lim_{\eps\to0} 
\frac{1}{\eps}
\int_{\R^{n-1}} \left[\int_{\gamma(y')}^{
\gamma(y')+\eps\,(\Xi(y')+\eps R(y'))^+}
\frac{dy_n}{|{\bar{x}}-y|^{n+2s}}
\right]\,dy'\\
&=& 
\int_{\R^{n-1}} 
\frac{ \Xi^+(y')
}{\big( |{\bar{x}}'-y'|^2+|{\bar{x}}_n-
\gamma(y')|^2 \big)^{\frac{n+2s}{2}}}
\,dy'
\end{eqnarray*}
and, similarly
$$ \lim_{\eps\to0}\frac{1}{\eps}
\int_{E\setminus E^*_\eps} \frac{dy}{|{\bar{x}}-y|^{n+2s}}\,=\,
\int_{\R^{n-1}} 
\frac{ \Xi^-(y')
}{\big( |{\bar{x}}'-y'|^2+|{\bar{x}}_n-
\gamma(y')|^2 \big)^{\frac{n+2s}{2}}}
\,dy'.$$
As a consequence,
\begin{eqnarray*}
&&\lim_{\eps\to0}\frac{1}{2\eps}
\big(H^s_E({\bar{x}}) - H^s_{E_\eps^*}({\bar{x}})\big)\\&=&
\lim_{\eps\to0}
\frac{1}{\eps}\left[
\int_{E^*_\eps \setminus E}\frac{dy}{|{\bar{x}}-y|^{n+2s}}
-\int_{E\setminus E^*_\eps} \frac{dy}{|{\bar{x}}-y|^{n+2s}} \right]
\\ &=&
\int_{\R^{n-1}}
\frac{ \Xi(y')
}{\big( |{\bar{x}}'-y'|^2+|{\bar{x}}_n-
\gamma(y')|^2 \big)^{\frac{n+2s}{2}}}
\,dy' \\
&=&
\int_{\R^{n-1}} \sqrt{ 1+|\nabla\gamma(y')|^2 }\,
\frac{\eta(y',\gamma(y'))
-\kappa\cdot\tilde\nu(y')
}{\big( |{\bar{x}}'-y'|^2+|
{\bar{x}}_n-
\gamma(y')|^2 \big)^{\frac{n+2s}{2}}}\,dy' \\
&=& \int_{\Sigma}
\frac{        
\eta(y)
-\kappa\cdot\nu(y)}{|{\bar{x}}-y|^{n+2s}}\,d{\mathcal{H}}^{n-1}(y),
\end{eqnarray*}
that is the desired result.
\end{proof}

An interesting consequence of Lemma \ref{78yuOOP} is that
graphs with vanishing nonlocal mean curvature cannot have horizontal
normals, as given by the following result:

\begin{theorem}\label{HONO}
Let~$E\subset\R^n$. Suppose that $\partial E$ is globally of class $C^2$
and that $H^s_E(x)=0$ for any $x\in \partial E$.

Let~$\nu=(\nu_1(x),\dots,\nu_n(x))$
be the exterior normal of~$E$ at $x\in\partial E$.

Then $\nu_n(x)\ne0$, for any $x\in\partial E$.
\end{theorem}

To prove Theorem \ref{HONO}, we
first compare deformations and translations of a graph.
Namely, we show that a normal deformation of size~$\eps \nu_n$
of a graph with normal~$\nu=(\nu_1,\dots,\nu_n)$ coincides
with a vertical translation of the graph itself, up to
order of~$\eps^2$. The precise result goes as follows:

\begin{lemma}\label{90:90}
Let~$\Sigma\subset\R^n$ be a graph of class $C^2$ globally, and let~$E$ be the corresponding
epigraph. Let~$\nu=(\nu_1,\dots,\nu_n)$
be the exterior normal of~$\Sigma=\partial E$.

Given~$\eps>0$, let
\begin{equation}\label{BA2}
\Sigma_\eps:=\{ x+\eps\nu_n(x)\,\nu(x),\; x\in\Sigma\}.\end{equation}
Then, if~$\eps$ is sufficiently small,~$\Sigma_\eps$ is a graph, for some epigraph~$E_\eps$, and
there exists a $C^2$-diffeomorphism~$\Psi$ of~$\R^n$
that is $C\eps^2$-close to the identity in~$C^2(\R^n)$, for some~$C>0$, such that
$$ \Psi(E_\eps) = E+\eps e_n.$$
\end{lemma}

\begin{proof} We denote by~$\gamma:\R^{n-1}\to\R$ the graph
that describes~$\Sigma$. In this way, we can write~$E=\{x_n <\gamma(x')\}$
and
$$ \nu(x) = \nu\big(x',\gamma(x')\big)=
\frac{\big( -\nabla \gamma(x'),1\big)}{\sqrt{ 1+|\nabla\gamma(x')|^2 }}.$$
Accordingly,
\begin{eqnarray*}
\Sigma_\eps&=&\left\{ \big(x',\gamma(x')\big)+\eps 
\frac{\big( -\nabla \gamma(x'),1\big)}{ 1+|\nabla\gamma(x')|^2 },
\;x'\in\R^{n-1}
\right\} 
\\&=&
\left\{ \left(x'-\eps
\frac{\nabla \gamma(x')}{ 1+|\nabla\gamma(x')|^2 },\,
\gamma(x')+\frac{\eps}{ 1+|\nabla\gamma(x')|^2 }
\right),\;x'\in\R^{n-1}
\right\}.
\end{eqnarray*}
To write~$\Sigma_\eps$ as a graph, we take as new coordinate
\begin{equation}\label{ouKAJJAIAJOJAJ1234JA} y'=y'(x'):=x'-\eps
\frac{\nabla \gamma(x')}{ 1+|\nabla\gamma(x')|^2 }.\end{equation}
Notice that, if~$\eps$ is sufficiently small
$$ \det \frac{\partial y'(x')}{\partial x'}\ne0.$$
Moreover, $|\nabla\gamma(x')|\le 1+|\nabla\gamma(x')|^2$ and therefore
$$ |y'(x)|\ge |x'| -\eps \to +\infty \,{\mbox{ as }}\,
|x'|\to+\infty.$$
As a consequence, by the Global Inverse Function Theorem (see e.g.
Corollary~4.3 in~\cite{MR0116352}),
we have that~$y'$ is a global diffeomorphism of class $C^2$ of~$\R^{n-1}$,
we write~$x'=x'(y')$ the inverse diffeomorphism and we have that
$$ \Sigma_\eps =
\left\{ \left(y',\,
\gamma \big(x'(y')\big)+
\frac{\eps}{ 1+\big|\nabla\gamma\big(x'(y')\big)\big|^2 }
\right),\;y'\in\R^{n-1}
\right\}.$$
So we can write the epigraph of~$\Sigma_\eps$ as
\begin{equation*}
E_\eps=
\left\{ y_n<
\gamma \big(x'(y')\big)+
\frac{\eps}{ 1+\big|\nabla\gamma\big(x'(y')\big)\big|^2 }
\right\}.\end{equation*}
Now we define
\begin{equation}\label{equaadsf13243} 
\Phi(y') := \gamma(y')-\gamma \big(x'(y')\big)
+\eps - \frac{\eps}{ 1+\big|\nabla\gamma\big(x'(y')\big)\big|^2 }\end{equation}
and~$z=\Psi(y)=\Psi(y',y_n):=y+\Phi(y')e_n$.
By construction, we have that
$$\Psi(E_\eps) = \big\{ z_n <\gamma(z')+\eps\big\}=E+\eps e_n.$$
To complete the proof of Lemma~\ref{90:90},
we need to show that
\begin{equation}\label{ouKAJJAIAJOJAJ1234JA:2} 
\|\Phi\|_{C^2(\R^n)}\le C\eps^2,
\end{equation}
for some~$C>0$. To this aim, we use \eqref{ouKAJJAIAJOJAJ1234JA} 
to see that
$$ x'=y'+\eps
\frac{\nabla \gamma(y')}{ 1+|\nabla\gamma(y')|^2 } + \phi_1(y'),$$
with $\|\phi_1\|_{C^2(\R^n)}\le C\eps^2$. Accordingly, by \eqref{equaadsf13243}, we have that
\begin{eqnarray*}
\Phi(y') &=& \gamma(y')-\gamma \left( y'+\eps
\frac{\nabla \gamma(y')}{ 1+|\nabla\gamma(y')|^2 } + \phi_1(y')\right)
+\eps - \frac{\eps}{ 1+\left|\nabla\gamma\left(
y'+\eps
\frac{\nabla \gamma(y')}{ 1+|\nabla\gamma(y')|^2 } + \phi_1(y')
\right)\right|^2}\\
&=&
\gamma(y')-\gamma(y')- \eps
\frac{|\nabla \gamma(y')|^2}{ 1+|\nabla\gamma(y')|^2 } 
+\eps - \frac{\eps}{1+|\nabla\gamma(y')|^2}
+ \phi_2(y')
\\&=& \phi_2(y'),
\end{eqnarray*}
with $\|\phi_2\|_{C^2(\R^n)}\le C\eps^2$.
This proves \eqref{ouKAJJAIAJOJAJ1234JA:2}, as desired.
\end{proof}

{F}rom Lemma~\ref{90:90} here and Theorem~1.1 in~\cite{MR3393254},
we obtain:

\begin{corollary}\label{COR:90:90}
In the setting of Lemma~\ref{90:90}, for any~$p\in\Sigma_\eps=\partial E_\eps$
we have that
$$ \big| H^s_{E_\eps} (p) - 
H^s_{E+\eps e_n} \big(\Psi(p)\big)
\big|\le C\eps^2,$$
for some~$C>0$.
\end{corollary}

Now we complete the proof of Theorem \ref{HONO}. To this aim, we observe that
\begin{equation}\label{joAKPAJJA}
\nu_n(x)\ge0 {\mbox{ for any }} x\in\partial E,
\end{equation}
since $E$ is a graph. Suppose that, by contradiction,
\begin{equation}\label{hiKNAJJA1kjmsx}
\nu_n(\bar x)=0 {\mbox{ for some }} \bar x\in\partial E.
\end{equation}
We use this and Lemma \ref{78yuOOP} with $\eta:=\nu_n$ and we find that
\begin{equation}\label{pl9uihtrsghA} \lim_{\eps\to0} \frac{1}{2\eps}
\big( H^s_E({\bar{x}}) -H^s_{E_\eps^*}({\bar{x}})\big)=
\int_{\Sigma}
\frac{
\nu_n(y)}{
|{\bar{x}}-y|^{n+2s}}\,d{\mathcal{H}}^{n-1}(y).\end{equation}
Also, comparing \eqref{BA1} (with $\eta:=\nu_n$) and \eqref{BA2}, 
and using again \eqref{hiKNAJJA1kjmsx}, we see that
$E^*_\eps= E_\eps$ and so Corollary \ref{COR:90:90} gives that
$$ H^s_{E_\eps^*} (\bar x) =
H^s_{E+\eps e_n} (\bar y)+O(\eps^2),$$
for some $\bar y\in\partial E+\eps e_n$.
Since $H^s_E$ vanishes, we can use the translation invariance
to see that also $H^s_{E+\eps e_n}$ vanishes. So we conclude that
$$ H^s_{E_\eps^*} (\bar x) = O(\eps^2).$$
These observations and \eqref{pl9uihtrsghA} imply that
$$ \int_{\Sigma}
\frac{
\nu_n(y)}{
|{\bar{x}}-y|^{n+2s}}\,d{\mathcal{H}}^{n-1}(y)=0.$$
Hence, in view of \eqref{joAKPAJJA}, we see that $\nu_n$ must vanish identically along $\Sigma$.
This says that $\Sigma$ is a vertical hyperplane, in contradiction with the graph assumption.
This ends the proof of Theorem \ref{HONO}.

\end{appendix}

\section*{Acknowledgements}
The first author has been supported by the Alexander 
von Humboldt Foundation. 
The second author has been supported by the ERC grant 
277749 \emph{E.P.S.I.L.O.N. Elliptic
Pde's and Symmetry of Interfaces and 
Layers for Odd Nonlinearities}.

{\footnotesize
\section*{Online lectures}

There are a few videotaped lectures online which collect some of the material
presented in this set of notes. The interest reader may look at

\begin{verbatim}
http://www.birs.ca/events/2014/5-day-workshops/14w5017/videos/watch/201405271048-Valdinoci.html

https://www.youtube.com/watch?v=2j2r1ykoyuE

https://www.youtube.com/watch?v=EDJ8uBpYpB4

https://www.youtube.com/watch?v=s_RRzgZ7VcM&list=PLj6jTBBj-5B_Vx5qA-HelhGUnGrCu7SdW&index=7

https://www.youtube.com/watch?v=okXncmRbCZc&index=14&list=PLj6jTBBj-5B_Vx5qA-HelhGUnGrCu7SdW

http://www.fields.utoronto.ca/video-archive/2016/06/2022-15336

http://www.mathtube.org/lecture/video/nonlocal-equations-various-perspectives-lecture-1

http://www.mathtube.org/lecture/video/nonlocal-equations-various-perspectives-lecture-2

http://www.mathtube.org/lecture/video/nonlocal-equations-various-perspectives-lecture-3

http://www.birs.ca/events/2016/5-day-workshops/16w5065/videos/watch/201609291100-Dipierro.html 
\end{verbatim}}

\end{document}